\documentclass[11pt,reqno]{amsart}
\usepackage{amssymb,amsmath,amscd,amsthm,amsfonts,wasysym,mathrsfs,amsxtra}
\usepackage{mathtools}
\usepackage{marginnote}
\usepackage{appendix}
\usepackage{hyperref}
\usepackage{tikz}
\usepackage[all,cmtip]{xy}
\usepackage{tikz-cd}
\newcommand{\s}{\mathbb{S}}

\newcommand{\CC}{\mathbb{C}}

\newcommand{\NN}{\mathbb{N}}
\newcommand{\PP}{\mathbb{P}}

\newcommand{\GG}{\mathbb{G}}
\newcommand{\TT}{\mathbb{T}}

\newcommand{\Aa}{\mathcal{A}}

\newcommand{\Cc}{\mathcal{C}}
\newcommand{\Ee}{\mathcal{E}}
\newcommand{\Gg}{\mathcal{G}}
\newcommand{\cg}{\mathfrak{g}}
\newcommand{\ch}{\mathfrak{h}}
\newcommand{\Dd}{\mathcal{D}}
\newcommand{\Ff}{\mathcal{F}}
\newcommand{\Hh}{\mathcal{H}}
\newcommand{\Kk}{\mathcal{K}}
\newcommand{\Nn}{\mathcal{N}}
\newcommand{\Tt}{\mathcal{T}}
\newcommand{\Pp}{\mathcal{P}}
\newcommand{\Oo}{\mathcal{O}}
\newcommand{\Uu}{\mathcal{U}}

\newcommand{\Qq}{\mathcal{Q}}
\newcommand{\Xx}{\mathcal{X}}

\newcommand{\R}{\mathcal{R}}

\newcommand{\V}{\mathcal{V}}

\newcommand{\GGG}{\mathscr{G}}

\newcommand{\ZZ}{\mathbb{Z}}

\newcommand{\LL}{\mathscr{L}}
\newcommand{\TTt}{\sf{T}}

\newcommand{\E}{\sf{E}}
\newcommand{\F}{\sf{F}}
\newcommand{\Hhh}{\sf{H}}

\newcommand{\spi}{\sf{\Psi}}
\newcommand{\SL}{\mathfrak{sl}}
\newcommand{\kk}{\underline{k}}
\newcommand{\Ll}{\underline{l}}
\newcommand{\bo}{\boldsymbol{1}}
\newcommand{\tdim}{\mathrm{dim}}
\newcommand{\tcodim}{\mathrm{codim}}
\newcommand{\Hom}{\mathrm{Hom}}
\newcommand{\End}{\mathrm{End}}
\newcommand{\Ext}{\mathrm{Ext}}
\newcommand{\Sym}{\mathrm{Sym}}
\newcommand{\tdet}{\mathrm{det}}
\newcommand{\Rm}{\mathrm}

\newcommand{\GLL}{\mathrm{GL}}
\newcommand{\SLL}{\mathrm{SL}}

\newtheorem{theorem}{Theorem}[section]

\newtheorem{lemma}[theorem]{Lemma}

\newtheorem{proposition}[theorem]{Proposition}

\newtheorem{corollary}[theorem]{Corollary}
\theoremstyle{definition}
\newtheorem{definition}[theorem]{Definition}

\newtheorem{example}[theorem]{Example}

\theoremstyle{remark}
\newtheorem{remark}[theorem]{Remark}
\newtheorem{conjecture}[theorem]{Conjecture}

\numberwithin{equation}{section}

\title[Shifted $\MakeLowercase{q}=0$ affine algebra]{A categorical action of the shifted $\MakeLowercase{q}=0$ affine algebra}

\begin{document}

\emergencystretch 3em

\address{National Center for Theoretical Sciences} \email{yhhsu@ncts.ntu.edu.tw}

\author[You-Hung Hsu]{You-Hung Hsu}


\makeatletter
\@namedef{subjclassname@2020}{%
	\textup{2020} Mathematics Subject Classification}
\makeatother


\maketitle

\begin{abstract}
We introduce a new algebra called the shifted $q = 0$ affine algebra, which arises naturally from the study of coherent sheaves on Grassmannians and $n$-step partial flag varieties via a natural correspondence.  It has similar presentation as the shifted quantum affine algebra defined by Finkelberg-Tsymbaliuk \cite{FT}. We then give a definition of its categorical action and prove that there is a categorical action on the bounded derived categories of coherent sheaves on $n$-step partial flag varieties. Finally, as an application, we
use it to construct a categorical action of the $q = 0$ affine Hecke algebra on the bounded derived category of coherent sheaves on the full flag variety.
\end{abstract}

\tableofcontents

\section{Introduction} \label{section1}

The (bounded) derived category of coherent sheaves on an algebraic variety plays a central role in modern algebraic geometry and related areas. The purpose of this article is to study categorical actions of a certain algebra (which will be called the shifted $q=0$ affine algebra) on the derived categories of coherent sheaves on Grassmannians and $n$-step partial flag varieties.

\subsection{Categorical $\SL_{2}$ action}
Roughly speaking, a categorical action of a Kac-Moody Lie algebra $\cg$ consists of a collection of functors and categories that recover actions of Chevalley generators at the level of Grothendieck groups. We consider the case $\cg=\SL_{2}(\CC)$. It has the standard basis 
\begin{equation*}
e=\begin{pmatrix}
0 & 1 \\
0 & 0 
\end{pmatrix}, \ 
f=\begin{pmatrix}
0 & 0 \\
1 & 0 
\end{pmatrix}, \ 
h=\begin{pmatrix}
1 & 0 \\
0 & -1 
\end{pmatrix}
\end{equation*} subject to the commutator relations
$[h,e]=2e, \ [h,f]=-2f, \ [e,f]=h$.

A finite-dimensional representation $V$ of $\SL_{2}(\CC)$ consists of a direct sum decomposition $V=\oplus_{\lambda}V_{\lambda}$ into weight spaces and linear maps $e:V_{\lambda}\rightarrow V_{\lambda+2}$ and $f:V_{\lambda}\rightarrow V_{\lambda-2}$, satisfying the relation $ef-fe|_{V_{\lambda}}=\lambda id_{V_{\lambda}}$. Such data can be depicted by the following diagram
\begin{equation} \label{sl2 pic}
\xymatrix{ 
	.... \ar@/^/[r]^{e} 
	& V_{\lambda-2}  \ar@/^/[l]^{f}  \ar@/^/[r]^{e}   
	& V_{\lambda}   \ar@/^/[r]^{e}   \ar@/^/[l]^{f}
	& V_{\lambda+2}   \ar@/^/[l]^{f}  \ar@/^/[r]^{e}  
	&....  \ar@/^/[l]^{f}  }
\end{equation} 

This characterization leads to the following notion of $\SL_{2}(\CC)$ acting on a category

\begin{definition} [\cite{Kam}] \label{def}
A naive categorical $\SL_{2}$ action consists of a sequence of additive categories $\Cc(\lambda)$ together with additive functors ${\E}:\Cc(\lambda) \rightarrow \Cc(\lambda+2)$ and ${\F}:\Cc(\lambda) \rightarrow \Cc(\lambda-2)$ for each $\lambda \in \ZZ$ such that there exist isomorphism of functors  \footnote{Note that we do not specify the data of these isomorphisms between functors in this definition.}  
	\begin{align}
	\begin{split} \label{absiso}
	&{\E\F}|_{\Cc(\lambda)} \cong {\F\E}|_{\Cc(\lambda)} \bigoplus Id_{\Cc(\lambda)}^{\oplus \lambda} \ \text{if} \ \lambda \geq 0, \\
	&{\F\E}|_{\Cc(\lambda)} \cong {\E\F}|_{\Cc(\lambda)} \bigoplus Id_{\Cc(\lambda)}^{\oplus -\lambda}  \ \text{if} \ \lambda \leq 0,
	\end{split}
	\end{align} where $Id_{\Cc(\lambda)}$ is the identity functor for $\Cc(\lambda)$.
\end{definition}

The diagram (\ref{sl2 pic}) for a categorical representation of $\SL_{2}$ becomes:
\begin{equation*}
\xymatrix{ 
	.... \ar@/^/[r]^{\E} 
	& \Cc(\lambda-2)  \ar@/^/[l]^{\F}  \ar@/^/[r]^{\E}   
	&\Cc(\lambda)   \ar@/^/[r]^{\E}  \ar@/^/[l]^{\F}
	& \Cc(\lambda+2)  \ar@/^/[l]^{\F}  \ar@/^/[r]^{\E}  
	&....  \ar@/^/[l]^{\F}  }
\end{equation*}

The work of Beilinson-Lusztig-MacPherson \cite{BLM} provides a geometric model of a categorical $\SL_2$ action, where the  weight categories are given by the derived categories $\Cc(\lambda)=\Dd^bCon(\GG(k,N))$ of constructible sheaves on Grassmannians with $\lambda=N-2k$. The functor ${\E}:\Dd^bCon(\GG(k,N))\rightarrow \Dd^bCon(\GG(k-1,N))$ is given by pull-push along the following correspondence
\begin{equation} \label{diagram9}  
\xymatrix{ 
	&&Fl(k-1,k)=\{0 \overset{k-1}{\subset} V' \overset{1}{\subset} V \overset{N-k}{\subset} \CC^N \} 
	\ar[ld]_{p_1} \ar[rd]^{p_2}   \\
	& \GG(k,N)  && \GG(k-1,N)
}
\end{equation} where $Fl(k-1,k)$ is the $3$-step partial flag variety and the numbers above the inclusions indicate the increase in dimensions. The functor ${\F}$ is given by the opposite pull-push.

We have the following theorem, due to  Beilinson-Lusztig-MacPherson and Chuang-Rouquier 
\begin{theorem}[\cite{BLM}, \cite{CR}] \label{theorem} 
	The categories $\Cc(\lambda)$ and functors $\E$, $\F$, defined above gives a naive categorical $\SL_{2}$ action. More precisely, the functors $\E$, $\F$ satisfy the relation (\ref{absiso}) in Definition \ref{def}.
\end{theorem}

The reason that we call it a naive categorical action is because we do not specify the natural isomorphisms in (\ref{absiso}). A categorical $\SL_{2}$ action is a naive one where these natural isomorphisms are specified via certain adjunctions; see Definition 1.5 in \cite{Kam} for details.

Understanding natural transformations between functors for relations in a categorical action is an important problem in higher representation theory, since ideally the isomorphisms between functors (like (\ref{absiso})) should be induced from certain natural adjunction data. In the past decade, there has been much progress on constructing categorical actions of semi-simple (or more generally Kac-Moody) Lie algebras $\cg$ and their $q$-analogues.

One answer to this problem for the $\SL_{2}$-categorification was given by Chuang-Rouquier \cite{CR} and was later generalized to (simply-laced) Kac-Moody algebras $\cg$ by Khovanov-Lauda \cite{KL1} \cite{KL2} \cite{KL3}, and Rouquier \cite{R}. After this, there were several developments and applications following Theorem \ref{theorem}. People have extensively studied the categorical $\SL_{2}$ actions in several flavours, one of them being the notion of geometric categorical $\SL_2$ action, introduced by Cautis-Kamnitzer-Licata in \cite{CKL2} via the language of Fourier-Mukai transformations; see also \cite{CKL1}, \cite{CKL3} for applications.

\subsection{Constructing categorical action via derived categories of coherent sheaves} 

\subsubsection{Main results}

Motivated by the above works, we consider categories of coherent sheaves instead of constructible sheaves. This means that our weight categories are $\Kk(\lambda)=\Dd^bCoh(\GG(k,N))$, the bounded derived categories of coherent sheaves on $\GG(k,N)$.

In the coherent setting, there is a natural line bundle on $Fl(k-1,k)$; namely, denoting $\V,\V'$ to be the tautological bundles on $Fl(k-1,k)$ of rank $k$ and $k-1$, respectively, then we have the tautological line bundle $\V/\V'$ on $Fl(k-1,k)$.  

Thus in contrast to the constructible setting where the functors are given by a pull-push along the correspondence (\ref{diagram9}), in the coherent setting we have a family of functors parameterized by powers of the line bundles $(\V/\V')$. More precisely, we have functors for $r \in \ZZ$
\begin{equation} \label{eq 8} 
{\E}_{r}:=p_{2*}(p_{1}^{*}\otimes(\V/\V')^{r}):\Dd^{b} Coh(\GG(k,N)) \rightarrow \Dd^{b} Coh(\GG(k-1,N))
\end{equation} and similarly for ${\F}_{r}$ in the opposite direction. The main goal of this article is to study the following problem.

\textbf{Problem:} Understand the algebra with (abstract) generators $e_{r}, \ f_{s}$  acting on $\bigoplus_{k} \Dd^{b} Coh(\GG(k,N))$ induced by the functors ${\E}_{r}$, ${\F}_{s}$ defined in (\ref{eq 8}), where $r, \ s \in \ZZ$. In particular, there are several natural questions that we can ask: 
\begin{enumerate}
     \item What are the categorical commutator relations between ${\E}_{r}$ and ${\F}_{s}$? \\
    \item What are the relations that we obtain in the algebra after decategorifying (passing to the Grothendieck group)? \\
    \item  Assuming that we obtain the algebra in (2), can we give a definition of its categorical action as for $\SL_2$ in Definition \ref{def}?
\end{enumerate}
   
The main results of this article answer the above questions. Before we state the main result, let us mention some remarks.
   
First, this algebra highly resembles to the loop algebra $L\SL_{2}:=\SL_{2} \otimes \CC[t,t^{-1}]$ (however, they are not the same) where the generators $e_{r}$, $f_{s}$ are analogous to $e\otimes t^r$, $f\otimes t^{s}$, $r,s \in \ZZ$.

Second, such a construction of decategorified actions goes back to the work by Nakajima \cite{N2}, where twists by line bundles for the loop structure appears when moving from cohomology (or Borel-Moore homology) to K-theory.

Third, to the second question, after decategorifying, we obtain an algebra with a presentation that is similar to the shifted quantum affine algebra defined by Finkelberg-Tsymbaliuk \cite{FT}. We call the resulting algebra \textit{the shifted $q=0$ affine algebra}, denoted by $\dot{\Uu}_{0,N}(L\SL_2)$. The $q=0$ refers to the fact that some of our relations can be obtained by taking $q=0$ directly from relations in the shifted quantum affine algebra, and its relation to the $q=0$ affine Hecke algebra (discussed in Section \ref{section 6}).

Finally, answering these questions also leads us to construct a categorical action of $\dot{\Uu}_{0,N}(L\SL_2)$ on $\bigoplus_{k}\Dd^bCoh(\GG(k,N))$. More generally, we consider the $\SL_{n}$ case, where the Grassmannian is replaced by the $n$-step partial flag varieties $Fl_{\kk}(\CC^N)$ (see (\ref{nfl}) for its definition). We summarize the main results of this article in the following theorem.

\begin{theorem} \label{Theorem'''}
\begin{enumerate}
    \item There are non-split exact triangles relating the functors ${\E}_{i,r}{\F}_{i,s}$ and ${\F}_{i,s}{\E}_{i,r}$ for certain $r,s$ (Proposition \ref{proposition 5}).
    \item The resulting algebra is a new algebra, the shifted $q=0$ affine algebra $\dot{\Uu}_{0,N}(L\SL_n)$ with generators and relations given in Definition \ref{definition 1}.
    \item We give a definition of categorical $\dot{\Uu}_{0,N}(L\SL_n)$ actions (Definition \ref{definition 2}). We prove that there is a categorical action of $\dot{\Uu}_{0,N}(L\SL_n)$ on $\bigoplus_{\kk}\Dd^bCoh(Fl_{\kk}(\CC^N))$ (Theorem \ref{theorem 1}).
\end{enumerate}
\end{theorem}

\subsubsection{The difference from the constructible setting and other new features}

In this subsection, we address details about our results, in particular the categorical commutator relations from Theorem \ref{Theorem'''} (1) arising from the geometric setting. We emphasise the contrast with the constructible setting (i.e. relation (\ref{absiso}) in Definition \ref{def}). To compare the two compositions of functors ${\E}_{r}{\F}_{s}$ and ${\F}_{s}{\E}_{r}$ with ${\E}_{r}$, ${\F}_{s}$ defined in (\ref{eq 8}), we use the language of FM (Fourier-Mukai) transformations to translate the comparison between compositions of functors to the comparison between convolutions of FM kernels.

We denote $\Ee_{r}\bo_{(k,N-k)}$ to be the FM kernel for ${\E}_{r}\bo_{(k,N-k)}:\Dd^{b} Coh(\GG(k,N)) \rightarrow \Dd^{b} Coh(\GG(k-1,N))$ and similarly $\Ff_{s}\bo_{(k,N-k)}$ for ${\F}_{s}\bo_{(k,N-k)}$ where $r, \ s \in \ZZ$. We then obtain the following isomorphisms
\begin{equation} \label{eq x}
(\Ee_{r} \ast \Ff_{s})\bo_{(k,N-k)} \cong (\Ff_{s}\ast \Ee_{r})\bo_{(k,N-k)} \ \text{for} \ 1-k \leq r+s \leq N-k-1
\end{equation} and the following exact triangles in $\Dd^bCoh(\GG(k,N) \times \GG(k,N))$
\begin{align} 
  &({\Ff}_{s} \ast {\Ee}_{r})\bo_{(k,N-k)} \rightarrow ({\Ee}_{r}\ast {\Ff}_{s})\bo_{(k,N-k)} \rightarrow {\Psi}^{+}\bo_{(k,N-k)},  \ \text{if} \  r+s=N-k  \label{eq y} \\
  &({\Ff}_{s} \ast {\Ee}_{r})\bo_{(k,N-k)} \rightarrow ({\Ee}_{r}\ast {\Ff}_{s})\bo_{(k,N-k)} \rightarrow ({\Psi}^{+} \ast \Hh_{1})\bo_{(k,N-k)},  \ \text{if} \  r+s=N-k+1 \label{eq z}
\end{align} where we denote $\ast$ to be the convolution of FM kernels, and  ${\Psi}^{+}\bo_{(k,N-k)}$, $\Hh_{1}\bo_{(k,N-k)}$ are certain FM kernels (see Definition \ref{definition 3} for details).

We mention some key properties of the above results. First, note that the commutator between ${\E}_{r}$ and ${\F}_{s}$ depends only on the integer $r+s$. Second, (\ref{eq x}),  (\ref{eq y}), and (\ref{eq z}) are reflected from the coherent sheaf cohomology $H^{*}(\PP^{N-1}, \Oo_{\PP^{N-1}}(-r-s-k))$. For example, (\ref{eq x}) corresponds to the vanishing  $H^{*}(\PP^{N-1}, \Oo_{\PP^{N-1}}(-r-s-k))=0$ for $-N+1 \leq -r-s-k \leq -1$ and implies that $[e_{r},f_{s}]1_{(k,N-k)}=0$ for $1-k \leq r+s \leq N-k-1$. Third, the exact triangles (\ref{eq y}), (\ref{eq z}) are non-split (see Remark \ref{nonsplit}), which is another very different feature from the corresponding constructible setting, i.e. Theorem \ref{theorem}.

We also obtain similar exact triangles when $r+s=-k, \ -k-1$ (see relation (16) in Definition \ref{definition 2} for details). For the case where $r+s \geq N-k+2$ and $r+s \leq -k-2$, although we have a description of the FM kernels that relate $({\Ff}_{s} \ast {\Ee}_{r})\bo_{(k,N-k)}$ and $({\Ee}_{r}\ast {\Ff}_{s})\bo_{(k,N-k)}$ as in (\ref{eq y}), (\ref{eq z}), we do not include those relations in either the definition of $\dot{\Uu}_{0,N}(L\SL_2)$ nor its categorical action (for a short discussion, see Subsection \ref{subsection 5.4}).

The main reason is that the presentation of the shifted $q=0$ affine algebra (see Definition \ref{definition 1}) that we choose to work with is motivated by the so-called Levendorskii presentation of the shifted quantum affine algebra (see Definition \ref{definition 6}) introduced by Finkelberg-Tsymbaliuk \cite{FT}. The advantage of using this presentation is its simplicity, i.e. it has a
finite number of generators and relations, and thus can be used to define a categorical action. 

In \cite{FT}, Finkelberg-Tsymbaliuk showed that the Levendorskii presentation is equivalent to the usual loop presentation of the shifted quantum affine algebra (see Theorem \ref{Theorem''}). Similarly, in Appendix \ref{appendix A}, we purpose a definition for the shifted $q=0$ affine algebra via the loop presentation (see Definition \ref{Definition 1}), and we also expect (see Conjecture \ref{conjecture 1}) that the two presentations are equivalent.

Finally, besides the commutator relations between ${\E}_{r}\bo_{(k,N-k)}$ and ${\F}_{s}\bo_{(k,N-k)}$, we also find other interesting hidden higher relations at the level of derived categories.

We define the functors ${\Hhh}_{1}\bo_{(k,N-k)}$ to be the Fourier-Mukai transform with kernel given by $\Hh_{1}\bo_{(k,N-k)}$. Then we study the categorical commutator relation between ${\Hhh}_{1}\bo_{(k,N-k)}$ and ${\E}_{r}\bo_{(k,N-k)}$. They are given by the following (non-split) exact triangles in $\Dd^bCoh(\GG(k,N) \times \GG(k-1,N))$
\begin{equation} \label{eq5}
({\Hh}_{1} \ast {\Ee}_{r})\bo_{(k,N-k)} \rightarrow ({\Ee}_{r}\ast {\Hh}_{1})\bo_{(k,N-k)} \rightarrow ({\Ee}_{r+1}\bigoplus {\Ee}_{r+1}[1])\bo_{(k,N-k)}.
\end{equation} 

The exact triangle (\ref{eq5}) implies that after  decategorifying (at the level of Grothendieck group), the commutator relations between the generators $e_{r}1_{(k,N-k)}$ and $h_{1}1_{(k,N-k)}$ in $\dot{\Uu}_{0,N}(L\SL_2)$ are trivial, i.e. $[h_{1},e_{r}]1_{(k,N-k)}=0$. But they are \textit{non}-trivial in the categorification. We also study the categorical relations between $\Hh_{1}\bo_{(k,N-k)}$ and $\Ff_{s}\bo_{(k,N-k)}$ and obtain similar results. We refer the reader to Theorem \ref{theorem 3} for details.

\subsubsection{The Grothendieck groups}
Even though the presentation of $\dot{\Uu}_{0,N}(L\SL_2)$  (or more generally $\dot{\Uu}_{0,N}(L\SL_n)$) that we use is not the loop presentation, the categorical action we construct on $\bigoplus_{k}\Dd^bCoh(\GG(k,N))$ from (3) in Theorem \ref{Theorem'''} can help us understand the commutator of the loop generators $e_{r}1_{(k,N-k)}$, $f_{s}1_{(k,N-k)}$ on the Grothendieck group $\bigoplus_{k} K(\GG(k,N))$ for all $r,s \in \ZZ$. The loop generators $e_{r}1_{(k,N-k)}$ act on $K(\GG(k,N))$ via decategorifying the functor ${\E}_{r}\bo_{(k,N-k)}$ in (\ref{eq 8}), i.e. 
\begin{equation*}
e_{r}1_{(k,N-k)} \coloneqq p_{2*}(p_{1}^{*} \cdot [(\V/\V')^{r}]):K(\GG(k,N)) \rightarrow K(\GG(k-1,N))
\end{equation*} where we denote $[(\V/\V')^{r}]$ to be the class of the line bundle $(\V/\V')^{r}$ in the Grothendieck group $K(\GG(k,N))$. There is a similar definition of $f_{s}1_{(k,N-k)}$ in the opposite direction. Then we have the following corollary. 

\begin{corollary} (Corollary \ref{corollary cr1})
The commutator relations in the Grothendieck group $K(\GG(k,N))$ for $e_{r}1_{(k,N-k)}$, $f_{s}1_{(k,N-k)}$ with $r,s \in \ZZ$ are given by
\begin{equation*}
[e_{r},f_{s}]1_{(k,N-k)}= \begin{cases}
	\otimes (-1)^{N-k-1}[\det(\CC^N/\V)][\Sym^{r+s-N+k}(\CC^N)] & \text{if} \  r+s \geq N-k \\
	0 & \text{if} \  -k+1 \leq r+s \leq N-k-1 \\
	\otimes (-1)^{k}[\det(\V)^{-1}][\Sym^{-r-s-k}(\CC^N)^{\vee}]& \text{if} \ r+s \leq -k
	\end{cases}.
\end{equation*}
\end{corollary}

\begin{remark}
We also have the above result for the $\SL_{n}$ case, see Corollary \ref{corollary cr2}.
\end{remark}

\subsection{Application to the $q=0$ affine Hecke algebra} \label{1.3}

Affine Hecke algebras and their degenerations play an important role in representation theory and related areas like number theory, knot homology, etc. They have been studied extensively over the past few decades. One of the main problems is to classify the finite dimensional irreducible representations of affine Hecke algebras. 

Starting with Lusztig \cite{Lu1}, people have approached the problem using equivariant K-theory. In \textit{loc. cit.}, Lusztig constructed an action of the affine Hecke algebra on the equivariant K-theory of the full flag variety. The generators of the finite Hecke algebra part act by the so-called \textit{Demazure-Lusztig} operators, while the generators of the translation part act by tensoring of line bundles. Later, Kazhdan-Lusztig \cite{KL0} \cite{KL} generalized this to construct representations of the affine Hecke algebra on the equivariant K-theory of Springer fibres and proved the Deligne-Langlands conjecture.

The representation theory of affine Hecke algebra and its related version is closely connected to the categorical representation of Kac-Moody algebras. Thus people also study the categorification or construct categorical actions of the (affine) Hecke algebras. The famous examples include the categorification of Hecke algebras via \textit{Soergel bimodules} (see \cite{EMTW} for an introduction), and Bezrukavnikov's two realizations of an affine Hecke algebra  \cite{Bez}.

Our second main result of this article is to categorify the above classical result by Lusztig \cite{Lu1} 
by lifting the action from K-theory to the derived category of coherent sheaves. More precisely, we construct a categorical action of the $q=0$ affine Hecke algebra, i.e. the affine Hecke algebra specialized at $q=0$, on the derived category of coherent sheaves on the full flag variety.

Let us explain this more detaily. We fix $G=\SLL_{N}(\CC)$ with $B \subset G$ be the Borel of upper triangular matrices. Then the full flag variety has the following description
\begin{equation} \label{eq'}
	G/B=\{0=V_{0} \subset V_{1} \subset ... \subset V_{N}=\CC^N \ | \ \dim V_{k}=k \ \text{for} \ \text{all} \ k  \},
\end{equation} and similarly the partial flag variety
\begin{equation} \label{eq''}
	G/P_{i}=\{0 \subset V_{1} \subset V_{2} \subset ...V_{i-1} \subset V_{i+1} \subset ... V_{N}=\CC^N \ | \ \dim V_{k} =k \ \text{for} \ k \neq i \}
\end{equation} where $P_{i}$ is a minimal parabolic subgroup for $1 \leq i \leq N-1$. 

For each $1 \leq i \leq N-1$, we have a natural projection $\pi_{i}:G/B \rightarrow G/P_{i}$ which is a $\PP^1$-fibration for all $1 \leq i \leq N-1$. Such $\pi_{i}$ induces natural pullback $\pi^{*}_{i}$ and pushforward $\pi_{i*}$ in K-theory for all $1 \leq i \leq N-1$. 

The \textit{Demazure-Lusztig} operators are certain $q$-analog of the Demazure operators, defined by $\delta_{i}\coloneqq \pi^{*}_{i}\pi_{i*}$, where definition goes back to the fundamental works by Bernstein-Gelfand-Gelfand \cite{BGG} and Demazure \cite{D} on the divided difference operators. 

On the other hand, let $\V_{i}$ be the tautological bundle of rank $i$ on $G/B$ for $0 \leq i \leq N$. Then we have the natural line bundles $\LL_{i}=\V_{i}/\V_{i-1}$ on $G/B$ for $1 \leq i \leq N$. 

The Demazure operators $\delta_{i}$, together with the operators given by tensoring with the line bundles $\LL_{j}$, generate the $q=0$ affine Hecke algebra, denoted by $\Hh_{N}(0)$, and thus give an action of $\Hh_{N}(0)$ on $K(G/B)$. We refer the readers to Definition \ref{definition 5} for the definition of $\Hh_{N}(0)$. The following is our second main result.

\begin{theorem} [Theorem \ref{theorem 4}] \label{Thm}
There is a categorical action of $\Hh_{N}(0)$ on $\Dd^bCoh(G/B)$.
\end{theorem}

One way to prove this theorem is to directly lift the action by replacing the generators with functors given by FM kernels. However, to check the relations for the action still involves eight convolutions of kernels and verify various exact triangles relating them (see Theorem \ref{theorem 4} for details). Instead of proving this directly, we interpret the Demazure operators in terms of elements in the shifted $q=0$ affine algebra and use its categorical action to give a short proof. 

We need to introduce more notation. For each $\kk=(k_{1},...,k_{n}) \vDash N$, the $n$-step partial flag variety is defined by 
\begin{equation} \label{nfl} 
	Fl_{\kk}(\CC^N):=\{V_{\bullet}=(0=V_{0} \subset V_1 \subset ... \subset V_{n}=\CC^N) \ | \ \tdim V_{i}/V_{i-1}=k_{i} \ \text{for} \ \text{all} \ i\}.	
\end{equation} With this notation, the full flag variety $G/B$ and partial flag varieties $G/P_{i}$ in (\ref{eq'}) and (\ref{eq''}) have the following description
\begin{equation*}
    G/B=Fl_{(1,1,...,1)}(\CC^N), \ G/P_{i}=Fl_{(1,1,...,1)+\alpha_{i}}(\CC^N)= Fl_{(1,1,...,1)-\alpha_{i}}(\CC^N)
\end{equation*} where $\alpha_{i}=(0,...,-1,1,...,0)$ is the simple root with $-1$ in the $i$th position and we have the following commutator diagram 
\begin{equation} \label{diagram''}
\xymatrix{ 
    & \delta_{i} \ar@(dl,dr) \\
    K(G/P_{i}=Fl_{{(1,1,...,1)-\alpha_{i}}}(\CC^N))   \ar@/^/[r]^{e_{i,r}}   
	& K(G/B=Fl_{{(1,1,...,1)}}(\CC^N))   \ar@/^/[r]^{e_{i,r}}   \ar@/^/[l]^{f_{i,s}}
	& K(G/P_{i}=Fl_{{(1,1,...,1)+\alpha_{i}}}(\CC^N))   \ar@/^/[l]^{f_{i,s}}
	}
\end{equation}

Using (\ref{diagram''}), the Demazure operators $\delta_{i}$ can be written as elements in $\dot{\Uu}_{0,N}(L\SL_N)$. As a consequence, all the categorical relations we have to check follow from relations in the categorical action of  $\dot{\Uu}_{0,N}(L\SL_N)$, which drastically reduces  calculations. 

Another interesting consequence of this result is we can generalize the interpretation of Demazure operators in diagram (\ref{diagram''}) to the $n$-step partial flag varieties $Fl_{\kk}(\CC^N)$, and construct more general Demazure (or idempotent) operators acting on its K-theory/derived category. See subsection \ref{1.4} for details.

Finally, we mention the related works by Arkhipov and Kanstrup. In the series of papers \cite{AK1}, \cite{AK2}, \cite{AK3}, and \cite{AK4}, Arkhipov and Kanstrup defined the notion of \textit{Demazure descent data} on a triangulated category originally as an attempt to understand the higher categorical Beilinson-Bernstein localization which developed by Ben-Zvi and Nalder \cite{BN}. The categorified Demazure operators in our action from Theorem \ref{Thm} provides a Demazure descent data on the triangulated category $\Dd^bCoh(G/B)$.

\subsection{Other applications} \label{1.4}
In this section, we mention some results and other applications of the categorical actions of the shifted $q=0$ affine algebra that we prove in \cite{Hsu1}, \cite{Hsu2}.

\subsubsection{Semiorthogonal decompositions}
In \cite{Hsu1}, we prove that a categorical action of $\dot{\Uu}_{0,N}(L\SL_{2})$ naturally gives rise to a semiorthogonal decomposition for each weight category.

We briefly explain what a categorical action of $\dot{\Uu}_{0,N}(L\SL_{2})$ is. A categorical $\dot{\Uu}_{0,N}(L\SL_2)$ action consists of a 2-category $\Kk$, which is triangulated, $\CC$-linear and idempotent-complete. The objects of $\Kk$ are denoted by $\Kk(k,N-k)$, where $0 \leq k \leq N$ and they are also triangulated categories. The generators $e_{r}1_{(k,N-k)}$, $f_{s}1_{(k,N-k)}$ act by functors ${\E}_{r}\bo_{(k,N-k)}$, ${\F}_{s}\bo_{(k,N-k)}$ that satisy certain categorical relations. (see Definition \ref{definition 2}) 

The main idea comes from the observation that the Kapranov exceptional collection \cite{Kap85} \cite{Kap88} can be obtained under the categorical action of the shifted $q=0$ affine algebra, and an exceptional collection naturally gives rise to a semiorthogonal decomposition. For $\GG(k,N)$, the Kapranov exceptional collection is given by $\{ \s_{\lambda}\V \}$ where $\s$ is the Schur functor and $\lambda=(\lambda_{1},...,\lambda_{k}) \in P(N-k,k)$ which is the set of Young diagram with at most $k$ rows and at most $N-k$ columns. Then those exceptional objects can be written as convolutions of kernels 
\begin{equation*}
    \s_{\lambda}\V \cong \Ff_{\lambda_{1}} \ast ... \ast \Ff_{\lambda_{k}}\bo_{(0,N)}.
\end{equation*}

\begin{theorem} (Theorem 4.3 and 4.7 in \cite{Hsu1}) \label{Theorem SOD}
Given a categorical action of $\dot{\Uu}_{0,N}(L\SL_{2})$ on a triangulated 2-category $\Kk$. Then each functor ${\F}_{\lambda}\bo_{(0,N)}\coloneqq{\F}_{\lambda_{1}}....{\F}_{\lambda_{k}}\bo_{(0,N)}:\Kk(0,N) \rightarrow \Kk(k,N-k)$ is fully faithful for $\lambda:=(\lambda_{1},...,\lambda_{k})\in P(N-k,k)$. Moreover, if we denote $\textnormal{Im}{\F}_{\lambda}\bo_{(0,N)}$ to be the full subcategory of $\Kk(k,N-k)$ generated by the essential images of ${\F}_{\lambda}\bo_{(0,N)}$, then those subcategories give rise to the following semiorthogonal decomposition
\begin{equation} \label{SOD}
	\Kk(k,N-k)=\langle \Aa(k,N-k),\ \textnormal{Im}{\F}_{(N-k,...,N-k)}\bo_{(0,N)}, ..., \ \textnormal{Im}{\F}_{(0,...,0)}\bo_{(0,N)} \rangle
\end{equation} where $\Aa(k,N-k)$ is the right orthogonal complement.
\end{theorem}

\subsubsection{Demazure operators for $Fl_{\kk}(\CC^N)$} In \cite{Hsu2}, we construct two families of idempotent operators acting on $K(Fl_{\kk}(\CC^N))$. Those operators can be viewed as a generalization of the Demazure operators $\delta_{i}$ that introduced in Subsection \ref{1.3} for the full flag variety $G/B$. The main idea comes from an interpretation of Demazure operators in terms of the elements in the shifted $q=0$ affine algebra used in this article to prove Theorem \ref{theorem 4}.

We generalize the interpretation of $\delta_{i}$ in diagram (\ref{diagram''}) to $Fl_{\kk}(\CC^N)$, and obtain certain desired operators $\delta'_{i}$, $\delta''_{i}$ for $K(Fl_{\kk}(\CC^N))$. Moreover, they are idempotent operators, i.e. $(\delta'_{i})^{2}=\delta'_{i}$, $(\delta''_{i})^{2}=\delta''_{i}$.

We denote  $H_{n}'(0)$ and $H_{n}''(0)$ to be the algebra generated by $\{\delta'_{i}\}$ and $\{\delta''_{i}\}$, respectively (see Definitions 3.7, 3.8 in \cite{Hsu2} for details). These two algebras are variants of the $q=0$ Hecke algebra where the generators $\{\delta'_{i}\}$ and $\{\delta''_{i}\}$ do not satisfy the braid relations.

Since the definition of categorical action of $\dot{\Uu}_{0,N}(L\SL_n)$ is abstractly defined, our result can be stated abstractly as follows.

\begin{theorem} (Theorem 5.3 in \cite{Hsu2}) 
Given a categorical action of $\dot{\Uu}_{0,N}(L\SL_n)$ on $\Kk$, there exist categorical actions of two variants of the $q=0$ Hecke algebra, denoted by $H_{n}'(0)$ and $H_{n}''(0)$, on $\Kk(\kk)$ where the generators $\delta'_{i}$ and $\delta''_{i}$ act by the functors ${\E}_{i,0}{\F}_{i,k_{i+1}} ({\spi}^{+}_{i})^{-1} \bo_{\kk}$ and ${\F}_{i,0}{\E}_{i,-k_{i}} ({\spi}^{-}_{i})^{-1}\bo_{\kk}$, respectively.
\end{theorem}

Since we prove that there is a categorical action of $\dot{\Uu}_{0,N}(L\SL_n)$ on $\bigoplus_{\kk}\Dd^bCoh(Fl_{\kk}(\CC^N))$, as a consequence we have the following corollary.

\begin{corollary}(Corollary 5.4 and 5.5 in \cite{Hsu2})
There exist categorical actions of $H_{n}'(0)$ and $H_{n}''(0)$ on $\Dd^bCoh(Fl_{\kk}(\CC^N))$, which induce actions on $K(Fl_{\kk}(\CC^N))$.
\end{corollary}

\subsection{Some related works and further directions}
We address the relations to other works and point out some possible interesting further directions that would like to study in the future. 

\subsubsection{Examples beyond the finite Grassmannians}
Since the main examples in this article are the usual Grassmannians, it would be interesting to construct categorical action of the shifted $q=0$ affine algebra on other examples. One possible example is a generalization to Grassmannians (more generally, Quot schemes) of coherent sheaves of homological dimension $\leq 1$. More precisely, let $X$ be a smooth projective variety and $\GGG$ be a coherent sheaf on $X$ of homological dimension $\leq 1$, i.e. such that there is a locally free resolution $\Ee^{-1} \rightarrow \Ee^{0} \rightarrow \GGG$. Then we consider the Grassmannians $\GG(\GGG, d)$ of rank $d$ locally free quotients of $\GGG$ and its derived category $\Dd^bCoh(\GG(\GGG,d))$.

In recent works, Jiang-Leung \cite{JL} obtain a semiorthogonal decomposition of $\Dd^bCoh(\GG(\GGG,d))$ for $d=1$ which generalizes the projective bundle formula by Orlov \cite{O}. Later Jiang \cite{J} purposed a conjecture about the semiorthogonal decomposition for general $d$ which was recently proved by Toda \cite{Y}.

The main reason we mention these results is because they provide examples with \textit{non}-trivial orthogonal complements which are potential candidates for $\Aa(k,N-k)$ in formula (\ref{SOD}) from Theorem \ref{Theorem SOD} (in contrast with the Kapranov exceptional collection, which is full for the usual Grassmannians). We wish to construct a categorical action on those examples so that we can obtain a representation-theoretic meaning of the subcategory $\bigoplus_{k}\Aa(k,N-k)$.

\subsubsection{Shifted quantum affine algebra}
In \cite{N1}, Nakajima constructed actions of quantum loop algebras on the equivariant K-theory of quiver varieties. Typical examples of Nakajima quiver varieties are the cotangent bundle of (type $A$) partial flag varieties $T^*Fl_{\kk}(\CC^N)$, and we have an action of $\Uu_{q}(L\SL_{n})$ on $\bigoplus_{\kk} K^{\CC^*}(T^*Fl_{\kk}(\CC^N))$. Since there are isomorphsims $K(T^*Fl_{\kk}(\CC^N)) \xrightarrow{\simeq} K(Fl_{\kk}(\CC^N))$, it would be interesting to see the relationship between the two actions.

On the other hand, Finkelberg-Tsymbaliuk \cite{FT} construct actions of the (truncated) shifted quantum affine algebras on the (localized) equivariant K-theory of affine Grassmannians and parabolic Laumon spaces. The variable $q$ in the shifted quantum affine algebras comes from a $\CC^*$-action. It would be interesting to interpret our algebra as a certain $q \rightarrow 0$ limit for the shifted quantum affine algebra from the $\CC^*$ action on these spaces.

\subsubsection{Other categorical relations and 2-representation}
As mentioned before, the presentation we use to define the categorical action has finite number of generators and relations. However, the loop realization is much more canonical and it is still natural to study other categorical relations that we do not address in this article, in particular, relations between ${\E}_{r}{\F}_{s}\bo_{(k,N-k)}$ and ${\F}_{s}{\E}_{r}\bo_{(k,N-k)}$ when $r+s \geq N-k+2$ and $r+s \leq -k-2$.

The final thing we mention is about 2-representations. The categorification of quantum groups have been studied extensively, e.g., \cite{CL}, \cite{KL1}, \cite{KL2}, \cite{KL3}, and \cite{R}. Most of those results lead to the construction of the so-called KLR (or quiver Hecke) algebras that act as natural transformations on the generating 1-morphisms ${\E}_{i}, \ {\F}_{j}$ and their compositions. Thus it is natural to study the higher relations, i.e.; the natural transformations between the functors in our categorical action. For instance, the exact triangles (\ref{eq y}), (\ref{eq z}) should be induced from certain natural transformations.

\subsection{Organization}

In Section \ref{section2}, we define the shifted $q=0$ affine algebras. We also mention the definition of shifted quantum affine algebra that defined by Finkelberg-Tsymbaliuk \cite{FT}.

In Section \ref{section3}, we give a definition for the categorical action of the shifted $q=0$ affine algebras.

In Section \ref{section 4}, we recall some background of the Fourier-Mukai transformations, which would be used in the next few sections in order to prove the categorical action.

In Section  \ref{section 5}, we prove the main theorem of this article, that is, there is a categorical action of shifted $q=0$ affine algebra on the bounded derived categories of coherent sheaves of Grassmannians and n-step partial flag varieties (Theorem \ref{theorem 1}). Finally, we calculate the commutators of the loop generators at the level of Grothendieck group.

In Section \ref{section 6}, we show that there is a categorical action of the $q=0$ affine Hecke algebras on the bounded derived category of coherent sheaves on the full flag variety by interpreting the Demazure operators in terms of the elements in the shifted $q=0$ affine algebra (Theorem \ref{theorem 4}).

\subsection{Acknowledgements}
This article is the author's thesis work, and the author would like to thank his supervisor Sabin Cautis for his patients and guidance. Many valuable ideas were provided by him. Next, the author would like to thank Roman Bezrukavnikov, Joel Kamnitzer, Yukinobu Toda and Yu Zhao for providing some helpful quick comments about this article, and also thank Wu-Yen Chuang for some helpful discussions. Finally, special thank to Harrison Chen for his careful reading of the introduction and for providing much helpful feedback.

\section{Shifted $q=0$ affine algebra} \label{section2}
In this section, we first recall the definition of shifted quantum affine algebra from \cite{FT}. Then we give the definition of shifted $q=0$ affine algebra. 

\subsection{Shifted quantum affine algebra}
We give the definition of shifted quantum affine algebras from \cite{FT}, which basically is from Section 5 in \textit{loc. cit.}.

First, we fix some notations. Let $\cg$ be a simple Lie algebra, $\ch \subset \cg$ be a Cartan subalgebra and $\big( \cdot , \cdot \big)$ be a non-degenerated invariant symmetric bilinear form on $\cg$. Let $\{\alpha_{i}\}_{i \in I} \subset \ch^*$ be the simple roots of $\cg$ relative to $\ch$ and $\{\alpha^{\vee}_{i}\}_{i \in I} \subset \ch$ be the simple coroots. Let $c_{ij}:=2\frac{( \alpha^{\vee}_{i},\alpha^{\vee}_{j})}{( \alpha^{\vee}_{i},\alpha^{\vee}_{i} )}$ be the entries of the Cartan matrix and $d_{i}:=\frac{ ( \alpha^{\vee}_{i},\alpha^{\vee}_{i} )}{2}$  such that $d_{i}c_{ij}=d_{j}c_{ji}$ for any $i,j \in I$. We also fix the notations $q_{i}:=q^{d_{i}}$, $[m]_{q}:=\frac{q^{m}-q^{-m}}{q-q^{-1}}$ and ${a \brack b}_{q}=\frac{[a-b+1]_{q}...[a]_{q}}{[1]_{q}...{b}_{q}}$.

\begin{definition} \label{definition'}
	Given two coweights $\mu^{+}, \mu^{-}$, set $b^{\pm}_{i}:=\alpha_{i}(\mu^{\pm})$. Then \textit{the shifted quantum affine algebra} (simply-connected version), denoted by $\Uu_{\mu^{+},\mu^{-}}$, is an associated $\CC(q)$ algebra generated by 
	\[
	\{e_{i,r},\ f_{i,r},\ (\psi^{\pm}_{i,\pm s^{\pm}_{i}}),\ (\psi^{\pm}_{i,\mp b^{\pm}_{i}})^{-1} \}^{r \in \ZZ, \ s_{i}^{\pm} \geq - b_{i}^{\pm}}_{i \in I}
	\]  subject to the following relations (for all $i, j \in I$ and $\epsilon, \epsilon' \in \{\pm\}$)
	\begin{equation}\tag{U1}   \label{U1}
    [\psi^{\epsilon}_{i}(z),\psi^{\epsilon'}_{j}(w)]=0, \ \psi^{\pm}_{i,\mp b^{\pm}_{i}}(\psi^{\pm}_{i,\mp b^{\pm}_{i}})^{-1}=(\psi^{\pm}_{i,\mp b^{\pm}_{i}})^{-1}\psi^{\pm}_{i,\mp b^{\pm}_{i}}=1,
	\end{equation}
	\begin{equation} \tag{U2} \label{U2}
    (z-q^{c_{ij}}_{i}w)e_{i}(z)e_{j}(w)=(q^{c_{ij}}_{i}z-w)e_{j}(w)e_{i}(z),
	\end{equation}
	\begin{equation}\tag{U3} \label{U3}
    (q^{c_{ij}}_{i}z-w)f_{i}(z)f_{j}(w)=(z-q^{c_{ij}}_{i}w)f_{j}(w)f_{i}(z),
	\end{equation}
	\begin{equation}\tag{U4} \label{U4}
     (z-q^{c_{ij}}_{i}w)\psi^{\epsilon}_{i}(z)e_{j}(w)=(q^{c_{ij}}_{i}z-w)e_{j}(w)\psi^{\epsilon}_{i}(z),
	\end{equation}
	\begin{equation} \tag{U5} \label{U5}
	(q^{c_{ij}}_{i}z-w)\psi^{\epsilon}_{i}(z)f_{j}(w)=(z-q^{c_{ij}}_{i}w)f_{j}(w)\psi^{\epsilon}_{i}(z),
	\end{equation}
	\begin{equation} \tag{U6} \label{U6}
	   [e_{i}(z),f_{j}(w)]=\frac{\delta_{ij}}{q_{i}-q^{-1}_{i}}\delta(\frac{z}{w})(\psi_{i}^{+}(z)-\psi_{i}^{-}(z)),
	\end{equation}
	\begin{align} 
	& \Sym_{z_{1},...,z_{1-c_{ij}}} \sum^{1-c_{ij}}_{r=0} (-1)^{r} {1-c_{ij} \brack r}_{q_{i}} e_{i}(z_{1})....e_{i}(z_{r})e_{j}(w)e_{i}(z_{r+1})...e_{i}(z_{1-c_{ij}})=0, \tag{U7} \label{U7}  \\
	&\Sym_{z_{1},...,z_{1-c_{ij}}} \sum^{1-c_{ij}}_{r=0} (-1)^{r} {1-c_{ij} \brack r}_{q_{i}} f_{i}(z_{1})....f_{i}(z_{r})f_{j}(w)f_{i}(z_{r+1})...f_{i}(z_{1-c_{ij}})=0, \tag{U8} \label{U8}
	\end{align}
	where $\Sym_{z_{1},...,z_{s}}$ stands for the symmetrization in $z_{1},...,z_{s}$ and the generating series are defined as follows 
	\begin{equation*}
	   e_{i}(z):=\sum_{r\in \ZZ}e_{i,r}z^{-r}, f_{i}(z):=\sum_{r\in \ZZ}f_{i,r}z^{-r}, \psi^{\pm}_{i}(z):=\sum_{r \geq -b^{\pm}_{i}}\psi^{\pm}_{i, \pm r}z^{\mp r}, \delta(z):=\sum_{r \in \ZZ} z^{r}.
	\end{equation*}
\end{definition}

Let us introduce another set of Cartan generators $\{h_{i,r}\}^{r>0}_{i \in I}$ via \begin{equation*}
    (\psi^{\pm}_{i, \mp b^{\pm}_{i}}z^{\pm b_{i}^{\pm}})^{-1}\psi^{\pm}_{i}(z)=\exp\Big(\pm (q_{i}-q^{-1}_{i}) \sum_{r>0} h_{i,\pm r}z^{\mp r}\Big).
\end{equation*} With this, the relations (\ref{U4}), (\ref{U5}) are equivalent to the following:
\begin{align*}
    &\psi^{\pm}_{i, \mp b^{\pm}_{i}}e_{j,s}=q^{\pm c_{ij}}_{i}e_{j,s}\psi^{\pm}_{i, \mp b^{\pm}_{i}}, \ [h_{i,r},e_{j,s}]=\frac{[rc_{ij}]_{q_{i}}}{r}e_{j,r+s}, \\
    &\psi^{\pm}_{i, \mp b^{\pm}_{i}}f_{j,s}=q^{\mp c_{ij}}_{i}f_{j,s}\psi^{\pm}_{i, \mp b^{\pm}_{i}}, \ [h_{i,r},f_{j,s}]=-\frac{[rc_{ij}]_{q_{i}}}{r}f_{j,r+s}.
\end{align*}

We mention some remarks about the properties of $\Uu_{\mu^{+},\mu^{-}}$ that have been addressed in \cite{FT}.

\begin{remark}
(1) The algebra $\Uu_{\mu^{+},\mu^{-}}$ depends only on $\mu:=\mu^{+}+\mu^{-}$ up to isomorphism. We say that $\Uu_{\mu^{+},\mu^{-}}$ is dominantly (resp. antidominantly) shifted if $\mu$ is a dominant (resp. antidominant) coweight.  \\
(2) We have $\Uu_{0,0}/(\psi^{+}_{i,0}\psi^{-}_{i,0}-1) \simeq \Uu_{q}(L\cg)$-the standard quantum loop algebra of $\cg$.
\end{remark}

When $\Uu_{\mu^{+},\mu^{-}}$ is antidominantly shifted (i.e. $\mu=\mu^{+}+\mu^{-}$ is antidominant), then it admits another presentation, which is the so-called Levendorskii type presentation.

\begin{definition} \label{definition 6} 
	Given antidominant coweights $\mu_{1}, \mu_{2}$, set $\mu=\mu_{1}+\mu_{2}$. Define $b_{1,i}:=\alpha_{i}(\mu_{1})$, $b_{2,i}:=\alpha_{i}(\mu_{2})$, $b_{i}=b_{1,i}+b_{2,i}$. Then we denote $\hat{\Uu}_{\mu_{1},\mu_{2}}$ to be the associated $\CC(q)$ algebra generated by 
	\[
	\{e_{i,r},\ f_{i,s},\ (\psi^{+}_{i,0})^{\pm 1},\ (\psi^{-}_{i,b_{i}})^{\pm 1},\ h_{i,\pm 1}\ | \ i\in I, \ b_{2,i}-1 \leq r \leq 0, \ b_{1,i} \leq s \leq 1 \}
	\] 
	subject to the following relations
	\begin{equation}\tag{U1'}   \label{U1'}
	\{(\psi^{+}_{i,0})^{\pm 1},\ (\psi^{-}_{i,b_{i}})^{\pm 1},\ h_{i,\pm 1} \}_{i \in I} \ \mathrm{pairwise\ commute},
	\end{equation}
	\begin{equation} \tag{U2'} \label{U2'}
	(\psi^{+}_{i,0})^{\pm 1} \cdot (\psi^{+}_{i,0})^{\mp 1}= (\psi^{-}_{i,b_{i}})^{\pm 1} \cdot (\psi^{-}_{i,b_{i}})^{\mp 1}=1,
	\end{equation}
	\begin{equation}\tag{U3'} \label{U3'}
	e_{i,r+1}e_{j,s}-q_{i}^{c_{ij}}e_{i,r}e_{j,s+1}=q_{i}^{c_{ij}}e_{j,s}e_{i,r+1}-e_{j,s+1}e_{i,r},
	\end{equation}
	\begin{equation}\tag{U4'} \label{U4'}
	q_{i}^{c_{ij}}f_{i,r+1}f_{j,s}-f_{i,r}f_{i,s+1}=f_{j,s}f_{i,r+1}-q_{i}^{c_{ij}}f_{j,s+1}f_{i,r},
	\end{equation}
	\begin{align} 
	&\psi_{i,0}^{+}e_{j,r}=q_{i}^{c_{ij}}e_{j,r}\psi_{i,0}^{+}, \ \psi_{i,b_{i}}^{-}e_{j,r}=q_{i}^{-c_{ij}}e_{j,r}\psi_{i,b_{i}}^{-},\ [h_{i,\pm 1},e_{j,r}]=[c_{ij}]_{q_{i}} e_{j,r\pm 1}, \tag{U5'} \label{U5'} \notag \\
	&\psi_{i,0}^{+}f_{j,s}=q_{i}^{-c_{ij}}f_{j,s}\psi_{i,0}^{+}, \ \psi_{i,b_{i}}^{-}f_{j,s}=q_{i}^{c_{ij}}f_{j,s}\psi_{i,b_{i}}^{-},\ [h_{i,\pm 1},f_{j,s}]=-[c_{ij}]_{q_{i}} f_{j,s\pm 1}, \tag{U6'} \label{U6'}
	\end{align}
	\begin{equation} \tag{U7'} \label{U7'} 
	[e_{i,r},f_{j,s}]=0 \ \mathrm{if} \ i \neq j\ \ \mathrm{and}\ \  [e_{i,r},f_{i,s}]= \begin{cases}
	\psi^{+}_{i,0}h_{i,1} & \text{if} \  r+s=1 \\
	\frac{\psi^{+}_{i,0}-\delta_{b_{i},0}\psi_{i,b_{i}}^{-}}{q_{i}-q_{i}^{-1}} & \text{if} \ r+s=0 \\
	0 & \text{if} \  b_{i}+1 \leq r+s \leq-1 \\
	\frac{-\psi^{-}_{i,b_{i}}+\delta_{b_{i},0}\psi_{i,0}^{-}}{q_{i}-q_{i}^{-1}} & \text{if} \ r+s=b_{i} \\
	\psi^{-}_{i,b_{i}}h_{i,-1} & \text{if} \ r+s=b_{i}-1
	\end{cases},
	\end{equation}
	\begin{equation}\tag{U8'} \label{U8'} 
	\sum_{r=0}^{1-c_{ij}}(-1)^{r} { 1-c_{ij} \brack r}_{q_{i}} e^{r}_{i,0}e_{j,0}e_{i,0}^{1-c_{ij}-r}=0, \ \sum_{r=0}^{1-c_{ij}}(-1)^{r}{ 1-c_{ij} \brack r}_{q_{i}} f^{r}_{i,0}f_{j,0}f_{i,0}^{1-c_{ij}-r}=0,
	\end{equation}
	\begin{equation}\tag{U9'} \label{U9'}
	[h_{i,1},[f_{i,1},[h_{i,1},e_{i,0}]]]=0,\ [h_{i,-1},[e_{i,b_{2,i}-1},[h_{i,-1},f_{i,b_{1,i}}]]]=0,
	\end{equation}
	for any $i,j \in I$ and $r,s$ such that the above relations make sense.
\end{definition}

With those generators, then we define inductively
\begin{align*}
    e_{i,r}&:=[2]^{-1}_{q_{i}}\begin{cases}
    [h_{i,1},e_{i,r-1}] & \text{if} \ r>0 \\
    [h_{i,-1},e_{i,r+1}] & \text{if} \ r<b_{2,i}-1,
    \end{cases} \\
    f_{i,r}&:=-[2]^{-1}_{q_{i}}\begin{cases}
    [h_{i,1},f_{i,r-1}] & \text{if} \ r>1 \\
    [h_{i,-1},f_{i,r+1}] & \text{if} \ r<b_{1,i},
    \end{cases} \\
    \psi_{i,r}^{+}&:=(q_{i}-q^{-1}_{i})[e_{i,r-1},f_{i,1}] \ \text{for} \ r>0, \\
    \psi_{i,r}^{-}&:=(q^{-1}_{i}-q_{i})[e_{i,r-b_{1,i}},f_{i,b_{1,i}}] \ \text{for} \ r<b_{i}.
\end{align*}

Then we have the following theorem, which says that in the antidominantly shifted setting, the two presentations from Definition \ref{definition'} and Definition \ref{definition 6} are equivalent.

\begin{theorem}[Theorem 5.5 in \cite{FT}] \label{Theorem''}
There is a $\CC(q)$-algebra isomorphism $\hat{\Uu}_{\mu_{1},\mu_{2}} \rightarrow \Uu_{0,\mu}$ such that 
\begin{equation*}
e_{i,r} \mapsto e_{i,r}, \ f_{i,r} \mapsto f_{i,r}, \ \psi^{\pm}_{i,\pm s^{\pm}_{i}} \mapsto \psi^{\pm}_{i,\pm s^{\pm}_{i}} \ \text{for} \ i \in I, \ r \in \ZZ, \ s^{+}_{i} \geq 0, \ s^{-}_{i} \geq -b_{i}.
\end{equation*}
\end{theorem}

\begin{remark}
    We list some relations explicitly for the readers when $\cg=\SL_{n}$, which is the main type of Lie algebras that we will study for the shifted $q=0$ affine algebra later. In this case, we have $c_{ij}=\begin{cases}
	2 & \text{if} \ i=j \\
	-1 & \text{if} \ |i-j|=1 \\
	0 & \text{if} \ |i-j| \geq 2
	\end{cases}$, and $d_{i}=1$ for all $i$. The Cartan matrix is given by 
	\[
	(c_{ij})=
	\begin{pmatrix}
	2 & -1 & 0 & \dots  & 0 & 0 \\
	-1 & 2 & -1 & \dots  & 0 & 0 \\
	\vdots & \vdots & \vdots & \ddots & \vdots & \vdots \\
	0 & 0 & 0 & \dots  & -1 & 2 
	\end{pmatrix}.
	\]
	
	Then $q_{i}=q$ for all $i$ and the numbers $c_{ij}$ in the relations of the algebra $\hat{\Uu}_{\mu_{1},\mu_{2}}$ in Definition \ref{definition 6} for $\cg=\SL_{n}$ are known. For example, some of the relations in (\ref{U3'}) are $e_{i,r+1}e_{i,s}-q^{2}e_{i,r}e_{i,s+1}=q^{2}e_{i,s}e_{i,r+1}-e_{i,s+1}e_{i,r}$, similarly for (\ref{U4'}). The relations in (\ref{U5'}) are $\psi_{i,0}^{+}e_{i,r}=q^{2}e_{i,r}\psi_{i,0}^{+}$, $\psi_{i,b_{i}}^{-}e_{i,r}=q^{-2}e_{i,r}\psi_{i,b_{i}}^{-}$, and $[h_{i,\pm 1},e_{i,r}]=[2]_{q} e_{i,r\pm 1}$, similarly for (\ref{U6'}). The relation (\ref{U8'}) is just the (quantum) Serre relations, for example, $e_{i,0}e_{j,0}e_{i,0}=\frac{1}{[2]_{q}}(e^{2}_{i,0}e_{j,0}+e_{j,0}e^{2}_{i,0})$.
\end{remark}

\subsection{Definition of the shifted $q=0$ affine algebras}

In this section, we define the shifted $q=0$ affine algebras. We define it by imitating Definition \ref{definition 6}, which is by finite generators and relations. The main reason we use such presentation is due to its simplicity and can be easily to define a categorical action for it (see next section).

On the other hand, we define another algebra in the appendix \ref{appendix A} by using the usual loop presentation (see Definition \ref{Definition 1}). Similarly to Theorem \ref{Theorem''}, we conjecture that the two presentations, i.e. Definition \ref{definition 6} and Definition \ref{Definition 1}, are equivalent (see Conjecture \ref{conjecture 1}).

In \cite{BLM}, they introduce the dot version (or idempotent modification) $\dot{\Uu}_{q}(\SL_2)$ of $\Uu_{q}(\SL_2)$, since our motivation comes from their geometric construction, the shifted $q=0$ affine algebras we introduce below is also an idempotent version. This means that we replace the identity by the direct sum of a system of projectors, one for each element of the weight lattices. They are orthogonal idempotents for approximating the unit element. We refer to part IV in \cite{Lu} for details of such modification.

Throughout the rest of this article, we fix a positive integer $N \geq 2$. Let
\[
C(n,N):=\{\underline{k}=(k_1,...,k_{n}) \in \NN^n\ | \ k_{1}+...+k_{n}=N \}.
\] 

We regard each $\underline{k}$ as a weight for $\SL_{n}$ via the identification of the weight lattice of $\SL_n$ with the quotient $\ZZ^n/(1,1,...,1)$. We choose the simple root $\alpha_{i}$ to be $(0...0,-1,1,0...0)$ where the $-1$ is in the $i$-th position for $1 \leq i \leq n-1$. Then we introduce the shifted $q=0$ affine algebra for $\SL_n$, which is defined by using finite generators and relations.

\begin{definition} \label{definition 1}
	We define the \textit{shifted q=0 affine algebra}, denote by $\dot{\Uu}_{0,N}(L\SL_n)$, to be the associative $\CC$-algebra generated by
	\begin{equation*}
	\bigcup_{\kk \in C(n,N)}\{1_{\kk}, \ e_{i,r}1_{\kk}, \ f_{i,s}1_{\kk},\ (\psi^{+}_{i})^{\pm 1}1_{\kk}, \ (\psi^{-}_{i})^{\pm 1}1_{\kk},\ h_{i,\pm 1}1_{\kk} \}_{1 \leq i \leq n-1}^{-k_{i}-1 \leq r \leq 0, \ 0 \leq s \leq k_{i+1}+1}
	\end{equation*}
	with the following relations
	\begin{equation} \tag{U01}   \label{U01}
	1_{\kk}1_{\Ll}=\delta_{\kk,\Ll}1_{\kk}, \ e_{i,r}1_{\kk}=1_{\kk+\alpha_{i}}e_{i,r}, \   f_{i,r}1_{\kk}=1_{\kk-\alpha_{i}}f_{i,r}, \ (\psi^{+}_{i})^{\pm 1}1_{\kk}=1_{\kk}(\psi^{+}_{i})^{\pm 1},\
	h_{i,\pm 1}1_{\kk}=1_{\kk}h_{i, \pm 1},	
	\end{equation}
	\begin{equation}\tag{U02}   \label{U02}
	\{(\psi^{+}_{i})^{\pm 1}1_{\kk},(\psi^{-}_{i})^{\pm 1}1_{\kk},h_{i,\pm 1}1_{\kk}\ | \ 1\leq i \leq n-1,\kk \in C(n,N)\} \ \mathrm{pairwise\ commute},
	\end{equation}
	\begin{equation} \tag{U03} \label{U03}
	(\psi^{+}_{i})^{\pm 1} \cdot (\psi^{+}_{i})^{\mp 1} 1_{\kk} = 1_{\kk}=(\psi^{-}_{i})^{\pm 1} \cdot (\psi^{-}_{i})^{\mp 1}1_{\kk},
	\end{equation}
	\begin{align*} \tag{U04} \label{U04}
	&e_{i,r}e_{j,s}1_{\kk}=\begin{cases}
	-e_{i,s+1}e_{i,r-1}1_{\kk} & \text{if} \ j=i \\
	e_{i+1,s}e_{i,r}1_{\kk}-e_{i+1,s-1}e_{i,r+1}1_{\kk} & \text{if} \ j=i+1 \\
	e_{i,r+1}e_{i-1,s-1}1_{\kk}-e_{i-1,s-1}e_{i,r+1}1_{\kk} &\text{if} \ j=i-1 \\
	e_{j,s}e_{i,r}1_{\kk} &\text{if} \ |i-j| \geq 2
	\end{cases},
	\end{align*}
	
	\begin{align*}\tag{U05} \label{U05}
	&f_{i,r}f_{j,s}1_{\kk}=\begin{cases}
	-f_{i,s-1}f_{i,r+1}1_{\kk} & \text{if} \ j=i \\
	f_{i,r-1}f_{i+1,s+1}1_{\kk}-f_{i+1,s+1}f_{i,r-1}1_{\kk} & \text{if} \ j=i+1 \\
	f_{i-1,s}f_{i,r}1_{\kk}-f_{i-1,s+1}f_{i,r-1}1_{\kk} &\text{if} \ j=i-1 \\
	f_{j,s}f_{i,r}1_{\kk} &\text{if} \ |i-j| \geq 2
	\end{cases},
	\end{align*}

	\begin{align*} \tag{U06} \label{U06}
	&\psi^{+}_{i}e_{j,r}1_{\kk}=
	\begin{cases}
	-e_{i,r+1}\psi^{+}_{i}1_{\kk} & \text{if} \ j=i \\
	-e_{i+1,r-1}\psi^{+}_{i}1_{\kk} & \text{if} \ j=i+1 \\
	e_{i-1,r}\psi^{+}_{i}1_{\kk}  &  \text{if} \ j=i-1 \\
	e_{j,r}\psi^{+}_{i}1_{\kk}  &  \text{if} \ |i-j| \geq 2 \\
	\end{cases},
	&\psi^{-}_{i}e_{j,r}1_{\kk}=
	\begin{cases}
	-e_{i,r+1}\psi^{-}_{i}1_{\kk} & \text{if} \ j=i \\
	e_{i+1,r}\psi^{-}_{i}1_{\kk} & \text{if} \ j=i+1 \\
	-e_{i-1,r-1}\psi^{-}_{i}1_{\kk}  &  \text{if} \ j=i-1 \\
	e_{j,r}\psi^{-}_{i}1_{\kk}  &  \text{if} \ |i-j| \geq 2 \\
	\end{cases},
	\end{align*}

	\begin{align*} \tag{U07} \label{U07}
	&\psi^{+}_{i}f_{j,r}1_{\kk}=
	\begin{cases}
	-f_{i,r-1}\psi^{+}_{i}1_{\kk} & \text{if} \ j=i \\
	-f_{i+1,r+1}\psi^{+}_{i}1_{\kk} & \text{if} \ j=i+1 \\
	f_{i-1,r}\psi^{+}_{i}1_{\kk}  &  \text{if} \ j=i-1 \\
	f_{j,r}\psi^{+}_{i}1_{\kk}  &  \text{if} \ |i-j| \geq 2 \\
	\end{cases},
	&\psi^{-}_{i}f_{j,r}1_{\kk}=
	\begin{cases}
	-f_{i,r-1}\psi^{-}_{i}1_{\kk} & \text{if} \ j=i \\
	f_{i+1,r}\psi^{-}_{i}1_{\kk} & \text{if} \ j=i+1 \\
	-f_{i-1,r+1}\psi^{-}_{i}1_{\kk}  &  \text{if} \ j=i-1 \\
	f_{j,r}\psi^{-}_{i}1_{\kk}  &  \text{if} \ |i-j| \geq 2 \\
	\end{cases},
	\end{align*}

	\begin{align*} \tag{U08} \label{U08} 
	&[h_{i,\pm 1},e_{j,r}]1_{\kk}=\begin{cases}
	0 & \text{if} \ i=j \\
	-e_{i+1,r\pm 1}1_{\kk} & \text{if} \ j=i+1 \\
	e_{i-1,r\pm 1}1_{\kk} & \text{if} \  j=i-1 \\
	0& \text{if} \ |i-j| \geq 2 \\
	\end{cases}, 
	&[h_{i,\pm 1},f_{j,r}]1_{\kk}=\begin{cases}
	0 & \text{if} \ i=j \\
	f_{i+1,r\pm 1}1_{\kk} & \text{if} \ j=i+1 \\
	-f_{i-1,r\pm 1}1_{\kk} & \text{if} \  j=i-1 \\
	0& \text{if} \ |i-j| \geq 2 \\
	\end{cases},
	\end{align*}
	
	\begin{equation} \tag{U09} \label{U09} 
	[e_{i,r},f_{j,s}]1_{\kk}=0 \ \mathrm{if} \ i \neq j\ \ \mathrm{and}\ \  [e_{i,r},f_{i,s}]1_{\kk}= \begin{cases}
	\psi^{+}_{i}h_{i,1}1_{\kk} & \text{if} \  r+s=k_{i+1}+1 \\
	\psi^{+}_{i}1_{\kk} & \text{if} \ r+s=k_{i+1} \\
	0 & \text{if} \  -k_{i}+1 \leq r+s \leq k_{i+1}-1 \\
	-\psi^{-}_{i}1_{\kk} & \text{if} \ r+s=-k_{i} \\
	-\psi^{-}_{i}h_{i,-1}1_{\kk} & \text{if} \ r+s=-k_{i}-1
	\end{cases},
	\end{equation}
	
	for any $1 \leq i,j \leq n-1$ and $r,s$ such that the above relations make sense.
\end{definition}

We discuss a bit about the use of the notation $\dot{\Uu}_{0,N}(L\SL_n)$, and why we call it the name shifted $q=0$ affine algebra. We make the discussion to be a list of remarks.

\begin{remark}
First, the $0$ stands for the word $"q=0"$. However, the word $"q=0"$ does not mean that we  substitute $q=0$ directly in the relations of the shifted quantum affine algebras.

Many of the relations in Definition \ref{definition 6} can not be substituted by $q=0$ directly, since $q^{-1}$ will appear when $|i-j|=1$ and also in the commutator relations. But some of them can be, from the relations (\ref{U3'}), (\ref{U4'}), in $\SL_n$ case we have $c_{ij}=2$ when $i=j$. Taking $q=0$, we can see that they become the relations (\ref{U04}), (\ref{U05}) when $i=j$. 
\end{remark}

\begin{remark}
Second, the number $N$ stands for the word "shifted".
Also, the word "shifted" does not really mean that our algebra is shifted by a certain coweight. The number $N$ indicates that our algebra depends on a choice of highest weight (for example, $(0,N)$ in the $\SL_{2}$ case). Moreover, the commutator relation (\ref{U09}) implies that the commutators between $e_{i,r}$ and $f_{i,s}$ on each weight space with weight $\kk$ vanishing on the range $1-k_{i} \leq r+s \leq k_{i+1}-1$. This can be modified to $-k_{i}-k_{i+1}+1 \leq -r-s-k_{i} \leq -1$, and it is similar to the range in the relation (\ref{U7'}) with $b_{i}=-k_{i}-k_{i+1}$. This pretty much suggests us that our algebra is also antidominantly shifted in the sense of \cite{FT}, but in a weight space-wise pattern.
\end{remark}

\begin{remark}
Third, although we call it "affine algebra", we use the loop algebra $L\SL_{n}$ notation instead of $\widehat{\SL_{n}}$. The reason is that the representations we consider in this article are all finite dimensional, and the central extension acts trivially on finite dimensional representations.
\end{remark}

Thus, although the precise relations between our algebras and the shifted quantum affine algebras are still unclear, the above evidences suggest us to call this algebra the name shifted $q=0$ affine algebra.

\section{Categorical $\dot{\Uu}_{0,N}(L\SL_n)$ action} \label{section3}

In this section, we give a definition of the categorical action for shifted $q=0$ affine algebra that defined in Definition \ref{definition 1}. 

Again, we use the notation $C(n,N)$ and $\alpha_{i}$ defined in Section \ref{section2}. We also denote by $\langle\cdot,\cdot\rangle:\ZZ^n \times \ZZ^n \rightarrow \ZZ$ the standard pairing. The definition follows the similar formalism as in \cite{CK2}, which also come from the definition of $(\cg, \theta)$ action defined in \cite{C}.

Before we give the definition, we have to mention that for the usual quantum affine algebra $\Uu_{q}(\widehat{\cg})$, there are two presentations, one is the Kac-Moody presentation and the other is the (Drinfeld-Jimbo) loop realization. The Kac-Moody presentation has the advantage that it is given by finite number of generators and relations, while the loop realization is better for checking actions (on geometry) in practice. The categorical actions for the two presentations are also quite different. 

For the Kac-Moody presentation, its categorical actions have been studied in much detail \cite{CR}, \cite{KL1} \cite{KL2}, \cite{KL3}, \cite{R}. On the other hand, the study of the categorical actions for the loop realization is much more less (see \cite{CK2} for a working definition). Thus, even though the presentation in Definition \ref{definition 1} is not so canonical compare to the usual loop presentation that given in the appendix A, the finiteness of generators and relations is the main reason for us to use  to work with and to give a definition of its categorical action.

\begin{definition}  \label{definition 2}
	A categorical $\dot{\Uu}_{0,N}(L\SL_n)$ action consists of a target 2-category $\Kk$, which is triangulated, $\CC$-linear and idempotent complete. The objects in $\Kk$ are
	\[
	\mathrm{Ob}(\Kk)=\{\Kk(\kk)\ |\ \kk \in C(n,N) \}
	\] where each $\Kk(\kk)$ is also a triangulated category, and each Hom space $\mathrm{Hom}(\Kk(\kk),\Kk(\Ll))$ is also  triangulated. 
	
	The 1-morphisms are given by the following:
	\begin{equation*}
		\bo_{\kk}, \ {\E}_{i,r}\bo_{\kk}=\bo_{\kk+\alpha_{i}}{\E}_{i,r}, \ {\F}_{i,s}\bo_{\kk}=\bo_{\kk-\alpha_{i}}{\F}_{i,s}, \ ({\spi}^{\pm}_{i})^{\pm 1}\bo_{\kk}=\bo_{\kk}({\spi}^{\pm}_{i})^{\pm 1}, \ {\Hhh}_{i,\pm1}\bo_{\kk}=\bo_{\kk}{\Hhh}_{i,\pm 1}   
	\end{equation*} where $1 \leq i \leq n-1$, $-k_{i}-1 \leq r \leq 0$, $0 \leq s \leq k_{i+1}+1$. Here $\bo_{\kk}$ is the identity functor of $\Kk(\kk)$. Those 1-morphisms subject to the following conditions.
	\begin{enumerate}
		\item The space of maps between any two 1-morphisms is finite dimensional.
		\item If $\alpha=\alpha_{i}$ or $\alpha=\alpha_{i}+\alpha_{j}$ for some $i,j$ with $\langle\alpha_{i},\alpha_{j}\rangle=-1$, then $\bo_{\kk+r\alpha}=0$ for $r \gg 0$ or  $r \ll 0$.
		\item Suppose $i \neq j$. If $\bo_{\kk+\alpha_{i}}$ and $\bo_{\kk+\alpha_{j}}$ are nonzero, then $\bo_{\kk}$ and $\bo_{\kk+\alpha_{i}+\alpha_{j}}$ are also nonzero.
		\item ${\Hhh}_{i,\pm1}\bo_{\kk}$ are adjoint to each other, i.e. $({\Hhh}_{i,1}\bo_{\kk})^{L} \cong \bo_{\kk}{\Hhh}_{i,-1}  \cong ({\Hhh}_{i,1}\bo_{\kk})^{R}$.
		\item The left and right adjoints of $\E_{i,r}$ and $\F_{i,s}$ are given by conjugation of $\spi^{\pm}_{i}$ up to homological shifts. More precisely,
		\begin{enumerate}
			\item $({\E}_{i,r}\bo_{\kk})^{R} \cong \bo_{\kk}{({\spi}^{+}_{i})^{r+1}}{\F}_{i,k_{i+1}+2}({\spi}^{+}_{i})^{-r-2}[-r-1]$ for all $1 \leq i \leq n-1$,
			\item $({\E}_{i,r}\bo_{\kk})^{L} \cong \bo_{\kk}({\spi}^{-}_{i})^{r+k_{i}-1}{\F}_{i,0}({\spi}^{-}_{i})^{-r-k_{i}}[r+k_{i}]$ for all $1 \leq i \leq n-1$,
			\item $({\F}_{i,s}\bo_{\kk})^{R} \cong \bo_{\kk}({\spi}^{-}_{i})^{-s+1}{\E}_{i,-k_{i}-2}({\spi}^{-}_{i})^{s-2}[s-1]$ for all $1 \leq i \leq n-1$,
			\item $({\F}_{i,s}\bo_{\kk})^{L} \cong \bo_{\kk}({\spi}^{+}_{i})^{-s+k_{i+1}-1}{\E}_{i,0}({\spi}^{+}_{i})^{s-k_{i+1}}[-s+k_{i+1}]$ for all $1 \leq i \leq n-1$.
		\end{enumerate}
		\item 	
		\begin{align*}
		&({\spi}_{i}^{\pm})^{\pm 1} ({\spi}_{j}^{\pm})^{\pm 1} \bo_{\kk} \cong ({\spi}_{j}^{\pm})^{\pm 1}({\spi}_{i}^{\pm})^{\pm 1} \bo_{\kk} \ \text{for all} \ i, j, \\
		&({\spi}_{i}^{+})^{\pm 1} ({\spi}_{i}^{+})^{\mp 1} \bo_{\kk} \cong ({\spi}_{i}^{-})^{\pm 1}({\spi}_{i}^{-})^{\mp 1} \bo_{\kk} \cong \bo_{\kk} \ \text{for all} \ i.
		\end{align*}
		\item ${\Hhh}_{i,\pm 1} {\Hhh}_{j,\pm 1} \bo_{\kk}  \cong  {\Hhh}_{j,\pm 1} {\Hhh}_{i,\pm 1} \bo_{\kk}$ for all $i,j$.
		\item ${\Hhh}_{i,\pm 1} {\spi}_{j}^{\pm} \bo_{\kk} \cong {\spi}_{j}^{\pm} {\Hhh}_{i,\pm 1} \bo_{\kk}$ for all $i,j$.
		\item The relations between ${\E}_{i,r}$, ${\E}_{j,s}$ are given by the following
		\begin{enumerate}
			\item We have 
			\[
			{\E}_{i,r+1}{\E}_{i,s}\bo_{\kk} \cong \begin{cases}
			{\E}_{i,s+1}{\E}_{i,r}\bo_{\kk}[-1] & \text{if} \ r-s \geq 1 \\
			0 & \text{if} \ r=s \\
			{\E}_{i,s+1}{\E}_{i,r}\bo_{\kk}[1] & \text{if} \ r-s \leq -1. 
			\end{cases}
			\]
			\item ${\E}_{i,r}$, ${\E}_{i+1,s}$  would related by the following exact triangle
			\[
			 {\E}_{i+1,s}{\E}_{i,r+1}\bo_{\kk} \rightarrow {\E}_{i+1,s+1}{\E}_{i,r}\bo_{\kk} \rightarrow {\E}_{i,r}{\E}_{i+1,s+1}\bo_{\kk}. 
			\]
			\item We have 
			\[
			{\E}_{i,r}{\E}_{j,s}\bo_{\kk} \cong {\E}_{j,s}{\E}_{i,r}\bo_{\kk}, \  \text{if} \ |i-j| \geq 2.
			\]
		\end{enumerate}
		\item The relations between ${\F}_{i,r}$, ${\F}_{j,s}$ are given by the following
		\begin{enumerate}
			\item We have
			\[
			{\F}_{i,r}{\F}_{i,s+1}\bo_{\kk} \cong \begin{cases}
			{\F}_{i,s}{\F}_{i,r+1}\bo_{\kk}[1] & \text{if} \ r-s \geq 1 \\
			0 & \text{if} \ r=s \\
			{\F}_{i,s}{\F}_{i,r+1}\bo_{\kk}[-1] & \text{if} \ r-s \leq -1. 
			\end{cases}  
			\]
			\item ${\F}_{i,r}$, ${\F}_{i+1,s}$ are related by the following exact triangles
			\[
			{\F}_{i,r+1}{\F}_{i+1,s}\bo_{\kk} \rightarrow {\F}_{i,r}{\F}_{i+1,s+1}\bo_{\kk} \rightarrow {\F}_{i+1,s+1}{\F}_{i,r}\bo_{\kk} .
			\] 
			\item  We have 
			\[
			{\F}_{i,r}{\F}_{j,s}\bo_{\kk} \cong {\F}_{j,s}{\F}_{i,r}\bo_{\kk},  \ \text{if} \ |i-j| \geq 2.
			\]
		\end{enumerate}

		\item The relations between ${\E}_{i,r}$, $\Psi^{\pm}_{j}$ are given by the following
		\begin{enumerate}
			\item For $i=j$, we have 
			\[
			{\spi}^{\pm}_{i}{\E}_{i,r}\bo_{\kk} \cong {\E}_{i,r+1}{\spi}^{\pm}_{i}\bo_{\kk}[\mp 1].
			\]
			\item For $|i-j|=1$, we have the following
			\[
			{\spi}^{\pm}_{i}{\E}_{i\pm1,r}\bo_{\kk} \cong {\E}_{i\pm1,r-1}{\spi}^{\pm}_{i}\bo_{\kk}[\pm 1],
			\]
			\[
			{\spi}^{\pm}_{i}{\E}_{i\mp 1,r}\bo_{\kk} \cong {\E}_{i\mp 1,r} {\spi}^{\pm}_{i}\bo_{\kk}.
			\]
			\item For $|i-j| \geq 2$, we have
			\[
			{\spi}^{\pm}_{i}{\E}_{j,r}\bo_{\kk} \cong {\E}_{j,r} {\spi}^{\pm}_{i}\bo_{\kk}.
			\]
		\end{enumerate}
		
		\item The relations between ${\F}_{i,r}$, $\Psi^{\pm}_{j}$ are given by the following
		\begin{enumerate}
			\item For $i=j$, we have 
			\[
			{\spi}^{\pm}_{i}{\F}_{i,r}\bo_{\kk} \cong {\F}_{i,r-1}{\spi}^{\pm}_{i}\bo_{\kk}[\pm 1].
			\]
			\item For $|i-j|=1$, we have the following
			\[
			{\spi}^{\pm}_{i}{\F}_{i\pm1,s}\bo_{\kk} \cong {\F}_{i\pm1,s+1}{\spi}^{\pm}_{i}\bo_{\kk}[\mp 1],
			\]
			\[
			{\spi}^{\pm}_{i}{\F}_{i\mp1,r}\bo_{\kk} \cong {\F}_{i\mp1,r}{\spi}^{\pm}_{i}\bo_{\kk}.
			\]
			\item For  $|i-j| \geq 2$, we have
			\[
			{\spi}^{\pm}_{i,}{\F}_{j,r}\bo_{\kk} \cong {\F}_{j,r}{\spi}^{\pm}_{i}\bo_{\kk}.
			\]
		\end{enumerate}
		
		\item The relations between ${\E}_{i,r}$, $\Hhh_{j,\pm 1}$ are given by the following
		\begin{enumerate}
			\item For $i=j$, they are related by the following exact triangles
			\[
			{\Hhh}_{i,1}{\E}_{i,r}\bo_{\kk} \rightarrow {\E}_{i,r}{\Hhh}_{i,1}\bo_{\kk} \rightarrow ({\E}_{i,r+1}\bigoplus {\E}_{i,r+1}[1])\bo_{\kk}, 
			\]
			\[
			 {\E}_{i,r}{\Hhh}_{i,-1}\bo_{\kk} \rightarrow {\Hhh}_{i,-1}{\E}_{i,r}\bo_{\kk} \rightarrow   ({\E}_{i,r-1}\bigoplus {\E}_{i,r-1}[1])\bo_{\kk}.
			\]
			\item For $|i-j|=1$, they are related by the following exact triangles
			\[
			 {\Hhh}_{i,1}{\E}_{i+1,r}\bo_{\kk} \rightarrow {\E}_{i+1,r}{\Hhh}_{i,1}\bo_{\kk} \rightarrow {\E}_{i+1,r+1}\bo_{\kk} ,
			\]
			\[
			{\E}_{i-1,r+1}\bo_{\kk} \rightarrow {\Hhh}_{i,1}{\E}_{i-1,r}\bo_{\kk} \rightarrow {\E}_{i-1,r}{\Hhh}_{i,1}\bo_{\kk} ,
			\]
			\[
			{\E}_{i+1,r-1}\bo_{\kk} \rightarrow {\E}_{i+1,r}{\Hhh}_{i,-1}\bo_{\kk} \rightarrow {\Hhh}_{i,-1}{\E}_{i+1,r}\bo_{\kk} , 
			\]
			\[
			 {\E}_{i-1,r}{\Hhh}_{i,-1}\bo_{\kk} \rightarrow {\Hhh}_{i,-1}{\E}_{i-1,r}\bo_{\kk} \rightarrow {\E}_{i-1,r-1}\bo_{\kk} . 
			\]
			\item For $|i-j|\geq 2$, we have 
			\[
			{\Hhh}_{i,\pm 1}{\E}_{j,r}\bo_{\kk} \cong {\E}_{j,r}{\Hhh}_{i,\pm 1}\bo_{\kk}.
			\]
		\end{enumerate}
		\item The relations between ${\F}_{i,r}$, $\Hhh_{j,\pm 1}$ are given by the following
		\begin{enumerate}
			\item For $i=j$, they are related by the following exact triangles
			\[
			{\F}_{i,r}{\Hhh}_{i,1}\bo_{\kk} \rightarrow {\Hhh}_{i,1}{\F}_{i,r}\bo_{\kk} \rightarrow  ({\F}_{i,r+1}\bigoplus {\F}_{i,r+1}[1])\bo_{\kk},
			\]
			\[
			{\Hhh}_{i,-1}{\F}_{i,r}\bo_{\kk} \rightarrow {\F}_{i,r}{\Hhh}_{i,-1}\bo_{\kk} \rightarrow ({\F}_{i,r-1}\bigoplus {\F}_{i,r-1}[1])\bo_{\kk} .
			\]
			\item For $|i-j|=1$, they are related by the following exact triangles
			\[
			 {\F}_{i+1,r}{\Hhh}_{i,1}\bo_{\kk} \rightarrow {\Hhh}_{i,1}{\F}_{i+1,r}\bo_{\kk} \rightarrow {\F}_{i+1,r+1}\bo_{\kk} ,
			\]
			\[
			 {\F}_{i-1,r+1}\bo_{\kk} \rightarrow {\F}_{i-1,r}{\Hhh}_{i,1}\bo_{\kk} \rightarrow {\Hhh}_{i,1}{\F}_{i-1,r}\bo_{\kk} ,
			\]
			\[
			 {\F}_{i+1,r-1}\bo_{\kk} \rightarrow {\Hhh}_{i,-1}{\F}_{i+1,r}\bo_{\kk} \rightarrow {\F}_{i+1,r}{\Hhh}_{i,-1}\bo_{\kk} ,
			\]
			\[
			{\Hhh}_{i,-1}{\F}_{i-1,r}\bo_{\kk} \rightarrow {\F}_{i-1,r}{\Hhh}_{i,-1}\bo_{\kk} \rightarrow {\F}_{i-1,r-1}\bo_{\kk} .
			\]
			\item For $|i-j|\geq 2$, we have 
			\[
			{\Hhh}_{i,\pm 1}{\F}_{j,r}\bo_{\kk} \cong {\F}_{j,r}{\Hhh}_{i,\pm 1}\bo_{\kk}.
			\]
		\end{enumerate}

		\item If $i \neq j$, then ${\E}_{i,r}{\F}_{j,s}\bo_{\kk} \cong 
		{\F}_{j,s}{\E}_{i,r}\bo_{\kk}$.
		\item For ${\E}_{i,r}{\F}_{i,s}\bo_{\kk}, {\F}_{i,s}{\E}_{i,r}\bo_{\kk} \in \mathrm{Hom}(\Kk(\kk),\Kk(\kk))$, they are related by  exact triangles, more precisely, 
		\begin{enumerate}
			\item
			\[
			{\F}_{i,s}{\E}_{i,r}{\bo}_{\kk} \rightarrow {\E}_{i,r}{\F}_{i,s}{\bo}_{\kk} \rightarrow {\spi}^{+}_{i}{\Hhh}_{i,1}{\bo}_{\kk} \ \text{if} \ r+s=k_{i+1}+1,
			\]
			\item \[
			{\F}_{i,s}{\E}_{i,r}\bo_{\kk} \rightarrow {\E}_{i,r}{\F}_{i,s}\bo_{\kk} \rightarrow {\spi}^{+}_{i}\bo_{\kk} 
			 \ \text{if} \  
			r+s =k_{i+1},
			\] 
			\item \[
			 {\E}_{i,r}{\F}_{i,s}\bo_{\kk} \rightarrow {\F}_{i,s}{\E}_{i,r}\bo_{\kk} \rightarrow {\spi}^{-}_{i}\bo_{\kk}   \ \text{if} \  r+s = -k_{i},
			\] 
			\item \[
			{\E}_{i,r}{\F}_{i,s}\bo_{\kk} \rightarrow {\F}_{i,s}{\E}_{i,r}\bo_{\kk} \rightarrow {\spi}^{-}_{i}{{\Hhh}_{i,-1}}\bo_{\kk}  \ \text{if} \  r+s = -k_{i}-1 ,
			\] 
			\item \[
			{\F}_{i,s}{\E}_{i,r}\bo_{\kk} \cong {\E}_{i,r}{\F}_{i,s}\bo_{\kk} \ \text{if} \ -k_{i}+1 \leq r+s \leq k_{i+1}-1 .
			\] 
		\end{enumerate}
	\end{enumerate}
    for all $r$, $s$ that make the above conditions make sense and the isomorphisms between functors appear in every condition are abstractly defined, i.e., we do not specify any 2-morphisms that induce those isomorphisms.
\end{definition}

First, we give some remarks about this definition.

\begin{remark}
	The 2-category $\Kk$ is called idempotent complete if for any 2-morphism $f$ with $f^2=f$, the image of $f$ is contained in $\Kk$.
\end{remark}

\begin{remark}
	Note that in our definition of categorical action, we do not have the linear maps
	\[\text{Span}\{\alpha_{i} \ | \ 1 \leq i\leq n-1\} \rightarrow \End^{2}(\bo_{\kk}), \ \kk \in C(n,N)
	\] which is used to give the element $\theta$ in the definition of  $(\cg,\theta)$ or $(\hat{\cg},\theta)$ action in \cite{C} or \cite{CL}.
	This is because usually the geometry of the spaces that appear in our setting does not have a natural flat deformation, see Section \ref{section 5} for our examples. The data of flat deformation can be used to obtain linear map $\text{Span}\{\alpha_{i}\ | \ 1 \leq i\leq n-1\} \rightarrow \End^{2}(\bo_{\kk})$, which was showed in \cite{C}.
\end{remark}

Next, we discuss a bit about the categorical action we purpose above.

By categorical action, we lift the generators $e_{i,r}1_{\kk}, \ f_{i,s}1_{\kk},\ (\psi^{\pm}_{i})^{\pm 1}1_{\kk},\ h_{i,\pm 1}1_{\kk}$ to the 1-morphisms ${\E}_{i,r}\bo_{\kk}$, ${\F}_{i,s}\bo_{\kk}$, $({\spi}^{\pm}_{i})^{\pm 1}\bo_{\kk}$, ${\Hhh}_{i,\pm1}\bo_{\kk}$. Then after lifting, we naively replace the equalities between elements in the relations (\ref{U01}),...,(\ref{U09}) to (abstract) isomorphisms between the 1-morphisms. Moreover, for relations with equality involves 3 elements, we may want to do some modifications so that it becomes a meaningful relations at the categorical level like Theorem \ref{theorem} for the categorical $\SL_{2}$ action. However, we need to be a bit careful since the categories where those 1-morphism act are require to be triangulated. Thus, instead of using direct sum of 1-morphisms for the categorical action, we use exact triangles to lift relations with equality involves 3 elements.

We list the comparison between the relations in Definition \ref{definition 1} and conditions in Definition \ref{definition 2}.

Relation (\ref{U01}) is obvious. 

Relations (\ref{U02}) and (\ref{U03}) lift to conditions (6), (7), and (8). 

Relations (\ref{U04}) and (\ref{U05}) lift to conditions (9), (10). Note that the exact triangles in (9)(b) and (10)(b) lift the relations with equality involves 3 elements.

Relations (\ref{U06}) and (\ref{U07}) lift to conditions (11) and (12). Note that the minus sign in the relations lifts to either homological shifted by $[1]$ or $[-1]$ that depend on whether the relation involves $\psi_{i}^{+}$ or $\psi_{i}^{-}$.

Relation (\ref{U09}) lifts to conditions (15) and (16).  

Finally, relation (\ref{U08}) lifts to conditions (13) and (14). However, we need to be careful for this time since condition (13)(a) and (14)(a) is highly non-trivial and not so intuitive compare to the above lifts of relations. After passing to the Grothendieck groups, the homological shift $[1]$ becomes multiplication by $(-1)$ and gives the desire relations. We will give a geometric example in Section \ref{section 5} that satisfy these two non-trivial conditions (i.e. Theorem \ref{theorem 3}). 

\section{Preliminaries on coherent sheaves and Fourier-Mukai kernels} \label{section 4}
In this section, we briefly recall the notions about Fourier-Mukai transforms/kernels and other related tools that we would use for proofs in later sections.  The readers can consult the book by Huybrechts \cite{H} for details.

We will mostly work with the bounded
derived category of coherent sheaves on an algebraic variety $X$, which we simply denote it by $\Dd^b(X)$. Throughout this article, functors between derived categories are assumed to be derived functors, for example, we will just write $f^*$, $f_{*}$ instead of $Lf^*$, $Rf_{*}$, resp. 

Let $X$, $Y$ be two smooth projective varieties. A Fourier-Mukai kernel is any object $\Pp$ in the
derived category of coherent sheaves on $X \times Y$ . Given $\Pp \in \Dd^b(X \times Y)$, we define
the associated Fourier-Mukai transform, which is the functor
\[
\Phi_{\Pp}:\Dd^b(X) \rightarrow \Dd^b(Y) 
\]
\[
\ \ \ \ \ \ \ \ \Ff \mapsto \pi_{2*}(\pi_{1}^*(\Ff) \otimes \Pp).
\]

We call $\Phi_{\Pp}$ the Fourier-Mukai transform with (Fourier-Mukai) kernel $\Pp$. For convenience, we would just write FM for Fourier-Mukai. The first property of FM transforms is that they have left and right adjoints that are themselves FM transforms.

\begin{proposition} (\cite{H} Proposition 5.9) \label{Proposition 1}
	For $\Phi_{\Pp}:\Dd^b(X) \rightarrow \Dd^b(Y)$ is the FM transform with kernel $\Pp$, define 
	\[
	\Pp_{L}=\Pp^{\vee} \otimes \pi^*_{2}\omega_{Y}[\tdim  Y], \ \Pp_{R}=\Pp^{\vee} \otimes \pi^*_{1}\omega_{X}[\tdim X].
	\] Then
	\[
	\Phi_{\Pp_{L}}:\Dd^b(Y) \rightarrow \Dd^b(X), \ \text{and} \ \Phi_{\Pp_{R}}:\Dd^b(Y) \rightarrow \Dd^b(X)
	\] are the left and right adjoints of $\Phi_{\Pp}$, respectively.
\end{proposition}

The second property is the composition of FM transforms is also a FM transform.

\begin{proposition} (\cite{H} Proposition 5.10) \label{Proposition 2}
	Let $X, Y, Z$ be smooth projective varieties over $\CC$. Consider objects $\Pp \in \Dd^b(X \times Y)$ and $\Qq \in \Dd^b(Y \times Z)$. They define FM transforms $\Phi_{\Pp}:\Dd^b(X) \rightarrow \Dd^b(Y)$, $\Phi_{\Qq}:\Dd^b(Y) \rightarrow \Dd^b(Z)$.  We would use $\ast$ to denote the operation for convolution, i.e.
	\[
	\Qq \ast \Pp:=\pi_{13*}(\pi_{12}^*(\Pp)\otimes \pi_{23}^{*}(\Qq)).
	\]
	
	Then for $\R=\Qq \ast \Pp \in \Dd^b(X \times Z)$, we have $\Phi_{\Qq} \circ \Phi_{\Pp} \cong \Phi_{\R}$. 
\end{proposition}

\begin{remark} \label{remark 1}
	Moreover by \cite{H} remark 5.11, we have $(\Qq \ast \Pp)_{L} \cong (\Pp)_{L} \ast (\Qq)_{L}$ and $(\Qq \ast \Pp)_{R} \cong (\Pp)_{R} \ast (\Qq)_{R}$.
\end{remark}

The final result we will need in later is about derived pushforward of coherent sheaves. Let $\V$ be a vector bundle of rank $n$ on a variety $X$, where $n \geq 2$. Then we can form the projective bundle $\PP(\V)$. We get in this way a $\PP^{n-1}$-fibration $\pi: \PP(\V) \rightarrow X$. Let  $\Oo_{\PP(\V)}(-1)$ be the relative tautological bundle and $\Oo_{\PP(\V)}(1)$ be the dual bundle, and we define  $\Oo_{\PP(\V)}(i):=\Oo_{\PP(\V)}(1)^{\otimes i}$ for $i \in \ZZ$. Then we have the following result.

\begin{proposition}\label{proposition 3} (Exercise 8.4 in Chapter 3 from \cite{Ha})
	\begin{equation*}
	\pi_{*} \Oo_{\PP(\V)}(i) \cong   \begin{cases} \Sym^i(\V^{\vee}) & \text{if} \ i \geq 0 \\ 0 & \text{if} \  1-n \leq i\leq -1 \\ \Sym^{-i-n}(\V) \otimes \tdet(\V)[1-n]   & \text{if} \ i \leq -n \end{cases} 
	\end{equation*}
	in the derived category $\Dd^b(X)$, where $\V^{\vee}=R\mathcal{H}om(\V,\Oo_{X}) \in \Dd^b(X)$.
\end{proposition}

\begin{remark}
Note that the above result is slightly different from the result in \cite{Ha}. The reason is that the convention for $\PP(\V)$ is different from \cite{Ha}. More precisely, in this article $\PP(\V)$ parametrizes one-dimensional "sub-bundles" of $\V$ while in \cite{Ha} $\PP(\V)$ parametrizes one-dimensional "quotient-bundles" of $\V$.
\end{remark}

\section{A geometric example} \label{section 5}

In this section, we give a geometric example that satisfies  Definition \ref{definition 2} of categorical $\dot{\Uu}_{0,N}(L\SL_n)$ action. That means we have to define the categories $\Kk(\kk)$ and $1$-morphisms ${\E}_{i,r}\bo_{\kk}$, ${\F}_{i,s}\bo_{\kk}$, ${\Hhh}_{i,\pm 1}\bo_{\kk}$, $({\spi}_{i}^{\pm})^{\pm 1}\bo_{\kk}$ so that they satisfy the conditions.

\subsection{An overview} \label{5.0}

Since we use the language of FM transformas, beside the spaces we use to define the categories $\Kk(\kk)$, we need to introduce more geometric spaces (varieties) that serve as the roles of correspondences to define the kernels for those $1$-morphisms. However, giving the definitions of those varieties directly in $\SL_{n}$ case without explanations may make some readers feel a bit technical at first glance and can not see the reasons behind them.

Thus instead of moving to the general definitions directly, we decide to give the readers an overview of the main ideas behind the definitions and proofs first.

Because many relations in the Definition \ref{definition 2} can be reduced to the $\SL_2$ case,  we mainly focus on the $\SL_{2}$ case here. In this case, the categories are the bounded derived category of coherent sheaves on Grassmannians $\Dd^b(\GG(k,N))$. And from the introduction, the correspondence we use to define the $1$-morphisms ${\E}_{r}\bo_{(k,N-k)}$ is the $3$-step partial flag variety $Fl(k-1,k)$ in diagram (\ref{diagram9}), similar for ${\F}_{s}\bo_{(k,N-k)}$.

To verify those conditions in Definition \ref{definition 2}, we have to calculate many convolutions of FM kernels. However, most of the conditions are easy to verify by using standard tools like base change, projection formula and Proposition \ref{proposition 3}. The most interesting one is the commutator relations, i.e. condition (16), and the nontrivial conditions are (13)(a), (14)(a).

First, we discuss condition (16) and we have to compare two convolution of FM kernels. The final step of convolution is pushforward induced from the projection $\pi_{13}$ to the product of the first and third components. In the constructible/coherent setting, to calculate the convolutions of kerenls for ${\E}{\F}$/${\E}_{r}{\F}_{s}$ and ${\F}{\E}$/${\F}_{s}{\E}_{r}$, we need to know what are the geometric spaces that the projection $\pi_{13}$ really restrict to. 

They are the following varieties 
\begin{align*}
Z&=\{(V,V',V'') \in \GG(k,N) \times \GG(k+1,N) \times \GG(k,N) \ | \ V, \ V'' \overset{1}{\subset} V' \}, \\
Z'&=\{(V,V''',V'') \in \GG(k,N) \times \GG(k-1,N) \times \GG(k,N) \ | \ V''' \overset{1}{\subset} V, \ V''\}
\end{align*} that naturally arise when calculate convolution of kerenls for ${\E}{\F}$/${\E}_{r}{\F}_{s}$ and ${\F}{\E}$/${\F}_{s}{\E}_{r}$ respectively (see Lemma \ref{lemma 5} for details). Restricting the projections to $Z$ and $Z'$, in order to distinguish them we use $\pi'_{13}$ for $Z'$. Then we obtain the following diagram  
\begin{equation} \label{diagram x}
    \xymatrix@R=0.5cm@C=-80pt{
    Z' \ar[rd]^{\pi'_{13}}  & & Z  \ar[ld]_{\pi_{13}} \\
    & Y=\{(V,V'') \in \GG(k,N) \times \GG(k,N) \ | \ \dim(V \cap V'') \geq k-1 \}.
    }
\end{equation} 

\begin{remark}
$Y$ is singular with 2 resolutions $Z$ and $Z'$.
\end{remark}

In the constructible setting, assume that $\lambda=N-2k \geq 0$, then in the diagram (\ref{diagram x}) the map $\pi'_{13}$ is a small resolution and $\pi_{13}$ is not a small resolution. Using the theory about perverse sheaves (or IC sheaves), we obtain the isomorphism ${\E\F}|_{\Cc(\lambda)} \cong {\F\E}|_{\Cc(\lambda)} \bigoplus Id_{\Cc(\lambda)}^{\oplus \lambda}$ \footnote{Here note that this is an isomorphism between kernels, we do not want to introduce more notations to make the readers getting confuse.}. The extra term $Id_{\Cc(\lambda)}^{\oplus \lambda}$ is contributed from $\pi_{13}$ since it is not small. Moreover, the extra term has a cohomological interpretation, which be thought as $Id_{\Cc(\lambda)} \otimes H_{sing}^{*}(\PP^{\lambda-1},\CC)$, see \cite{CKL2}.

However, in the coherent setting, we do not have powerful tools like the decomposition theorem. We use the fibered product 
$X:=Z' \times_{Y} Z$ which is defined as follows
\begin{equation} \label{x}
  X \coloneqq \{(V''',V,V'',V')  \in \GG(k-1,N) \times \GG(k,N) \times \GG(k,N) \times \GG(k+1,N) \ |  \ V''' \subset V \subset V', \ V''' \subset V''\subset V' \} 
\end{equation} and denote $p:X \rightarrow Y$ is the natural projection. Moreover, there is a divisor $D \subset X$ which is the locus where $V=V''$ and cutted out by the natural sections $\V''/\V''' \rightarrow \V'/\V$ where $\V''', \ \V, \ \V'', \ \V'$ are the natural tautological bundles on $X$. Thus we have the following short exact sequence  
\begin{equation} \label{ses''}
0 \rightarrow \V''/\V''' \rightarrow \V'/\V \rightarrow \Oo_{D} \otimes \V'/\V \rightarrow 0.
\end{equation}

To compare the kernels for ${\E}_{r}{\F}_{s}$ and ${\F}_{s}{\E}_{r}$, instead of directly pushing forward to $Y$, we pullback them to $X$ and using the short exact sequence (\ref{ses''}) together with Proposition \ref{proposition 3} to help us prove the result (see the proof of Proposition \ref{proposition 5}).

The kernels for $\spi^{\pm}\bo_{(k,N-k)}$, $\Hhh_{\pm 1}\bo_{(k,N-k)}$ are produced under the above comparison, and they are reflected by the non-vanishing of coherent sheaf cohomology $H^{*}(\PP^{N-1}, \Oo_{\PP^{N-1}}(-r-s-k)) \neq 0$. Also, the kernels for $\spi^{\pm}\bo_{(k,N-k)}$ are just line bundles up to homological shifts while the kernels for $\Hhh_{\pm 1}\bo_{(k,N-k)}$ are far more difficult to describe (see Definition \ref{definition 3} or proof of Proposition \ref{proposition 5}).

Let $\Hh_{1}\bo_{(k,N-k)}$ denote the kernel for $\Hhh_{ 1}\bo_{(k,N-k)}$. Then the key property we know is that it fits to the following exact triangle in $\Dd^b(\GG(k,N) \times \GG(k,N))$ 
\begin{equation} \label{ses'''}
    \Delta_{*}\V \rightarrow  \Hh_{1}\bo_{(k,N-k)} \rightarrow \Delta_{*}\CC^N/\V
\end{equation} where $\Delta$ is the diagonal map. Moreover, the extension class in $\Ext^{1}$ determines $\Hh_{1}\bo_{(k,N-k)}$ such that it is neither isomorphic to $\Delta_{*}\CC^N$ nor to $\Delta_{*}(\V\oplus\CC^N/\V)$ (see Theorem \ref{theorem'}), similarly for $\Hh_{-1}\bo_{(k,N-k)}$. This pretty much suggests that $\Hh_{\pm 1}\bo_{(k,N-k)}$ are supported on a neighborhood of the diagonal in $\GG(k,N) \times \GG(k,N)$.

Now we discuss condition (13)(a), (14)(a). Similarly, we have to compare the kernels for ${\Hhh}_{1}{\E}_{r}\bo_{(k,N-k)}$ and ${\E}_{r}{\Hhh}_{1}\bo_{(k,N-k)}$. The main idea is to keep track the extension class for the exact triangle (\ref{ses'''}) under the process of convolutions. In order to induce a morphsim from ${\Hhh}_{1}{\E}_{r}\bo_{(k,N-k)}$ to ${\E}_{r}{\Hhh}_{1}\bo_{(k,N-k)}$, we have to verify commutativity of a square that involve the two elements in the extension class for the resulting exact triangles after convolutions. Finally, by taking cone, we show that the cone is the kernel for ${\E}_{r+1}\oplus{\E}_{r+1}[1]$ (see proof of Theorem \ref{theorem 3}).

\subsection{Correspondences and the main result.} \label{5.1}
In this subsection, first we define those geometric spaces (varieties) that will be used to define the categories $\Kk(\kk)$ and FM kerenls for the 1-morphisms ${\E}_{i,r}\bo_{\kk}$, ${\F}_{i,s}\bo_{\kk}$, ${\Hhh}_{i,\pm 1}\bo_{\kk}$, $({\spi}_{i}^{\pm})^{\pm 1}\bo_{\kk}$. They can be viewed as a $\SL_{n}$ generalization of the varieties mention in Section \ref{5.0}.

For each $\kk \in C(n,N)$, we define the $n$-step partial flag variety 
\begin{equation} \label{eq fl}
Fl_{\kk}(\CC^N):=\{V_{\bullet}=(0=V_{0} \subset V_1 \subset ... \subset V_{n}=\CC^N) \ | \ \tdim V_{i}/V_{i-1}=k_{i} \ \text{for} \ \text{all} \ i\}.
\end{equation}

We denote $Y(\kk)=Fl_{\kk}(\CC^N)$ and $\Dd^b(Y(\kk))$ to be the bounded derived categories of coherent sheaves on $Y(\kk)$. On $Y(\kk)$ we denote $\V_{i}$ to be the tautological bundle whose fibre over a point $(0=V_{0} \subset V_1 \subset ... \subset V_{n}=\CC^N)$ is $V_{i}$. We define the following correspondence, which can be though as $\SL_{n}$ generalization of $Fl(k-1,k)$.

\begin{definition}
We define $W^{1}_{i}(\kk)$ to be the subvariety of $Y(\kk) \times Y(\kk+\alpha_{i})$
\begin{equation*}
W^{1}_{i}(\kk):=\{(V_{\bullet},V_{\bullet}')  \ | \ V_{j}=V'_{j} \ \Rm{for} \ j \neq i,  \ V'_{i} \subset V_{i}\} \subset   Y(\kk) \times Y(\kk+\alpha_{i})
\end{equation*} and similarly its transpose correspondence 
\begin{equation*}
^{\TTt}W^{1}_{i}(\kk):=\{(V_{\bullet},V_{\bullet}')  \ | \ V_{j}=V'_{j} \ \Rm{for} \ j \neq i,  \ V_{i} \subset V'_{i}\} \subset  Y(\kk+\alpha_{i}) \times Y(\kk)
\end{equation*} for all $\kk \in C(n,N)$ and $i$.  
\end{definition}

We denote $\iota(\kk):W^{1}_{i}(\kk) \hookrightarrow Y(\kk) \times Y(\kk+\alpha_{i})$ and $^{\TTt}\iota(\kk):^{\TTt}W^{1}_{i}(\kk) \hookrightarrow Y(\kk+\alpha_{i}) \times Y(\kk)$ to be the natural inclusions. We also have the natural line bundle $\V_{i}/\V'_{i}$ on $W^{1}_{i}(\kk)$ and similarly $\V'_{i}/\V_{i}$ on $^{\TTt}W^{1}_{i}(\kk)$, where $1 \leq i \leq n$,.

Next, we define the following varieties that can be thought as 
$\SL_{n}$ generalization of $X$ in (\ref{x}).

\begin{definition}
We define $X_{i}(\kk)$ to be the subvariety of $Y(\kk+\alpha_{i}) \times Y(\kk) \times Y(\kk) \times Y(\kk-\alpha_{i})$
\begin{equation*}
X_{i}(\kk) \coloneqq \{ (V_{\bullet}''',V_{\bullet},V_{\bullet}'',V_{\bullet}')  \ | \ 
 V'''_{i}\subset V_{i} \subset V'_{i}, \
V'''_{i}\subset V''_{i} \subset V'_{i}, \ V'''_{j}=V_{j}=V''_{j}=V'_{j} \ \forall \ j \neq i \} 
\end{equation*} and the divisor $D_{i}(\kk) \subset X_{i}(\kk)$ that defined by
\begin{equation*}
D_{i}(\kk):=\{ (V_{\bullet}''',V_{\bullet},V_{\bullet}'',V_{\bullet}')  \ | \ 
 V'''_{i}\subset V_{i}=V''_{i} \subset V'_{i}, \ V'''_{j}=V_{j}=V''_{j}=V'_{j} \ \forall \ j \neq i \}.
\end{equation*} 
\end{definition}

Note $D_{i}(\kk)$ is cut out by the natural section of the line bundles $\Hh om(\V_{i}''/\V_{i}''',\V_{i}'/\V_{i})$. More precisely, we have $\Oo_{X_{i}(\kk)}(-D_{i}(\kk)) \cong \V_{i}''/\V_{i}''' \otimes (\V_{i}'/\V_{i})^{-1}$ and the following short exact sequences
\begin{equation*}
0 \rightarrow \V_{i}''/\V_{i}''' \rightarrow \V_{i}'/\V_{i} \rightarrow \Oo_{D_{i}(\kk)} \otimes \V_{i}'/\V_{i} \rightarrow 0.
\end{equation*} We obtain similar results if we view $D_{i}(\kk)$ is cut out by the natural section of the line bundles $\Hh om(\V_{i}/\V_{i}''',\V_{i}'/\V''_{i})$.

Finally, we define the $\SL_{n}$ generalization of $Y$ in diagram (\ref{diagram x}).
\begin{definition}
\begin{equation*}
Y_{i}(\kk)=\{(V_{\bullet},V_{\bullet}'') \in  Y(\kk) \times Y(\kk) \ | \
\dim V_{i}  \cap V''_{i} \geq \sum_{l=1}^{i}k_{l}-1, \ V_{j}=V''_{j} \ \forall \ j \neq i \}.
\end{equation*}
\end{definition}

Let $p_{i}(\kk):X_{i}(\kk) \rightarrow Y_{i}(\kk)$ be the natural projection defined by forgetting $V_{\bullet}'''$ and $V_{\bullet}'$, $t_{i}(\kk):Y_{i}(\kk) \rightarrow Y(\kk) \times Y(\kk)$ be the inclusion and $\Delta(\kk):Y(\kk) \rightarrow Y(\kk) \times Y(\kk)$ to be the diagonal map.

Then we define the 1-morphisms via using FM transforms with kernels involve the geometric spaces we introduce above.

\begin{definition} \label{definition 3}
	We define  $\bo_{\kk}$, ${\E}_{i,r}\bo_{\kk}$, $\bo_{\kk}{\F}_{i,s}$, $({\spi}^{\pm}_{i})^{\pm 1}\bo_{\kk}$, ${\Hhh}_{i,\pm 1}\bo_{\kk}$ to be FM transforms with the corresponding kernels
	\begin{equation*}
	\begin{split}
	\bo_{\kk}&:= \Delta(\kk)_{*}\Oo_{Y(\kk)} \in \Dd^b(Y(\kk)\times Y(\kk)), \\
	\Ee_{i,r}\bo_{\kk}&:=\iota(\kk)_{*} (\V_{i}/\V_{i}')^{r} \in \Dd^b(Y(\kk)\times Y(\kk+\alpha_{i})), \\
	\bo_{\kk}\Ff_{i,r}&:=^{\TTt}\iota(\kk)_{*} (\V'_{i}/\V_{i})^{r} \in \Dd^b(Y(\kk+\alpha_{i})\times Y(\kk)), \\
	(\Psi^{+}_{i})^{\pm 1}\bo_{\kk}&:=\Delta(\kk)_{*}\tdet (\V_{i+1}/\V_{i})^{\pm 1}[\pm(1-k_{i+1})] \in \Dd^b(Y(\kk)\times Y(\kk)), \\
	(\Psi^{-}_{i})^{\pm 1}\bo_{\kk}&:=\Delta(\kk)_{*} \tdet (\V_{i}/\V_{i-1})^{\mp 1}[\pm(1-k_{i})] \in \Dd^b(Y(\kk)\times Y(\kk)), \\
	\Hh_{i,1}\bo_{\kk}&:=(\Psi_{i}^{+}\bo_{\kk})^{-1} \ast [t_{i}(\kk)_{*}p_{i}(\kk)_{*}(\Oo_{2D_{i}(\kk)} \otimes (\V_{i}'/\V_{i})^{k_{i+1}+1})] \in \Dd^b(Y(\kk) \times Y(\kk)), \\
	\Hh_{i,-1}\bo_{\kk}&:=(\Psi_{i}^{-}\bo_{\kk})^{-1} \ast [t_{i}(\kk)_{*}p_{i}(\kk)_{*}(\Oo_{2D_{i}(\kk)} \otimes (\V_{i}/\V_{i}''')^{-k_{i}-1})] \in \Dd^b(Y(\kk) \times Y(\kk)),
	\end{split}
	\end{equation*} respectively.
\end{definition}

Now we can state the main result of this article, which is the following theorem.

\begin{theorem} \label{theorem 1} 
	Let $\Kk$ be the triangulated 2-categories whose nonzero objects are $\Kk(\kk)=\Dd^b(Y(\kk))$ where $\kk \in C(n,N)$, the 1-morphisms are kernels defined in Definition \ref{definition 3} and the 2-morphisms are maps between kernels.  Then this gives a categorical $\dot{\Uu}_{0,N}(L\SL_n)$ action.
\end{theorem}

We devote the rest of this section to a proof of this theorem.

\subsection{The $\SL_2$ case of the main result}

In this subsection, we prove there is a categorical $\dot{\Uu}_{0,N}(L\SL_2)$ action first, which is the following theorem. 

\begin{theorem} \label{theorem 2} 
	The data above define a categorical $\dot{\Uu}_{0,N}(L\SL_2)$ action on $\bigoplus_{k}\Dd^b(\GG(k,N))$.
\end{theorem}

Now $n=2$, to simplify the notations further, we drop $i$ from all the functors with $i$ in their notation, and thus the 1-morphisms are $\bo_{(k,N-k)}$, ${\E}_{r}\bo_{(k,N-k)}$, ${\F}_{s}\bo_{(k,N-k)}$, ${\Hhh}_{\pm 1}\bo_{(k,N-k)}$, $({\spi}^{\pm})^{\pm 1}\bo_{(k,N-k)}$. Furthermore, for the weight $\kk$ we will just write $(k,N-k)$ with $0 \leq k \leq N$.

To prove this, we need to prove the conditions (1), (4), (5) (9a), (10a), (11a), (12a), (13a), (14a), (16) and (6), (7), (8) with $i=j$ in Definition \ref{definition 2}.

Note that condition (1) is obvious, since the varieties are finite Grassmannians $\GG(k,N)$, which are smooth and proper, the Hom spaces are finite dimensional. Before we check the rest, let us remark that since all functors above are defined by using kernels, we check the relations in terms of kernels by Proposition \ref{Proposition 2}.

The proofs we give below for all the conditions are only for the functors ${\E}_{r}\bo_{(k,N-k)}$ (since the proof for conditions that involve ${\F}_{s}\bo_{(k,N-k)}$ is similar). Also, we will only prove the first relation in conditions that involves many isomorphisms or exact triangles, e.g. we will only prove ${\spi}^{+}{\E}_{r}\bo_{(k,N-k)} \cong {\E}_{r+1}{\spi}^{+}\bo_{(k,N-k)}$ for condition (11)(a). Finally, in order to make calculations to be simple, we will omit $(k,N-k)$ for all the maps in Definition \ref{definition 3}, i.e. we will just write $\iota$, $^{\TTt}\iota$, $\Delta$, $t$, $p$ if there is no confusion.

It is helpful to keep the following picture
\[
\xymatrix{
	&& \GG(k,N) \times \GG(k-1,N) \ar[ldd]_{\pi_1} \ar[rdd]^{\pi_2}\\
	&& Fl(k-1,k)  \ar[ld]_{p_1} \ar[rd]^{p_2}  \ar[u]^{\iota}  \\
	& \GG(k,N)  
	&& \GG(k-1,N)  }
\] where $\iota:Fl(k-1,k) \rightarrow \GG(k,N) \times \GG(k-1,N)$ is the natural inclusion and $\pi_{1}$, $\pi_{2}$, $p_{1}$, $p_{2}$ are the natural projections.

The first is condition (5).

\begin{lemma} \label{lemma 2} (Condition (5)). ${\E}_{r}\bo_{(k,N-k)}$ and ${\F}_{s}\bo_{(k,N-k)}$ are left and right adjoints up to homological shifts and twists of $\spi^{\pm}\bo_{(k,N-k)}$:
	\begin{equation*}
	({\Ee}_{r}\bo_{(k,N-k)})_{R} \cong \bo_{(k,N-k)}{({\Psi}^{+})^{r+1}} \ast {\Ff}_{N-k+2} \ast ({\Psi}^{+})^{-r-2}[-r-1],
	\end{equation*}
\end{lemma}

\begin{proof}
	By Proposition \ref{Proposition 1}, the right adjoint of $\Ee_{r}\bo_{(k,N-k)}$ is given by
	\[
	\{\iota_{*}  (\V/\V')^{r}\}^{\vee} \otimes \pi_{1}^*(\omega_{\GG(k,N)})[\tdim\GG(k,N)].
	\] where
	\[
	\{\iota_{*}  (\V/\V')^{r}\}^{\vee} \cong \iota_{*} ((\V/\V')^{-r} \otimes \omega_{Fl(k-1,k)}) \otimes \omega_{\GG(k,N)\times \GG(k-1,N)}^{-1}[-\tcodim Fl(k-1,k)].
	\]
	
	To calculate $\omega_{Fl(k-1,k)}$, we have $\omega_{Fl(k-1,k)} \cong \omega_{p_{2}} \otimes p^*_{2}\omega_{\GG(k-1,N)}$ where $\omega_{p_{2}}$ is the relative canonical bundle. Since the relative cotangent bundle is $\Tt^{\vee}_{p_{2}}\cong \V/\V' \otimes (\CC^N/\V)^{\vee}$, we obtain $\omega_{p_{2}} \cong (\V/\V')^{N-k} \otimes \det(\CC^N/\V)^{-1}$. A calculation gives $\tdim\GG(k,N)-\tcodim Fl(k-1,k)=N-k$.
	
	So summarizing above and use $\V/\V' \cong \det(\CC^N/\V') \otimes \det(\CC^N/\V)^{-1}$ we have 
	\begin{align*} 
	&\{\iota_{*}  (\V/\V')^{r}\}^{\vee} \otimes \pi_{1}^*(\omega_{\GG(k,N)})[\tdim\GG(k,N)] 
	\cong \iota_{*}(\omega_{rel}\otimes (\V/\V')^{-r} ) [N-k] \ \rm{(using \ projection \ formula)} \\
	& \cong \iota_{*} ((\V/\V')^{-r+N-k+1} \otimes \det(\CC^N/\V')^{-1})[N-k]. \\
	& \cong \iota_{*} ((\V/\V')^{N-k+2} \otimes \det(\CC^N/\V)^{r+1} \otimes \det(\CC^N/\V')^{-r-2})[N-k] \\
	& \cong \iota_{*} ((\V/\V')^{N-k+2} \otimes \det(\CC^N/\V)^{r+1}[(r+1)(1+k-N)] \otimes \det(\CC^N/\V')^{-r-2}[(r+2)(N-k)])[-r-1] 
	\end{align*}
	
	Note that the kernel $(\Psi^{+})^{-1}\bo_{(k-1,N-k+1)}$ is defined by  $\Delta_{*}\det(\CC^N/\V')^{-1}[N-k]$, while $\Psi^{+}\bo_{(k,N-k)}$ is defined by  $\Delta_{*}\det(\CC^N/\V)[1+k-N]$. We conclude that it is isomorphic to $\bo_{(k,N-k)}{({\Psi}^{+})^{r+1}}\ast {\Ff}_{N-k+2}\ast ({\Psi}^{+})^{-r-2}[-r-1]$.
\end{proof}

We decide to skip the proofs for conditions (9a), (10a), (11a), (12a) for the reason that the argument is fairly standard by using base change, projection formula, and Proposition \ref{proposition 3}. The readers can consult the proof of Lemma \ref{lemma 5} for a rough argument.

\begin{lemma} \label{lemma 4} (Conditions (9a), (10a)).
	\begin{equation*}
	(\Ee_{r+1}\ast \Ee_{s})\bo_{(k,N-k)} \cong 
	(\Ee_{s+1}\ast \Ee_{r})\bo_{(k,N-k)}[-1] \ \text{if} \ r-s \geq 1 
    \end{equation*}
\end{lemma}

\begin{lemma} \label{lemma 3} (Conditions (11a), (12a)). 
\begin{equation*}
	(\Psi^{+} \ast \Ee_{r})\bo_{(k,N-k)} \cong (\Ee_{r+1}\ast \Psi^{+})\bo_{(k,N-k)}[-1]
	\end{equation*}
\end{lemma}

Then we prove condition (16). Observe that from Lemma \ref{lemma 3}, we have $(\Psi^{+} \ast \Ee_{r})\bo_{(k,N-k)} \cong (\Ee_{r+1}\ast \Psi^{+})\bo_{(k,N-k)}[-1]$. Since $\Psi^{+}\bo_{(k,N-k)}$ is invertible, we get 
\begin{equation*} 
[\Psi^{+}\ast \Ee_{r}\ast (\Psi^{+})^{-1} ]\bo_{(k,N-k)} \cong \Ee_{r+1}\bo_{(k,N-k)}[-1],
\end{equation*} and apply this inductively we obtain
\begin{equation} \label{eq a1}
[(\Psi^{+})^r \ast \Ee_{0} \ast (\Psi^{+})^{-r}] \bo_{(k,N-k)} \cong \Ee_{r}\bo_{(k,N-k)}[-r]
\end{equation} where $(\Psi^{+})^r$ means $(\Psi^{+})$ convolution with itself  $r$ times. 

Similarly, for $\Ff_{s}\bo_{(k,N-k)}$, we have 
\begin{equation} \label{eq a2}
[(\Psi^{+})^{-s} \ast \Ff_{0} \ast (\Psi^{+})^{s}] \bo_{(k,N-k)} \cong \Ff_{s}\bo_{(k,N-k)}[-s].
\end{equation}

By (\ref{eq a1}) and (\ref{eq a2}), we obtain that 
\begin{align*}
& (\Ee_{r}\ast \Ff_{s})\bo_{(k,N-k)} \cong  [(\Psi^{+})^r \ast \Ee_{0} \ast \Ff_{s+r} \ast (\Psi^{+})^{-r} ]\bo_{(k,N-k)}, \\
& (\Ff_{s}\ast \Ee_{r})\bo_{(k,N-k)}
\cong  [(\Psi^{+})^r \ast \Ff_{r+s} \ast \Ee_{0} \ast (\Psi^{+})^{-r} ]\bo_{(k,N-k)}.
\end{align*}

Since $\Psi^{+}\bo_{(k,N-k)}$ is invertible, in order to compare $( \Ff_{s}\ast \Ee_{r})\bo_{(k,N-k)}$ and $(\Ee_{r}\ast \Ff_{s})\bo_{(k,N-k)}$, it suffices to compare $(\Ff_{r+s} \ast \Ee_{0})\bo_{(k,N-k)}$ and $ (\Ee_{0} \ast \Ff_{r+s})\bo_{(k,N-k)}$. Hence, we prove the following proposition.

\begin{proposition} \label{proposition 5}
	We have the following (non-split) exact triangles in $\Dd^b(\GG(k,N) \times \GG(k,N))$.
	\begin{align*}
	&(\Ff_{N-k+1}\ast \Ee_{0})\bo_{(k,N-k)} \rightarrow (\Ee_{0}\ast \Ff_{N-k+1})\bo_{(k,N-k)} \rightarrow (\Psi^{+} \ast \Hh_{1})\bo_{(k,N-k)}, \\
	&(\Ff_{N-k}\ast \Ee_{0})\bo_{(k,N-k)} \rightarrow (\Ee_{0}\ast \Ff_{N-k})\bo_{(k,N-k)} \rightarrow \Psi^{+}\bo_{(k,N-k)}.
	\end{align*}
Finally, we have 
\begin{equation*}
(\Ff_{s}\ast \Ee_{0})\bo_{(k,N-k)} \cong (\Ee_{0} \ast \Ff_{s})\bo_{(k,N-k)} \ \text{if} \ 0 \leq s \leq N-k-1.
\end{equation*}
\end{proposition}

The first thing we deal with is when $s=0$, which is the following lemma.

\begin{lemma} \label{lemma 5}
	$(\Ee_{0} \ast \Ff_{0})\bo_{(k,N-k)} \cong (\Ff_{0} \ast \Ee_{0})\bo_{(k,N-k)}$ 
\end{lemma}

\begin{proof}
	By definition
	\begin{equation}
	(\Ee_{0} \ast \Ff_{0})\bo_{(k,N-k)} \cong
	\pi_{13*}(\pi^*_{12} {^{\TTt}}\iota_{*}\Oo_{Fl(k,k+1)} \otimes \pi_{23}^*\iota_{*}\Oo_{Fl(k,k+1)}). \label{5.2.14}
	\end{equation} 
	
	We will keep using fibred product diagram, base change and projection formula to calculate (\ref{5.2.14}). Since the argument is pretty standard, we decide to omit the details and only mention the key steps in order to make the article short.

	We have the following fibred product diagrams 
	\begin{equation*} 
	\begin{tikzcd} 
	Fl(k,k+1) \times \GG(k,N)   \arrow[r, "^{\TTt}\iota \times id"]  \arrow[d, "a_{1}"] &\GG(k,N)\times \GG(k+1,N) \times \GG(k,N) \arrow[d, "\pi_{12}"]\\
	Fl(k,k+1)  \arrow[r, "^{\TTt}\iota"] & \GG(k,N)\times \GG(k+1,N)
	\end{tikzcd}
	\end{equation*}
	\begin{equation*} 
	\begin{tikzcd} 
	\GG(k,N) \times Fl(k,k+1)   \arrow[r, "id \times \iota"]  \arrow[d, "a_{2}"] &\GG(k,N)\times \GG(k+1,N) \times \GG(k,N) \arrow[d, "\pi_{23}"]\\
	Fl(k,k+1)  \arrow[r, "\iota"] & \GG(k+1,N)\times \GG(k,N)
	\end{tikzcd}
	\end{equation*} where $a_{1}$, $a_{2}$ are the natural projections. Then we obtain  
	\begin{equation}
	(\ref{5.2.14}) 
	\cong \pi_{13*}((^{\TTt}\iota \times id)_{*}\Oo_{Fl(k,k+1) \times \GG(k,N)} \otimes (id\times \iota)_{*}\Oo_{\GG(k,N) \times Fl(k,k+1)}). \label{5.2.16}
	\end{equation}
	
	Next, the following fibred product diagram 
	\begin{equation*} 
	\begin{tikzcd} 
	Z   \arrow[r, "b_{1}"]  \arrow[d, "b_{2}"] &Fl(k,k+1) \times \GG(k,N) \arrow[d, "^{\TTt}\iota \times id"]\\
	\GG(k,N) \times Fl(k,k+1)  \arrow[r, "id \times \iota"] & \GG(k,N) \times \GG(k+1,N)\times \GG(k,N)
	\end{tikzcd}
	\end{equation*} where $Z$ is the following variety
	\begin{equation*}
	Z \coloneqq \{ (V,V',V'') \ | \ \dim V=\dim V''=k,\ \dim V'=k+1,\ V \subset V',\ V'' \subset V' \}
	\end{equation*} gives us that
	\begin{equation} \label{eq cdd}
	(\ref{5.2.16}) \cong  \pi_{13*}(id\times \iota)_{*}b_{2*}(\Oo_{Z}).
	\end{equation}
	
	Finally, we have the following fibred product diagram
	\begin{equation*} 
	\begin{tikzcd} 
	Z   \arrow[r, "j_{1}"]  \arrow[d, "\pi_{13}|_{Z}"] &\GG(k,N) \times \GG(k+1,N) \times \GG(k,N) \arrow[d, "\pi_{13}"]\\
	Y  \arrow[r, "t"] & \GG(k,N) \times  \GG(k,N)
	\end{tikzcd}
	\end{equation*} where $Y=\pi_{13}(Z)=\{(V,V'')\ |\ \tdim V\cap V'' \geq k-1 \}$, and $j_1=(id \times \iota) \circ b_2$, $t$ are the inclusions. 
	
	Note that $Y \subset \GG(k,N) \times \GG(k,N)$ is a Schubert variety, it has rational singularities. So $(\pi_{13}|_{Z*})(\Oo_{Z}) \cong \Oo_{Y}$ and thus
	\begin{equation*}
	(\Ee_{0} \ast \Ff_{0})\bo_{(k,N-k)} \cong (\ref{eq cdd}) \cong t_{*}(\pi_{13}|_{Z*})(\Oo_{Z}) \cong t_{*}\Oo_{Y}.
	\end{equation*} 
	
	Using the same method for calculating $(\Ff_{0} \ast \Ee_{0})\bo_{(k,N-k)}$, we end up with the following fibred product diagram
	\begin{equation*} 
	\begin{tikzcd} 
	Z'   \arrow[r, "j_{2}"]  \arrow[d, "\pi_{13'}|_{Z'}"] &\GG(k,N) \times \GG(k-1,N) \times \GG(k,N) \arrow[d, "\pi_{13'}"]\\
	Y  \arrow[r, "t"] & \GG(k,N) \times  \GG(k,N)
	\end{tikzcd}
	\end{equation*}  where 
	\begin{equation*}
	Z'\coloneqq \{(V,V''',V'') \ | \ \dim V=\dim V''=k, \ \dim V'''=k-1, \  V''' \subset V,\ V''' \subset  V''\}
	\end{equation*} and $j_{2}$ is the inclusion.

	Again, we have $\pi_{13}(Z')=\{(V,V'') \ | \ \dim V\cap V'' \geq k-1\}=Y$, which is the same as we get when we calculate $(\Ee_{0} \ast \Ff_{0})\bo_{(k,N-k)}$. Thus 
	\begin{equation*}
    (\Ff_{0} \ast \Ee_{0})\bo_{(k,N-k)} \cong	\pi_{13'*}j_{2*}(\Oo_{Z'}) \cong t_{*}(\pi_{13'}|_{Z'*})(\Oo_{Z'}) \cong t_{*}\Oo_{Y}
	\end{equation*} which prove the lemma.
\end{proof}

Now we move to the case where $s$ is nonzero. To compare $(\Ee_{0} \ast \Ff_{s})\bo_{(k,N-k)}$ and $(\Ff_{s} \ast \Ee_{0})\bo_{(k,N-k)}$, by using the above fibred product diagrams, we have to compare the following two objects
\[
\pi_{13*}(j_{1*} (\V'/\V)^{s} ) \cong t_{*}(\pi_{13}|_{Z*})(\V'/\V)^{s} \
\text{and} \ 
\pi_{13'*}(j_{2*} (\V''/\V''')^{s} ) \cong t_{*}(\pi_{13'}|_{Z'*})(\V''/\V''')^{s} 
\] in the derived category $\Dd^b(\GG(k,N)\times \GG(k,N))$.

Note that both are pushforwards to $Y$. In order to handle the case where we tensor non-trivial line bundles, instead of directly pushing forward to $Y$, we lift the line bundles to a much larger space, i.e. their fibred product. The fibred product space $X:=Z \times_{Y} Z'$, is  given by
\begin{equation*}
X=\{(V''',V,V'',V')\ | \ \tdim(V \cap V'') \geq k-1, \ \ V''' \subset V \subset V', \ V''' \subset V''\subset V' \}.
\end{equation*}

We have the following fibred product diagram
\begin{equation} \label{diagram1}
\begin{tikzcd} [column sep=large]
X=Z \times_{Y} Z' \arrow[r, "g_{1}"]  \arrow[d, "g_{2}"] & Z\arrow[d, "\pi_{13}|_{Z}"]\\
Z' \arrow[r, "\pi_{13'}|_{Z'}"] &Y
\end{tikzcd}
\end{equation} where $g_1$ and $g_2$ are the natural projections. We denote $p:X \rightarrow Y$ to be the natural projection.

On $X$, we have the divisor $D = Fl(k-1,k,k+1)=\{(V''',V,V') \ | \ V''' \subset V \subset V'\}$ which is the locus where $V=V''$ and it is the vanishing locus of the natural sections $\V''/\V''' \rightarrow \V'/\V$. Thus we have the following short exact sequence
\begin{equation*}
0 \rightarrow \V''/\V''' \rightarrow \V'/\V \rightarrow \Oo_{D} \otimes \V'/\V \rightarrow 0.
\end{equation*}

We can relate it to Proposition \ref{proposition 3}, more precisely, it is easy to see that the restriction of the line bundle $\V'/\V$ to $D$ is the pullback of the relative tautological bundle $\Oo_{\PP(\CC^N/\V)}(-1)$ on $\PP(\CC^N/\V)=Fl(k,k+1) \subset Z$. Thus we obtain
\begin{equation} \label{ses1}
0 \rightarrow \V''/\V''' \rightarrow \V'/\V \rightarrow \Oo_{D} \otimes \Oo_{\PP(\CC^N/\V)}(-1)  \rightarrow 0,
\end{equation}

We will use the above short exact sequence to help us prove Proposition \ref{proposition 5}.

\begin{proof}[Proof of Proposition \ref{proposition 5}]
	
	We already settled the case $s=0$. For $1 \leq s \leq N-k+1$ , the main thing we have to do is compare/relate $(\Ee_{0} \ast \Ff_{s})\bo_{(k,N-k)}$ and $(\Ff_{s} \ast \Ee_{0})\bo_{(k,N-k)}$. Note that $(\pi_{13}|_{Z*})(\Oo_{Z}) \cong \Oo_{Y}$. This implies that $(\pi_{13}|_{Z*})(\pi_{13}|_{Z})^{*} \cong id_{Y}$. Using the fibred product Diagram (\ref{diagram1}), since $g_2$ is the base change of $\pi_{13}|_{Z}$, we have $g_{2*}g_{2}^{*} \cong id_{Z'}$. Similarly, $(\pi_{13'}|_{Z'*})(\Oo_{Z'}) \cong \Oo_{Y}$ implies that $g_{1*}g_{1}^{*} \cong id_{Z}$.
	 
	So from the above discussion
	\begin{equation*}
	(\Ee_{0} \ast \Ff_{s})\bo_{(k,N-k)} \cong t_{*}(\pi_{13}|_{Z*})(\V'/\V)^{s} \cong t_{*}(\pi_{13}|_{Z*})g_{1*}g_{1}^{*}(\V'/\V)^{s} \cong t_{*}p_{*}(\V'/\V)^{s},
	\end{equation*} and similarly
	\begin{equation*}
	(\Ff_{s} \ast \Ee_{0})\bo_{(k,N-k)} \cong
	t_{*}(\pi_{13'}|_{Z'*})(\V''/\V''')^{s} \cong t_{*}(\pi_{13'}|_{Z'*})g_{2*}g_{2}^{*} (\V''/\V''')^{s} \cong t_{*}p_{*}(\V''/\V''')^{s}.
	\end{equation*}
	
	Thus, at first, we have to compare $p_{*}(\V'/\V)^{s}$ and $p_{*} (\V''/\V''')^{s}$ in $\Dd^b(Y)$.
	
	Note that for each $n \geq 1$, we have the following short exact sequence on $X$
	\begin{equation} \label{ses8} 
	0 \rightarrow (\V''/\V''')^{n} \rightarrow (\V'/\V)^{n} \rightarrow \Oo_{nD} \otimes (\V'/\V)^{n} \rightarrow 0.
	\end{equation}
	
	Tensoring (\ref{ses1}) by $(\V''/\V''')^{n-1}$, we get 
	\[
	0 \rightarrow (\V''/\V''')^{n} \rightarrow \V'/\V \otimes (\V''/\V''')^{n-1}  \rightarrow \Oo_{D} \otimes \V'/\V \otimes (\V/\V''')^{n-1} \rightarrow 0 .
	\]
	
	Both of them are exact triangles in $\Dd^b(X)$, and they can be completed together to form the following diagram of morphisms between exact triangles	
	\[
	\xymatrix{
		(\V''/\V''')^{n} \ar[d]^{id} \ar[r] &\V'/\V \otimes (\V''/\V''')^{n-1}  \ar[d] \ar[r] &\Oo_{D} \otimes \V'/\V \otimes (\V/\V''')^{n-1} \ar[d] \\
		(\V''/\V''')^{n}  \ar[d] \ar[r] &(\V'/\V)^{n} \ar[d] \ar[r] &\Oo_{nD} \otimes (\V'/\V)^{n} \ar[d] \\
		0 \ar[r] &\Oo_{(n-1)D} \otimes  (\V'/\V)^{n} \ar[r] &\Oo_{(n-1)D} \otimes  (\V'/\V)^{n}
	}
	\]
		
	So we obtain the exact triangle 
	\begin{equation} \label{ses9}
	\Oo_{D} \otimes \V'/\V \otimes (\V/\V''')^{n-1} \rightarrow \Oo_{nD} \otimes (\V'/\V)^{n} \rightarrow \Oo_{(n-1)D} \otimes  (\V'/\V)^{n} 
	\end{equation} for all $n \geq 1$ (Here we take $\Oo_{(n-1)D}$ to be $0$ when $n=1$). 
	
	We will use the exact triangles (\ref{ses8}) and (\ref{ses9}) to prove our result.
	
	The first case is $1 \leq s \leq N-k-1$. For (\ref{ses8}) with $n=s$ we have 
	\[
	0 \rightarrow (\V''/\V''')^{s} \rightarrow (\V'/\V)^{s} \rightarrow \Oo_{sD} \otimes (\V'/\V)^{s} \rightarrow 0.
	\]
	
	Applying the derived pushforward $p_{*}$, we obtain 
	\[
	p_{*}(\V''/\V''')^{s} \rightarrow p_{*}(\V'/\V)^{s} \rightarrow p_{*}(\Oo_{sD} \otimes (\V'/\V)^{s}).
	\]
	
	It suffices to prove that $p_{*}(\Oo_{sD} \otimes (\V'/\V)^{s}) \cong 0$ so that we can obtain  $(\Ee_{0} \ast \Ff_{s})\bo_{(k,N-k)} \cong (\Ff_{s} \ast \Ee_{0})\bo_{(k,N-k)}$ after applying $t_{*}$.
	
	Using (\ref{ses9}) with $n=s$, we have 
	\begin{equation*}
	\Oo_{D} \otimes \V'/\V \otimes (\V/\V''')^{s-1} \rightarrow \Oo_{sD} \otimes (\V'/\V)^{s} \rightarrow \Oo_{(s-1)D} \otimes  (\V'/\V)^{s}.
	\end{equation*}
	
	Applying the derived pushforward $p_{*}$ to calculate $p_{*}(\Oo_{D} \otimes \V'/\V \otimes (\V/\V''')^{s-1} )$, we using the projection formula. 
	\begin{align*}
	p_{*}(\Oo_{D} \otimes \V'/\V \otimes (\V/\V''')^{s-1}) 
	& \cong p_{*}(\Oo_{D} \otimes \Oo_{\PP(\CC^N/\V)}(-1) \otimes \Oo_{\PP(\V^{\vee})}(s-1)) \\ 
	&\cong \pi_{13*}p_{1*}(\Oo_{D} \otimes p_{1}^*(\Oo_{\PP(\CC^N/\V)}(-1)) \otimes \Oo_{\PP(\V^{\vee})}(s-1)) \\
	&\cong \pi_{13*}(\Oo_{Fl(k,k+1)} \otimes \Oo_{\PP(\CC^N/\V)}(-1) \otimes \Sym^{s-1}(\V)) \\
	&\cong \pi_{13*}(\Oo_{Fl(k,k+1)} \otimes \Oo_{\PP(\CC^N/\V)}(-1) \otimes \pi_{13}^*(\Sym^{s-1}(\V))) \\
	&\cong  \Sym^{s-1}(\V) \otimes \pi_{13*}(\Oo_{\PP(\CC^N/\V)}(-1)) \cong 0.
	\end{align*}
	
	Thus we get $p_{*}(\Oo_{sD} \otimes (\V'/\V)^{s}) \cong p_{*}(\Oo_{(s-1)D} \otimes  (\V'/\V)^{s})$. Next, consider the exact triangle (\ref{ses9}) with $n=s-1$, we have 
	\begin{equation*}
	\Oo_{D} \otimes \V'/\V \otimes (\V/\V''')^{s-2} \rightarrow \Oo_{(s-1)D} \otimes (\V'/\V)^{s-1} \rightarrow \Oo_{(s-2)D} \otimes  (\V'/\V)^{s-1}.
	\end{equation*}
	
	Tensoring by $\V'/\V$ and then applying $p_{*}$, using the same argument as above, we obtain $p_{*}(\Oo_{(s-1)D} \otimes (\V'/\V)^{s}) \cong p_{*}(\Oo_{(s-2)D} \otimes  (\V'/\V)^{s})$. Continuing this procedure, we will end up with 
	\begin{equation} \label{eq iso} 
	p_{*}(\Oo_{sD} \otimes (\V'/\V)^{s}) \cong ... \cong p_{*}(\Oo_{2D} \otimes (\V'/\V)^{s}).
	\end{equation}
	
	For the exact triangle (\ref{ses9}) with $n=2$, tensoring by $(\V'/\V)^{s-2}$, we get 
	\[
	\Oo_{D} \otimes (\V'/\V)^{s-1} \otimes \V/\V''' \rightarrow \Oo_{2D} \otimes (\V'/\V)^{s} \rightarrow \Oo_{D} \otimes  (\V'/\V)^{s}.
	\] Applying $p_{*}$, then we get $p_{*}(\Oo_{D} \otimes (\V'/\V)^{s-1} \otimes \V/\V''') \cong 0$ and $p_{*}(\Oo_{D} \otimes  (\V'/\V)^{s}) \cong 0$. The first isomorphism is via using projection formula and Proposition \ref{proposition 3} with $1 \leq s \leq N-k-1$, while the second one is via Proposition \ref{proposition 3} with $1 \leq s \leq N-k-1$.
	
	Hence we get $p_{*}(\Oo_{2D} \otimes (\V'/\V)^{s}) \cong 0$ and (\ref{eq iso}) implies that $p_{*}(\Oo_{sD} \otimes (\V'/\V)^{s}) \cong 0$, so we prove the first case.
	
	The next case is $s=N-k$. Applying $p_{*}$ to (\ref{ses8}) with $n=N-k$, we have 
	\[
	p_{*} (\V''/\V''')^{N-k} \rightarrow p_{*}(\V'/\V)^{N-k} \rightarrow p_{*}(\Oo_{(N-k)D} \otimes (\V'/\V)^{N-k}). 
	\]
	
	It suffices to prove $p_{*}(\Oo_{(N-k)D} \otimes (\V'/\V)^{N-k}) \cong j_{*}\det(\CC^N/\V)[1+k-N]$, where $j:\Delta=\GG(k,N) \rightarrow Y$ is the natural inclusion. Note that we have the inclusion $t:Y \rightarrow \GG(k,N) \times \GG(k,N)$, and so $\Delta=t \circ j$.
	
	Again, using the similar argument as in the proof of the case $1 \leq s \leq N-k-1$, we also have 
	\[
	p_{*}(\Oo_{(N-k)D} \otimes (\V'/\V)^{N-k}) \cong ... \cong p_{*}(\Oo_{2D} \otimes (\V'/\V)^{N-k}).
	\]
	
	Considering (\ref{ses9}) with $n=2$ and tensoring it with $(\V'/\V)^{N-k-2}$, we get 
	\[
	\Oo_{D} \otimes (\V'/\V)^{N-k-1} \otimes \V/\V''' \rightarrow \Oo_{2D} \otimes (\V'/\V)^{N-k} \rightarrow \Oo_{D} \otimes  (\V'/\V)^{N-k},
	\] applying $p_{*}$, we obtain $p_{*}(\Oo_{D} \otimes (\V'/\V)^{N-k-1} \otimes \V/\V''') \cong 0$ via projection formula, while $p_{*}(\Oo_{D} \otimes  (\V'/\V)^{N-k}) \cong j_{*}\det(\CC^N/\V)[1+k-N]$ is by proposition \ref{proposition 3}.
	
	Hence we get 
	\begin{align*}
	 p_{*}(\Oo_{(N-k)D} \otimes (\V'/\V)^{N-k}) &\cong p_{*}(\Oo_{2D} \otimes (\V'/\V)^{N-k}) \\
	 &\cong p_{*}(\Oo_{D} \otimes  (\V'/\V)^{N-k}) \cong j_{*}\det(\CC^N/\V)[1+k-N]
	\end{align*} and $t_{*}p_{*}(\Oo_{(N-k)D} \otimes (\V'/\V)^{N-k}) \cong \Delta_{*}\det(\CC^N/\V)[1+k-N]$, which prove the second case.
	
	Finally, for the third case $s=N-k+1$, we have the following exact triangle in $\Dd^b(Y)$
	\begin{equation} \label{ses10}
	p_{*} (\V''/\V''')^{N-k+1} \rightarrow p_{*}(\V'/\V)^{N-k+1} \rightarrow p_{*}(\Oo_{(N-k+1)D} \otimes (\V'/\V)^{N-k+1})
	\end{equation} by applying $p_{*}$ to (\ref{ses8}) with $n=N-k+1$.
	
	Using the same argument, we still have the following 
	\[
	p_{*}(\Oo_{(N-k+1)D} \otimes (\V'/\V)^{N-k+1}) \cong ... \cong p_{*}(\Oo_{2D} \otimes (\V'/\V)^{N-k+1}).
	\]
	
	So after applying $t_{*}$ to (\ref{ses10}), we get the following exact triangle
	\[
	 (\Ff_{N-k+1} \ast \Ee_{0})\bo_{(k,N-k)} \rightarrow  (\Ee_{0} \ast \Ff_{N-k+1})\bo_{(k,N-k)} \rightarrow t_{*}p_{*}(\Oo_{2D} \otimes (\V'/\V)^{N-k+1})
	\] and $t_{*}p_{*}(\Oo_{2D} \otimes (\V'/\V)^{N-k+1})$ is exactly the kernel $(\Psi^{+} \ast \Hh_{1})\bo_{(k,N-k)}$.
	
	Hence we prove the third case and complete the proof of this proposition.
\end{proof}

\begin{remark} \label{nonsplit}
In this remark, we show that the exact triangles in proposition \ref{proposition 5} are non-split. We give a proof for the second exact triangle, and the other is similar. We prove this via showing that the Hom spaces of morphisms between objects are zero. Indeed, using adjunction (condition (5)) we have 
\begin{align*}
&\Hom( (\Ee_{0} \ast \Ff_{N-k})\bo_{(k,N-k)}, (\Ff_{N-k} \ast \Ee_{0})\bo_{(k,N-k)})  \\
&\cong \Hom( (\Ff_{N-k})_{L} \ast \Ee_{0}\bo_{(k-1,N-k+1)},   \Ee_{0} \ast (\Ff_{N-k})_{L}\bo_{(k-1,N-k+1)}) \\
& \cong \Hom( \Ee_{0} \ast (\Psi^{+})^{-1} \ast \Ee_{0}\bo_{(k-1,N-k+1)}[1],   \Ee_{0} \ast (\Psi^{+})^{-1} \Ee_{0} \bo_{(k-1,N-k+1)}) \\
& \cong \Hom( \Ee_{0} \ast  \Ee_{-1} \ast (\Psi^{+})^{-1} \bo_{(k-1,N-k+1)}[2],   \Ee_{0} \ast  \Ee_{-1} \ast (\Psi^{+})^{-1} \bo_{(k-1,N-k+1)}[1])  \\
& \cong \Hom(0,0) \cong 0. \ (\text{using condition 9(a)})
\end{align*} 
\end{remark}

Next, we want to understand more about the new kernel $\Hh_{1}$, which is produced under the process of comparing $(\Ff_{N-k+1} \ast \Ee_{0})\bo_{(k,N-k)}$ and $(\Ee_{0} \ast \Ff_{N-k+1})\bo_{(k,N-k)}$. Note that tensoring (\ref{ses9}) with $n=2$ by $(\V'/\V)^{N-k-1}$, we get 
\[
\Oo_{D} \otimes (\V'/\V)^{N-k} \otimes \V/\V''' \rightarrow \Oo_{2D} \otimes (\V'/\V)^{N-k+1} \rightarrow \Oo_{D} \otimes  (\V'/\V)^{N-k+1}.
\]
	
Applying $t_{*}p_{*}$, using projection formula and Proposition \ref{proposition 3}, we get the following exact triangle
\begin{equation} \label{ses11}
\Delta_{*} \V \otimes \det(\CC^N/\V)[1+k-N] \rightarrow t_{*}p_{*}(\Oo_{2D} \otimes (\V'/\V)^{N-k+1} ) \rightarrow  \Delta_{*} \CC^N/\V \otimes \det(\CC^N/\V)[1+k-N].
\end{equation}

Taking convolution of the exact triangle (\ref{ses11})  with $(\Psi^{+}\bo_{(k,N-k)})^{-1}$, we get the following exact triangle
\[
\Delta_{*} \V \rightarrow ((\Psi^{+}\bo_{(k,N-k)})^{-1} \ast t_{*}p_{*}(\Oo_{2D} \otimes (\V'/\V)^{N-k+1} )) \cong \Hh_{1}\bo_{(k,N-k)} \rightarrow  \Delta_{*} \CC^N/\V .
\]
	
This implies that $\Hh_{1}\bo_{(k,N-k)}$ is given by an element in 
\begin{equation*}
\Ext^1_{\GG(k,N)\times\GG(k,N)}(\Delta_{*}\CC^N/\V ,\Delta_{*}\V ) \cong \Ext^1_{\Delta}(\CC^N/\V,\V) \bigoplus \End(\Omega_{\Delta}).
\end{equation*}

Thus we have to know what element determine it. This will be proved in the following Theorem \ref{theorem'}. But Before we prove it, we need a lemma that will be used in the proof.

\begin{lemma} \label{lemma 7} 
	$\End(\Omega_{\GG(k,N)}) \cong \CC$.
\end{lemma}

\begin{proof}
	Instead of proving $\End(\Omega_{\GG(k,N)}) \cong \CC$, we prove that $\End(\Tt_{\GG(k,N)}) \cong \CC$, where $\Tt_{\GG(k,N)}$ is the tangent bundle.
	
	By Theorem 1.2.9 in Chapter 2 of \cite{OSS}, it is enough to show that $\Tt_{\GG(k,N)}$ is a stable bundle. Also, by \cite{L}, we have existence of a Kahler-Einstein metric implies that the tangent bundle is stable.
	
	By \cite{AP}, we know that all homogeneous varieties are Kahler-Einstein, so we are done.
\end{proof}

\begin{theorem} \label{theorem'}  
	$\Hh_{1}\bo_{(k,N-k)}$ is given by the nonzero element $(0,id) \in \Ext^1_{\Delta}(\CC^N/\V,\V) \bigoplus \End(\Omega_{\Delta})$.
\end{theorem}
	
\begin{proof}
	We denote the element that determine $\Hh_{1}\bo_{(k,N-k)}$ by $(a,b) \in \Ext^1_{\Delta}(\CC^N/\V,\V) \bigoplus \End(\Omega_{\Delta})$. Then it suffices to show $a=0$ and $b \neq 0$, since $\dim_{\CC} \End(\Omega_{\Delta}) =1$ by Lemma \ref{lemma 7}.
	
	Considering the following 2 exact triangles
	\[
	\Hh_{1}\bo_{(k,N-k)} \rightarrow  \Delta_{*} \CC^N/\V  \xrightarrow{(a,b)} \Delta_{*} \V[1],
	\]
	\[
	 \Delta_{*} \CC^N/\V \xrightarrow{id}  \Delta_{*} \CC^N/\V  \rightarrow 0. 
	\]

	Using the natural maps, we obtain the following diagram
	\begin{equation}
	\begin{tikzcd} \label{diagram4}
	 \Hh_{1}\bo_{(k,N-k)}  \arrow[r] \arrow[d]  
	&  \Delta_{*} \CC^N/\V  \arrow[r,"{(a,b)}"] \arrow[d,"0"]  
	& \Delta_{*} \V[1]    \\
	 \Delta_{*} \CC^N/\V \arrow[r,"0"]
	& 0  \arrow[r]  
	& \Delta_{*} \CC^N/\V[1]  .
	\end{tikzcd}	
	\end{equation}
	
	Obviously, the left square is commutative. So by axioms of triangulated category, there exists a map $s: \Delta_{*} \V[1] \rightarrow  \Delta_{*} \CC^N/\V[1]$ such that it completes the above diagram (\ref{diagram4}) to a morphism of exact triangles.
	
	By taking cones, we can complete (\ref{diagram4}) to the following morphism of exact triangles
	\[
	\xymatrix{
		\Hh_{1}\bo_{(k,N-k)} \ar[r] \ar[d]  
		&  \Delta_{*} \CC^N/\V  \ar[r]^{(a,b)} \ar[d]^{0}
		& \Delta_{*} \V[1] \ar[d]^{s}  \\
		\Delta_{*} \CC^N/\V  \ar[r]^{0} \ar[d]^{(a,b)}
		& 0  \ar[r]  \ar[d]
		& \Delta_{*} \CC^N/\V[1] \ar[d] \\
		\Delta_{*} \V[1] \ar[r]^{s}
		& \Delta_{*} \CC^N/\V[1] \ar[r]
		& C(s).
	}
	\]
	
	We label the morphisms in the following exact triangle
	\begin{equation*}
	\Delta_{*} \CC^N/\V  \xrightarrow{a_{1}} C(s)[-1] \xrightarrow{a_{2}} \Delta_{*} \V[1] \xrightarrow{s} \Delta_{*} \CC^N/\V  [1].
	\end{equation*}
		
	Then we also have the following morphism of exact triangles
	\begin{equation*}
	\begin{tikzcd}
	\Delta_{*} \V \arrow[r,"{s[-1]}"] \arrow[d]
	& \Delta_{*} \CC^N/\V \arrow[r,"a_1"] \arrow[d,"id"]
	& C(s)[-1] \arrow[d,"a_2"] \\
	\Hh_{1}\bo_{(k,N-k)}  \arrow[r] 
	&  \Delta_{*} \CC^N/\V  \arrow[r,"{(a,b)}"]  
	& \Delta_{*} \V[1] .
	\end{tikzcd}	
	\end{equation*}
	
	The commutativity of the right square implies that $(a,b) \circ id = a_2 \circ a_1 =0$. Similarly to the map $(a,b)$ which was obtained via the calculation of adjunctions, under the the following calculation
	\begin{equation*}
	\Hom_{\GG(k,N)\times\GG(k,N)}(\Delta_{*}\CC^N/\V ,\Delta_{*}\CC^N/\V)  \cong \Hom_{\Delta}(\CC^N/\V,\CC^N/\V) \ (\text{since} \ \Hom^{-i} \cong 0 \ \text{for all} \ i \geq 1),
	\end{equation*} the map $id:\Delta_{*} \CC^N/\V \rightarrow \Delta_{*} \CC^N/\V $ should correspond to the element $(id,0,0,...,0)$. 
	 
	So the composition $(a,b) \circ id  =0$, under the above adjunction calculations, is equal to $(a,b,0,0,...,0) \circ (id,0,0,...,0)=(a,0,0,...,0)=0$. Hence this implies $a=0$.
	
	Next we show that $b \neq 0$. We will prove this by using the argument of contradiction. So let us assume that $b=0$. Then we get $(a,b)=0$ which implies that $\Hh_{1}\bo_{(k,N-k)} \cong \Delta_{*}(\V \bigoplus \CC^N/\V)$. Thus we get 
	\begin{equation*}
	(\Psi^{+} \ast \Hh_{1})\bo_{(k,N-k)} \cong \Delta_{*}((\V \bigoplus \CC^N/\V) \otimes \det(\CC^N/\V))[1+k-N].
	\end{equation*}

	Applying the derived pullback $\Delta^*$, we get 
	\begin{equation} \label{eq1'}
	\Delta^*(\Psi^{+} \ast \Hh_{1})\bo_{(k,N-k)} \cong (\bigoplus_{i=0}^{k(N-k)} \bigwedge^{i}\Nn_{\Delta}^{\vee}[i]) \otimes (\V \bigoplus \CC^N/\V) \otimes \det(\CC^N/\V)[1+k-N].
	\end{equation}
	
	By definition, we have $(\Psi^{+} \ast \Hh_{1})\bo_{(k,N-k)} \cong t_{*}p_{*}(\Oo_{2D} \otimes (\V'/\V)^{N-k+1})$, and we also have $\Delta = j \circ t$. Thus 
	\begin{align}
	\begin{split}
	\Delta^*(\Psi^{+} \ast \Hh_{1})\bo_{(k,N-k)} 
	& \cong j^*t^*t_{*}p_{*}(\Oo_{2D} \otimes (\V'/\V)^{N-k+1}) \\ 
	& \cong j^*(\bigoplus_{i=0}^{\tcodim Y}\bigwedge^{i}\Nn^{\vee}_{Y}[i]) \otimes j^*p_{*}(\Oo_{2D} \otimes (\V'/\V)^{N-k+1}). \label{5.19}
	\end{split}
	\end{align}
	
	An easy calculation shows that $\tdim Y =k(N-k)+N-1$, so $\tcodim Y = k(N-k)-N+1$. On the other hand, we also calculate $j^*p_{*}(\Oo_{2D} \otimes (\V'/\V)^{N-k+1})$ explicitly. Using the following fibred product diagram 
	\begin{equation*} 
	\begin{tikzcd} [column sep=large]
	D \arrow[r, "i"]  \arrow[d, "p|_{D}"] & X \arrow[d, "p"]\\
	\Delta \arrow[r, "j"] & Y
	\end{tikzcd}
	\end{equation*} we get
	\[
	j^*p_{*}(\Oo_{2D} \otimes (\V'/\V)^{N-k+1}) \cong p|_{D*}i^*(\Oo_{2D} \otimes (\V'/\V)^{N-k+1}).
	\]
	
	Tensoring (\ref{ses8}) with $n=2$ by $(\V'/\V)^{N-k-1}$, we get 
	\begin{equation*}
	(\V''/\V''')^2 \otimes (\V'/\V)^{N-k-1} \xrightarrow{c\otimes id} (\V'/\V)^{N-k+1} \rightarrow \Oo_{2D} \otimes (\V'/\V)^{N-k+1}
	\end{equation*} where we denote $c:(\V''/\V''')^2 \rightarrow (\V'/\V)^2$ to be the natural inclusion.
	
	Applying the pullback $i^*$, since on $D$ we have $V=V''$, the map $c$ will become $0$ and thus
	\[
	i^*((\V''/\V''')^2 \otimes (\V'/\V)^{N-k-1}) \xrightarrow{0} i^*((\V'/\V)^{N-k+1}) \rightarrow i^*(\Oo_{2D} \otimes (\V'/\V)^{N-k+1}).
	\]
	
	This is again an exact triangle, which says that 
	\[
	i^*(\Oo_{2D} \otimes (\V'/\V)^{N-k+1}) \cong 	i^*((\V''/\V''')^2 \otimes (\V'/\V)^{N-k-1}) [1] \bigoplus i^*((\V'/\V)^{N-k+1}).
	\]
	
	Using projection formula and Proposition \ref{proposition 3}, we can conclude that 
	\begin{equation*}
	p|_{D*}i^*(\Oo_{2D} \otimes (\V'/\V)^{N-k+1})  \cong \CC^N/\V \otimes \det(\CC^N/V) [1+k-N].
	\end{equation*}
	
	Thus 
	\begin{equation} \label{eq2'}
	(\ref{5.19}) \cong 
	j^*(\bigoplus_{i=0}^{k(N-k)-N+1}\bigwedge^{i}\Nn^{\vee}_{Y}[i]) \otimes\CC^N/\V \otimes \det(\CC^N/V) [1+k-N].
	\end{equation}
	
	Now, we have $(\ref{eq1'}) \cong  (\ref{eq2'})$. Tensoring both sides by $\det(\CC^N/\V)^{-1}$ and shifted by $[N-k-1]$, we get 
	\begin{equation} \label{eq3'}
	j^*(\bigoplus_{i=0}^{k(N-k)-N+1}\bigwedge^{i}\Nn^{\vee}_{Y}[i]) \otimes\CC^N/\V \cong (\bigoplus_{i=0}^{k(N-k)} \bigwedge^{i}\Nn_{\Delta}^{\vee}[i]) \otimes (\V \bigoplus \CC^N/\V).
	\end{equation}
	
	Note that the differentials on both sides are all equal to $0$, and $N \geq 2$ implies $k(N-k)-N+1 < k(N-k)$. Taking cohomological sheaves $\Hh^{-k(N-k)}$ on (\ref{eq3'}) we get \begin{equation*}
	    0 \cong \det(\Nn^{\vee}_{\Delta}) \otimes (\V \bigoplus \CC^N/\V) 
	\end{equation*} which is a contradiction.
	
	Hence the assumption $b=0$ is false, and we must have $b \neq 0$. The proof is completed.
		
	\end{proof}

Finally, we prove conditions (13a), (14a). 

\begin{theorem} \label{theorem 3}(Conditions (13a), (14a)).
	We have the following (non-split) exact triangles 
	\begin{equation*}
	(\Hh_{1} \ast \Ee_{r})\bo_{(k,N-k)} \rightarrow (\Ee_{r}\ast \Hh_{ 1})\bo_{(k,N-k)}  \rightarrow (\Ee_{r+1}\bigoplus \Ee_{r+1}[1])\bo_{(k,N-k)}.
	\end{equation*}
\end{theorem}

\begin{proof}
    It is sufficient to prove the case where $r=0$. Recall that from Theorem \ref{theorem'}, we have  $\Hh_{1}\bo_{(k,N-k)}  \in \Dd^b(\GG(k,N)\times \GG(k,N))$  fits in the following exact triangle
	\begin{equation*}
	\Delta_{*}(\V) \rightarrow \Hh_{1}\bo_{(k,N-k)}  \rightarrow \Delta_{*}(\CC^N/\V),
	\end{equation*} and it is determined by the element 
	$(0,id) \in  \Ext^1_{\GG(k,N)}(\CC^N/\V,\V) \bigoplus \End(\Omega_{\GG(k,N)})$.
	
	The first part of the proof is to understand $(\Ee_{0} \ast \Hh_{1})\bo_{(k,N-k)}$ and $(\Hh_{1} \ast \Ee_{0})\bo_{(k,N-k)}$. The main idea is to keep tracking the element $(0,id)$ during the process of convolutions. Also, since we will keep calculate $\Ext^1(\Gg_{1}, \Gg_{2})$ in the proof, we will omit the geometric space in the notation in case if it causes no confusion.
	
	We calculate $(\Ee_{0} \ast \Hh_{1})\bo_{(k,N-k)}$ only, the argument is similar for $(\Hh_{1} \ast \Ee_{0})\bo_{(k,N-k)}$. By definition, it is given by $\pi_{13*}(\pi_{12}^*(\Hh_{1}\bo_{(k,N-k)} )\otimes \pi_{23}^*\iota_{*}\Oo_{Fl(k-1,k)})$. For the first step, we have the following exact triangle
	\begin{equation} \label{eq et1}
	 \pi_{12}^*\Delta_{*}(\V) \rightarrow \pi_{12}^*\Hh_{1}\bo_{(k,N-k)} \rightarrow \pi_{12}^* \Delta_{*}(\CC^N/\V).
	\end{equation} Using adjunction, it is easy to calculate that $\pi_{12}^*\Hh_{1}1_{(k,N-k)}$ is given by an element in
	\begin{equation*}
	\Ext^1(\pi_{12}^*\Delta_{*}(\CC^N/\V),\pi_{12}^*\Delta_{*}(\V)) \cong   \Ext^1_{\GG(k,N)}(\CC^N/\V,\V) \bigoplus \End(\Omega_{\GG(k,N)}).
	\end{equation*} Thus $\pi_{12}^*\Hh_{1}\bo_{(k,N-k)}$ is still given by $(0,id) \in \Ext^1_{\GG(k,N)}(\CC^N/\V,\V) \bigoplus \End(\Omega_{\GG(k,N)})$.
	
    The second step is to tensor the exact triangle (\ref{eq et1}) by $\pi_{23}^*\iota_{*}\Oo_{Fl(k-1,k)}$, and we get the following exact triangle
	\begin{equation} \label{5.2.25}
	\pi_{12}^*\Delta_{*}(\V) \otimes \pi_{23}^*\iota_{*}\Oo_{Fl(k-1,k)} \rightarrow \pi_{12}^*\Hh_{1}\bo_{(k,N-k)} \otimes \pi_{23}^*\iota_{*}\Oo_{Fl(k-1,k)} \rightarrow \pi_{12}^* \Delta_{*}(\CC^N/\V) \otimes \pi_{23}^*\iota_{*}\Oo_{Fl(k-1,k)}.
	\end{equation}
	
	So $\pi_{12}^*\Hh_{1}\bo_{(k,N-k)}\otimes \pi_{23}^*\iota_{*}\Oo_{Fl(k-1,k)}$ is given by an element in
	\begin{equation} \label{ext1}
	\Ext^1(\pi_{12}^*\Delta_{*}(\CC^N/\V)\otimes \pi_{23}^*\iota_{*}\Oo_{Fl(k-1,k)},\pi_{12}^*\Delta_{*}(\V)\otimes \pi_{23}^*\iota_{*}\Oo_{Fl(k-1,k)}). 
	\end{equation}

	To calculate (\ref{ext1}), we use the following fibred product diagram
	\begin{equation*}
	\begin{tikzcd} 
	\GG(k,N)\times \GG(k-1,N)   \arrow[r, "\Delta \times id"]  \arrow[d, "\pi_{1}"] &\GG(k,N)\times \GG(k,N) \times \GG(k-1,N) \arrow[d, "\pi_{12}"]\\
	\GG(k,N)  \arrow[r, "\Delta"] & \GG(k,N)\times \GG(k,N).
	\end{tikzcd}
	\end{equation*}

	Using base change, projection formula and the identiy $\pi_{23} \circ (\Delta \times id) = id$, we obtain
	\begin{align}
	(\ref{ext1}) 
	&\cong \Ext^1(\pi_{1}^*(\CC^N/\V)\otimes \iota_{*}\Oo_{Fl(k-1,k)},\pi_{1}^*(\V)\otimes\iota_{*}\Oo_{Fl(k-1,k)}) \label{ext4} \\
	&\bigoplus \Hom(\Nn_{1}^{\vee} \otimes \pi_{1}^*(\CC^N/\V)\otimes \iota_{*}\Oo_{Fl(k-1,k)},\pi_{1}^*(\V)\otimes\iota_{*}\Oo_{Fl(k-1,k)}) \label{ext5}, 
	\end{align} where $\Nn_{1}^{\vee}$ denotes the conormal bundle $\Nn^{\vee}_{\Delta\times \GG(k-1,N)/\GG(k,N)\times\GG(k,N)\times\GG(k-1,N)}$.
	
	Similarly, applying adjunction again, it is easy to calculate  
	\begin{equation*}
    (\ref{ext4}) \cong \Ext^1_{Fl(k-1,k)}(\iota^*\pi_{1}^*(\CC^N/\V),\iota^*\pi_{1}^*(\V)) \bigoplus \Hom(\iota^*\pi_{1}^*(\CC^N/\V)\otimes \Nn^{\vee}_{2},\iota^*\pi_{1}^*(\V)), 
	\end{equation*} where $\Nn^{\vee}_{2}$ denotes the conormal bundle $\Nn^{\vee}_{Fl(k-1,k)/\GG(k,N)\times \GG(k-1,N)}$.
	
	Note that we have 
	\begin{align*}
	\Nn^{\vee}_{1} \cong  (\CC^N/\V)^{\vee} \otimes \V \ \text{and} \
	\Nn^{\vee}_{2} \cong (\CC^N/\V)^{\vee} \otimes \V',
	\end{align*} so we obtain
	\begin{align*}
	(\ref{ext5}) \cong \Hom((\CC^N/\V)^{\vee} \otimes \V, (\CC^N/\V)^{\vee} \otimes \V) \cong \End(\Omega_{\GG(k,N)}),
	\end{align*} and similarly
	\begin{equation*}
	\Hom(\iota^*\pi_{1}^*(\CC^N/\V)\otimes \Nn^{\vee}_{2},\iota^*\pi_{1}^*(\V)) \cong \Hom_{Fl(k-1,k)}( (\CC^N/\V)^{\vee} \otimes \V',(\CC^N/\V)^{\vee}\otimes \V).
	\end{equation*}
	
	Combining them, we get 
	\begin{equation*}
	(\ref{ext1}) 
	\cong \Ext^1_{Fl(k-1,k)}(\CC^N/\V,\V) \bigoplus  \Hom((\CC^N/\V)^{\vee}\otimes \V',(\CC^N/\V)^{\vee} \otimes \V ) \bigoplus \End(\Omega_{\GG(k,N)}). \label{iso2}
	\end{equation*}

	This tells us that the exact triangle (\ref{5.2.25}) is determined by an element 
	\[
	(0,a,id) \in \Ext^1_{Fl(k-1,k)}(\CC^N/\V,\V) \bigoplus  \Hom((\CC^N/\V)^{\vee}\otimes \V',(\CC^N/\V)^{\vee} \otimes \V ) \bigoplus \End(\Omega_{\GG(k,N)}) 
	\] for some $a \in \Hom((\CC^N/\V)^{\vee}\otimes \V',(\CC^N/\V)^{\vee} \otimes \V )$.

	For the final third step, we apply $\pi_{13*}$ and get the following exact triangle
	\begin{equation}  \label{ses5}
	\pi_{13*}(\pi_{12}^*\Delta_{*}(\V) \otimes \pi_{23}^*\iota_{*}\Oo_{Fl(k-1,k)}) \rightarrow (\Ee_{0} \ast \Hh_{1})\bo_{(k,N-k)}
	\rightarrow 
	\pi_{13*}(\pi_{12}^* \Delta_{*}(\CC^N/\V) \otimes \pi_{23}^*\iota_{*}\Oo_{Fl(k-1,k)}).
	\end{equation}

	Using the base change and  $\pi_{13} \circ (\Delta\times id) = \pi_{23} \circ (\Delta\times id)  = id$, it is easy to calculate that 
	\begin{equation*}
	\pi_{13*}(\pi_{12}^*\Delta_{*}(\V)\otimes \pi_{23}^*\iota_{*}\Oo_{Fl(k-1,k)}) 
	\cong \pi_{1}^*(\V)\otimes \iota_{*}\Oo_{Fl(k-1,k)} \cong \iota_{*}\V,
	\end{equation*} and similarly for the case for $\CC^N/\V$. Thus the exact triangle (\ref{ses5}) becomes
	\begin{equation} \label{ses14} 
	\iota_{*}(\V) \rightarrow (\Ee_{0} \ast \Hh_{1})\bo_{(k,N-k)} \rightarrow  \iota_{*}(\CC^N/\V)
	\end{equation}
	and $(\Ee_{0} \ast \Hh_{1})1_{(k,N-k)}$ is given by an element in the following
	\begin{equation*}
	\Ext^1(\iota_{*}(\CC^N/\V),\iota_{*}(\V))
	\cong \Ext_{Fl(k-1,k)}^1(\CC^N/\V,\V) \bigoplus \Hom((\CC^N/\V)^{\vee} \otimes \V',(\CC^N/\V)^{\vee} \otimes \V).
	\end{equation*}

	Since at the second step, the exact triangle (\ref{5.2.25}) is determined by the element 
	\begin{equation*}
	(0,a,id) \in \Ext^1_{Fl(k-1,k)}(\CC^N/\V,\V) \bigoplus  \Hom((\CC^N/\V)^{\vee}\otimes \V',(\CC^N/\V)^{\vee} \otimes \V )  \bigoplus \End(\Omega_{\GG(k,N)})
	\end{equation*} we have the $(\Ee_{0} \ast \Hh_{1})1_{(k,N-k)}$ in (\ref{ses14}) is determined by the element
	\[
	(0,a) \in \Ext_{Fl(k-1,k)}^1(\CC^N/\V,\V) \bigoplus \Hom((\CC^N/\V)^{\vee} \otimes \V',(\CC^N/\V)^{\vee} \otimes \V).
	\]

	On the other hand, using the same argument $(\Hh_{1} \ast \Ee_{0})\bo_{(k,N-k)}$ fits into the following exact triangle
	\begin{equation} \label{sess}
	\iota_{*}(\V') \rightarrow (\Hh_{1} \ast \Ee_{0})\bo_{(k,N-k)} \rightarrow  \iota_{*}(\CC^N/\V') 
	\end{equation} and is determined by an element 
	\begin{equation*}
	(0, b) \in  \Ext_{Fl(k-1,k)}^1(\CC^N/\V',\V') \bigoplus \Hom((\CC^N/\V)^{\vee} \otimes \V',(\CC^N/\V')^{\vee} \otimes \V').
	\end{equation*}

	With the two exact triangles (\ref{ses14}) and (\ref{sess}), we form the following diagram of exact triangles
	\begin{equation} \label{diagram5} 
	\xymatrix{
		\iota_{*}(\V')  \ar[r] \ar[d]^{i}  
		& (\Hh_{1} \ast \Ee_{0})\bo_{(k,N-k)}  \ar[r]
		& \iota_{*}(\CC^N/\V')  \ar[r]^{(0,b)} \ar[d]^{\pi} 
		& \iota_{*}(\V')[1] \ar[d]^{i[1]} \\
		 \iota_{*}(\V) \ar[r] 
		& (\Ee_{0} \ast \Hh_{1})\bo_{(k,N-k)} \ar[r]  
		& \iota_{*}(\CC^N/\V)  \ar[r]^{(0,a)}  
		& \iota_{*}(\V)[1]
	}
	\end{equation}

	The second part of the proof is to show that the right square in (\ref{diagram5}) commutes so that there is an induced map $(\Hh_{1} \ast \Ee_{0} )\bo_{(k,N-k)}  \rightarrow (\Ee_{0} \ast \Hh_{1})\bo_{(k,N-k)} $,
	
	For simplicity, we denote $a$ for the map $(0,a)$ and similarly $b$ for $(0,b)$. So we have to show that $i[1] \circ b = a \circ \pi$.
	
	Let us denote the maps in the following exact triangles
	\[
	(\Ee_{0} \ast \Hh_{1})\bo_{(k,N-k)} \xrightarrow{x}  
	 \iota_{*}(\CC^N/\V)  \xrightarrow{a}
	 \iota_{*}(\V)[1],
	\]
	\[
	\iota_{*}(\CC^N/\V') \xrightarrow{\pi} \iota_{*}(\CC^N/\V) \xrightarrow{y} \iota_{*}(\V/\V')[1],
	\] and consider the following diagram of exact triangles
    \begin{equation}
    \begin{tikzcd}  \label{diagram3}
	(\Ee_{0} \ast \Hh_{1})\bo_{(k,N-k)}  \arrow[r,"x"] \arrow[d,"x"]  
	&  \iota_{*}(\CC^N/\V)  \arrow[r,"a"] \arrow[d,"y"]  
	& \iota_{*}(\V)[1]    \\
	\iota_{*}(\CC^N/\V)  \arrow[r,"y"]
	& \iota_{*}(\V/\V')[1]  \arrow[r]  
	& \iota_{*}(\CC^N/\V')[1]  .
    \end{tikzcd}	
    \end{equation}
	
	Since the left square is commutative, by axioms of triangulated categories, there exists a morphism $f:\iota_{*}(\V)[1] \rightarrow  \iota_{*}(\CC^N/\V')[1] $ such that it complete to a morphism between exact triangles.
	
	Next, we complete diagram (\ref{diagram3}) to the following morphisms between exact triangles via taling cone
	\[
	\xymatrix{
		(\Ee_{0} \ast \Hh_{1})\bo_{(k,N-k)}  \ar[r]^{x} \ar[d]^{x}  
		&  \iota_{*}(\CC^N/\V)  \ar[r]^{a} \ar[d]^{y}  
		& \iota_{*}(\V)[1]  \ar[d]^{f}  \\
		\iota_{*}(\CC^N/\V)  \ar[r]^{y} \ar[d]
		& \iota_{*}(\V/\V')[1]  \ar[r]  \ar[d]
		& \iota_{*}(\CC^N/\V')[1] \ar[d]^{\theta} \\
		\iota_{*}(\V)[1] \ar[r]^{f} 
		& \iota_{*}(\CC^N/\V')[1] \ar[r]^{\theta} 
		& C(f)
	}
	\]
	
	We consider the exact triangle on right hand side, i.e.
	\[
	C(f)[-1] \xrightarrow{\delta} \iota_{*}(\V)[1] \xrightarrow{f} \iota_{*}(\V)[1] \xrightarrow{\theta} C(f)
	\] for some map $\delta$. Then we also have the following morphism of exact triangles
	\begin{equation*}
	\begin{tikzcd}  
	\iota_{*}(\V) \arrow[r,"{f[-1]}"] \arrow[d] 
	& \iota_{*}(\CC^N/\V') \arrow[r,"{\theta[-1]}"] \arrow[d,"\pi"] 
	& C(f)[-1] \arrow[d, "{\delta}"] \\
	(\Ee_{0} \ast \Hh_{1})\bo_{(k,N-k)}  \arrow[r,"x"] 
	&  \iota_{*}(\CC^N/\V)  \arrow[r,"a"]
	& \iota_{*}(\V)[1] .    
	\end{tikzcd}	
	\end{equation*}
	
	By the commutative of the squares, we get $a \circ \pi = \delta \circ \theta[-1] =0$. Using the same argument as above we can also obtain $i[1] \circ b =0$. Thus $i[1] \circ b = a \circ \pi$ and the diagram commutes. We get an induced map $\gamma:(\Hh_{1} \ast \Ee_{0} )\bo_{(k,N-k)}  \rightarrow (\Ee_{0} \ast \Hh_{1})\bo_{(k,N-k)}$. The diagram (\ref{diagram5}) becomes
	\[
	\xymatrix{
		\iota_{*}(\V') \ar[d]^{i} \ar[r] &(\Hh_{1} \ast \Ee_{0} )1_{(k,N-k)}  \ar[r] \ar[d]^{\gamma} &\iota_{*}(\CC^N/\V') \ar[r]^{b} \ar[d]^{\pi} & \iota_{*}(\V')[1] \ar[d]^{i[1]} \\
		\iota_{*}(\V)  \ar[r] &(\Ee_{0} \ast \Hh_{1})1_{(k,N-k)} \ar[r]  &\iota_{*}(\CC^N/\V) \ar[r]^{a} & \iota_{*}(\V)[1],
	}
	\] then we complete the diagram into exact triangles via taking cones
	\[
	\xymatrix{
		\iota_{*}(\V') \ar[d]^{i} \ar[r] &(\Hh_{1} \ast \Ee_{0} )1_{(k,N-k)}  \ar[r] \ar[d]^{\gamma} &\iota_{*}(\CC^N/\V') \ar[r]^{b} \ar[d]^{\pi} & \iota_{*}(\V')[1] \ar[d]^{i[1]} \\
		\iota_{*}(\V)  \ar[r] \ar[d]^{j} &(\Ee_{0} \ast \Hh_{1})1_{(k,N-k)} \ar[r] \ar[d]  &\iota_{*}(\CC^N/\V) \ar[r]^{a} \ar[d]^{y} & \iota_{*}(\V)[1] \ar[d]^{j[1]} \\
		\iota_{*}(\V/\V') \ar[r] & Cone(\gamma) \ar[r]& 	\iota_{*}(\V/\V')[1] \ar[r]^{s} &	\iota_{*}(\V/\V')[1]. 
	}
	\]
	
	The final part of the proof is to show that
	\begin{equation*}
	Cone(\gamma) \cong \iota_{*}(\V/\V') \bigoplus \iota_{*}(\V/\V')[1]
	\end{equation*} and since that $\Ee_{1}\bo_{(k,N-k)}\coloneqq \iota_{*}(\V/\V')$, we obtain the desire exact triangle.

	This is the same as showing the map $s:\iota_{*}(\V/\V')[1] \rightarrow	\iota_{*}(\V/\V')[1] $ is zero. Note that 
	\begin{equation*}
	\Hom(\iota_{*}(\V/\V')[1],\iota_{*}(\V/\V')[1])\cong \CC,
	\end{equation*} so the only maps $s:\iota_{*}(\V/\V')[1] \rightarrow	\iota_{*}(\V/\V')[1] $ are either zero or isomorphism (identity up to a nonzero constant). Assume it is not zero, i.e. $s=id$, then from the commutativity of the square, we have $s \circ y = y = j[1] \circ a$.
	
	Both the maps $y$ and $j[1] \circ a$ are in 
	\[
	\Ext^{1}(\iota_{*}\CC^N/\V, \iota_{*}\V/\V') \cong \Ext^{1}_{Fl(k-1,k)}(\CC^N/\V, \V/\V') \bigoplus \Hom(\V' \otimes (\CC^N/\V)^{\vee}, \V/\V' \otimes (\CC^N/\V)^{\vee}).
	\]
	
	Clearly, we have $y \in \Ext^{1}_{Fl(k-1,k)}(\CC^N/\V, \V/\V')$ which is the first direct summand. However, since $a \in \Hom(\V' \otimes (\CC^N/\V)^{\vee}, \V/\V' \otimes (\CC^N/\V)^{\vee})$, we must have $j[1] \circ a \in \Hom(\V' \otimes (\CC^N/\V)^{\vee}, \V/\V' \otimes (\CC^N/\V)^{\vee})$, which is the second direct summand.
	
	So the only possible case is $y=j[1] \circ a =0$, which contradicts to the fact $y \neq 0$. Hence $s=0$.
	
	Thus we get $Cone(\gamma) \cong \iota_{*}(\V/\V') \bigoplus \iota_{*}(\V/\V')[1]$ and the proof is complete.
	
\end{proof}

\begin{example}
We give an example for Theorem \ref{theorem 4}, which is the case where the weights are $(k,N-k)=(2,0), \ (1,1),\ (0,2)$. The corresponding weight categories are described in the following picture.
\begin{equation*}
\begin{tikzcd}
\Dd^b(\GG(2,2)) \arrow[r,shift left,"{\E}_{r}"] &\Dd^b(\GG(1,2)=\PP^1) \arrow[l,shift left,"{\F}_{s}"] \arrow[r,shift left,"{\E}_{r}"] &\Dd^b(\GG(0,2)) \arrow[l,shift left,"{\F}_{s}"].
\end{tikzcd}
\end{equation*}

Consider $\PP^1 \times \PP^1 =\{(V,V'')\ | \ \dim V=\dim V''=1\}$, and let $\V$, $\V''$ to be the tautological line bundles on $\PP^1 \times \PP^1$ and $\CC^2$ to be the trivial bundle of rank 2 on $\PP^1$. Denote $\Delta \subset \PP^1 \times \PP^1$ to be the diagonal which is also a divisor. Then by definition $\Hh_{1}\bo_{(1,1)}=\Oo_{2\Delta} \otimes \CC^2/\V$, and it is determined by the following exact triangle in $\Dd^b(\PP^1 \times \PP^1)$ 
\begin{equation*}
\Delta_{*}\V \rightarrow \Hh_{1}\bo_{(1,1)}=\Oo_{2\Delta} \otimes \CC^2/\V \rightarrow \Delta_{*} \CC^2/\V .	
\end{equation*}

The kernel $\Ee_{0} \bo_{(1,1)}$ is $\Oo_{\PP^1}$, a simple calculation gives us the following exact triangles in $\Dd^b(\PP^1)$
\begin{align*}
    & \V \rightarrow (\Ee_{0} \ast \Hh_{1})\bo_{(1,1)}=\V \oplus \CC^2/\V \rightarrow  \CC^2/\V  \\
    & 0 \rightarrow (\Hh_{1} \ast \Ee_{0} )\bo_{(1,1)}=\CC^2 \rightarrow \CC^2 .
\end{align*}

Combining them together, we obtain the following exact triangle
\begin{equation*} 
 (\Hh_{1} \ast \Ee_{0})\bo_{(1,1)}=\CC^2 \rightarrow (\Ee_{0} \ast \Hh_{1})\bo_{(1,1)} = \V \oplus \CC^2/\V \rightarrow \V \oplus \V[1]
\end{equation*} where $\V$ is the kernel $\Ee_{1}\bo_{(1,1)}$, which agrees with (\ref{eq5}). 
\end{example}

Finally, we prove condition (4).

\begin{lemma}   \label{lemma 12} (Condition (4)). $\Hh_{1}\bo_{(k,N-k)}$ and $\Hh_{-1}\bo_{(k,N-k)}$ are biadjoint to each other.
    \begin{equation*}
	(\Hh_{1}\bo_{(k,N-k)})_{L} \cong \Hh_{-1}\bo_{(k,N-k)} \cong (\Hh_{1}\bo_{(k,N-k)})_{R}
	\end{equation*}
\end{lemma}

\begin{proof}
	It suffices to prove the left adjoint case, the right adjoint is similar.
	
	By definition and taking left adjoint, we get 
	\begin{align*}
	(\Hh_{1}\bo_{(k,N-k)})_{L}&=\{(\Psi^{+}\bo_{(k,N-k)})^{-1} \ast t_{*}p_{*}(\Oo_{2D} \otimes (\V'/\V)^{N-k+1})\}_{L} \\ 
	&\cong (t_{*}p_{*}(\Oo_{2D} \otimes (\V'/\V)^{N-k+1}))_{L} \ast ((\Psi^{+}\bo_{(k,N-k)})^{-1})_{L}  
	\end{align*}

	Note that since $(\Psi^{+}\bo_{(k,N-k)})^{-1}$ is invertible, the left adjoint is isomorphic to its inverse, i.e. $((\Psi^{+}\bo_{(k,N-k)})^{-1})_{L} \cong \Psi^{+}\bo_{(k,N-k)}$.

	Next, we need to compute $(t_{*}p_{*}(\Oo_{2D} \otimes (\V'/\V)^{N-k+1}))_{L}$. By definition, we have 
	\begin{equation*} 
	(t_{*}p_{*}(\Oo_{2D} \otimes (\V'/\V)^{N-k+1}))_{L} \cong (t_{*}p_{*}(\Oo_{2D} \otimes (\V'/\V)^{N-k+1}))^{\vee} \otimes \pi_{2}^*\omega_{\GG(k,N)}[\tdim \GG(k,N)].
	\end{equation*}
	
	Using Grothendieck-Verdier duality, we get
	\begin{align}
	\begin{split}
	&(t_{*}p_{*}(\Oo_{2D} \otimes (\V'/\V)^{N-k+1}))^{\vee} \\
	&\cong t_{*}p_{*}((\Oo_{2D} \otimes (\V'/\V)^{N-k+1})^{\vee} \otimes \omega_{X} ) \otimes \omega_{\GG(k,N)\times\GG(k,N)}^{-1}[\tdim X-\tdim\GG(k,N)\times\GG(k,N)]. \label{adj2}
	\end{split}
	\end{align}
	
	We know that $D=Fl(k-1,k,k+1) \subset X$ is a divisor, so $\tdim X=k(N-k)+N-1$. Next, we have 
	\begin{equation} \label{adj3}
	(\Oo_{2D} \otimes (\V'/\V)^{N-k+1})^{\vee} \cong  (\Oo_{2D})^{\vee} \otimes ((\V'/\V)^{N-k+1})^{\vee} \cong (\Oo_{2D})^{\vee} \otimes (\V'/\V)^{-N+k-1}
	\end{equation}
	
	To calculate $(\Oo_{2D})^{\vee}$, we have the following exact triangle
	\[
	 i_{*}\Oo_{D} \otimes \Oo_{X}(-D) \rightarrow \Oo_{2D} \rightarrow i_{*}\Oo_{D} 
	\] taking dual, we get 
	\[
	 (i_{*}\Oo_{D})^{\vee} \rightarrow (\Oo_{2D})^{\vee} \rightarrow  (i_{*}\Oo_{D})^{\vee} \otimes \Oo_{X}(D) .
	\]
	
	Since $(i_*\Oo_{D})^{\vee} \cong  i_*\Oo_{D} \otimes \Oo_{X}(D)[-1]$, the exact triangle becomes
	\[
	 i_*\Oo_{D} \otimes \Oo_{X}(D)[-1] \rightarrow (\Oo_{2D})^{\vee} \rightarrow  i_*\Oo_{D} \otimes \Oo_{X}(2D)[-1] .
	\]
	
	We conclude that $(\Oo_{2D})^{\vee} \cong \Oo_{2D} \otimes \Oo_{X}(2D)[-1]$. Thus 
	\begin{equation*}
	(\ref{adj3})  \cong
	\Oo_{2D} \otimes \Oo_{X}(2D)  \otimes (\V'/\V)^{-N+k-1} [-1]
	\end{equation*} and thus
	\begin{equation*}
	(\ref{adj2}) \cong
	t_{*}p_{*}(\Oo_{2D} \otimes \Oo_{X}(2D) \otimes (\V'/\V)^{-N+k-1} \otimes \omega_{X} ) \otimes \omega_{\GG(k,N)\times\GG(k,N)}^{-1}[-k(N-k)+N-2].
	\end{equation*}
	
	We already know that $\Oo_{X}(D)|_{D} \cong (\V/\V''')^{-1} \otimes (\V'/\V)$, and an easy calculation gives $\omega_{D} \cong (\V/\V''')^{-k} \otimes (\V'/\V)^{N-k} \otimes \det(\V) \otimes \det(\CC^N/\V)^{-1} \otimes \omega_{\GG(k,N)}$.
	
	Thus by the adjunction formula, we obtain 
	\begin{align*}
	 \omega_{X}|_{D} &\cong (\V/\V''')\otimes (\V'/\V)^{-1} \otimes \omega_{D} \\
	 & \cong (\V/\V''')^{-k+1} \otimes (\V'/\V)^{N-k-1} \otimes \det(\V) \otimes \det(\CC^N/\V)^{-1} \otimes \omega_{\GG(k,N)}
	\end{align*}

	So combining all of them and using the projection formula, we get 
	\begin{equation*}
	(t_{*}p_{*}(\Oo_{2D} \otimes (\V'/\V)^{N-k+1}))_{L} 
	\cong t_{*}p_{*}(\Oo_{2D} \otimes (\V/\V''')^{-k-1}) \otimes \Delta_{*}\det(\V) \otimes \Delta_{*}\det(\CC^N/\V)^{-1}[N-2] 
	\end{equation*} 
	
	Thus, 
	\begin{align*}
	(\Hh_{1}\bo_{(k,N-k)})_{L} &\cong  (t_{*}p_{*}(\Oo_{2D} \otimes (\V'/\V)^{N-k+1}))_{L} \ast \Psi^{+}\bo_{(k,N-k)} \\ &\cong \pi_{13*}(  \pi_{12}^*\Delta_{*}\det(\V) \otimes \pi_{23}^*(t_{*}p_{*}(\Oo_{2D} \otimes (\V/\V''')^{-k-1})) )[k-1] \\
	& \cong (\Psi^{-}\bo_{(k,N-k)})^{-1} \ast [t_{*}p_{*}(\Oo_{2D} \otimes (\V/\V''')^{-k-1})]
	\end{align*}  which is the definition for $\Hh_{-1}\bo_{(k,N-k)}$.
	
\end{proof}

So far we prove many conditions for the $\SL_{2}$ case, the only conditions that remains to prove are (6), (7), (8) with $i=j$. We will leave them to the next subsection where we prove them for all $i,j$. Thus plus all the above works we prove Theorem \ref{theorem 2}.

\subsection{The rest relations}

Since we have already proved the $\SL_2$ case, many relations in Definition \ref{definition 2} are a direct generalization of the $\SL_2$ version. In order to prove Theorem \ref{theorem 1}, we prove the rest relations for the $\SL_n$ case in this section,  this means that we prove those that are not address in the $\SL_2$ case.

First, we note that conditions (2), (3) are obvious. Next, it is easy to check conditions (9c), (10c), (11c), (12c), (13c), (14c), (15) where $|i-j| \geq 2$. 

So it remains to check the relations (6), (7), (8), (9b), (10b), (11b), (12b), (13b), (14b).

Condition (6) is easy to show since ${\Psi}^{\pm}_{i}\bo_{\kk}$ are just line bundles.

The next is condition (8).

\begin{lemma} (Condition (8)). \label{lemma 13}
	\begin{equation*}
	   (\Hh_{i,1} \ast \Psi_{j}^{+})\bo_{\kk} \cong  ( \Psi_{j}^{+} \ast \Hh_{i, 1})\bo_{\kk}
	\end{equation*} for all $i,j$.
\end{lemma}

\begin{proof}
	We know that the kernel $\Hh_{i,1}1_{\kk}$ in the following exact triangle
	\[
	\Delta_{*}\V_{i}/\V_{i-1} \rightarrow \Hh_{i,1}\bo_{\kk} \rightarrow \Delta_{*}\V_{i+1}/\V_{i} 
	\] is determined by the element in $\Ext^1_{Y(\kk)\times Y(\kk)}(\Delta_{*}(\V_{i+1}/\V_{i}),\Delta_{*}(\V_{i}/\V_{i-1}))$.
	
	By Definition \ref{definition 2}, $\Psi_{j}^{+}\bo_{\kk} \coloneqq \Delta_{*}\tdet (\V_{j+1}/\V_{j})[1-k_{j+1}]$. Taking convolution, we get the following exact triangle
	\begin{align*}
	&\Delta_{*}(\V_{i}/\V_{i-1} \otimes \tdet (\V_{j+1}/\V_{j}))[1-k_{j+1}]  \rightarrow ( \Psi_{j}^{+} \ast \Hh_{i,1})\bo_{\kk} \rightarrow  \Delta_{*}(\V_{i+1}/\V_{i}) \otimes \tdet (\V_{j+1}/\V_{j}))[1-k_{j+1}] \\
	 &\Delta_{*}(\V_{i}/\V_{i-1} \otimes \tdet (\V_{j+1}/\V_{j}))[1-k_{j+1}]  \rightarrow (\Hh_{i,1} \ast \Psi_{j}^{+})\bo_{\kk} \rightarrow  \Delta_{*}(\V_{i+1}/\V_{i}) \otimes \tdet (\V_{j+1}/\V_{j}))[1-k_{j+1}].
	\end{align*}
	
	Since tensoring a line bundle does not change the class that determines the exact triangle, the above two exact triangles are still determined by the same element in $\Ext^1_{Y(\kk)\times Y(\kk)}(\Delta_{*}(\V_{i+1}/\V_{i}),\Delta_{*}(\V_{i}/\V_{i-1})) $.
	
	Thus we obtain the isomorphism $(\Hh_{i,1} \ast \Psi_{j}^{+})\bo_{\kk} \cong  ( \Psi_{j}^{+} \ast \Hh_{i,1})\bo_{\kk}$.
	
\end{proof}

Next, we prove condition (7), which is the following.

\begin{lemma} (Condition (7)).
	\begin{equation*}
	(\Hh_{i, 1} \ast \Hh_{j, 1})\bo_{\kk} \cong  ( \Hh_{j, 1} \ast \Hh_{i, 1})\bo_{\kk}
	\end{equation*} for all $i,j$.
\end{lemma}

\begin{proof}
	We know that $\Hh_{i, 1}\bo_{\kk}$ fits in the following exact triangle
	\begin{equation} \label{ses18}
	\Delta_{*}\V_{i}/\V_{i-1} \rightarrow \Hh_{i,1}\bo_{\kk} \rightarrow \Delta_{*}\V_{i+1}/\V_{i}.
	\end{equation}

	From Theorem \ref{theorem'}, $\Hh_{i,1}\bo_{\kk}$ is determined by an element in
	\begin{equation*}
	\Ext^1(\Delta_{*}(\V_{i+1}/\V_{i}),\Delta_{*}(\V_{i}/\V_{i-1})) \cong \Ext^1_{Y(\kk)}(\V_{i+1}/\V_{i},\V_{i}/\V_{i-1}) \bigoplus \Hom(\Omega_{Y(\kk)}, (\V_{i+1}/\V_{i})^{\vee} \otimes \V_{i}/\V_{i-1}),
	\end{equation*} and we denote this element to be $(0,a_{i})$.
	
	Taking convolution of the exact triangle (\ref{ses18}) with $\Hh_{j, 1}\bo_{\kk}$, we can obtain the following two exact triangles
	\begin{align*}
	&\Hh_{j, 1}\bo_{\kk} \ast \Delta_{*}\V_{i}/\V_{i-1} \rightarrow (\Hh_{j, 1}  \ast \Hh_{i,1})\bo_{\kk} \rightarrow \Hh_{j, 1}\bo_{\kk} \ast \Delta_{*}\V_{i+1}/\V_{i}, \\
	&\Delta_{*}\V_{i}/\V_{i-1} \ast \Hh_{j, 1}\bo_{\kk} \rightarrow (\Hh_{i,1} \ast \Hh_{j, 1} )\bo_{\kk} \rightarrow \Delta_{*}\V_{i+1}/\V_{i} \ast \Hh_{j, 1}\bo_{\kk}.
	\end{align*}
	
	Note that $\Hh_{j, 1}\bo_{\kk}$ is determined by the following exact triangle
	\begin{equation} \label{ses19}
	\Delta_{*}\V_{j}/\V_{j-1} \rightarrow \Hh_{j,1}\bo_{\kk} \rightarrow \Delta_{*}\V_{j+1}/\V_{j}
	\end{equation}
	
	Let (\ref{ses19}) convolute with $\Delta_{*}\V_{i}/\V_{i-1}$, we obtain the following exact triangles
	\begin{align*}
	&\Delta_{*}(\V_{i}/\V_{i-1}\otimes \V_{j}/\V_{j-1}) \rightarrow \Delta_{*}\V_{i}/\V_{i-1} \ast \Hh_{j,1}\bo_{\kk}\rightarrow \Delta_{*}(\V_{i}/\V_{i-1} \otimes \V_{j+1}/\V_{j}), \\
	&\Delta_{*}(\V_{j}/\V_{j-1} \otimes \V_{i}/\V_{i-1}) \rightarrow \Hh_{j,1}\bo_{\kk} \ast \Delta_{*}\V_{i}/\V_{i-1}  \rightarrow \Delta_{*}(\V_{j+1}/\V_{j} \otimes \V_{i}/\V_{i-1}).
	\end{align*}
	
	We show that $\Delta_{*}\V_{i}/\V_{i-1} \ast \Hh_{j,1}\bo_{\kk} \cong \Hh_{j,1}\bo_{\kk} \ast \Delta_{*}\V_{i}/\V_{i-1}$. Note that we have the following diagram for morphism of exact triangles
	\begin{equation} \label{diagram 8}
	\xymatrix{
		\Hh_{j,1}\bo_{\kk} \ast \Delta_{*}\V_{i}/\V_{i-1} \ar[r]
		& \Delta_{*}(\V_{i}/\V_{i-1} \otimes \V_{j+1}/\V_{j}) \ar[r] \ar[d]^{id} 
		&\Delta_{*}(\V_{i}/\V_{i-1}\otimes \V_{j}/\V_{j-1})[1] \ar[d]^{id[1]} \\
		\Delta_{*}\V_{i}/\V_{i-1} \ast \Hh_{j,1}\bo_{\kk}  \ar[r]  
		& \Delta_{*}(\V_{i}/\V_{i-1} \otimes \V_{j+1}/\V_{j}) \ar[r]
		&\Delta_{*}(\V_{i}/\V_{i-1}\otimes \V_{j}/\V_{j-1})[1].
	}
	\end{equation}
	
	In order to induce a map between $\Delta_{*}\V_{i}/\V_{i-1} \ast \Hh_{j,1}\bo_{\kk}$ and $\Hh_{j,1}\bo_{\kk} \ast \Delta_{*}\V_{i}/\V_{i-1}$, we need to show that the right square  in diagram (\ref{diagram 8}) is commutative. This can be done by using the same argument as in the proof of Theorem \ref{theorem 3}. Thus it induces a map $x:\Hh_{j,1}\bo_{\kk} \ast \Delta_{*}\V_{i}/\V_{i-1} \rightarrow \Delta_{*}\V_{i}/\V_{i-1} \ast \Hh_{j,1}\bo_{\kk} $. By taking cones and complete diagram (\ref{diagram 8}) to be morphisms between exact triangles, we see that $x$ is an isomorphism.
	
	Similarly, we also obtain an isomorphism $y:\Hh_{j,1}\bo_{\kk} \ast \Delta_{*}\V_{i+1}/\V_{i} \rightarrow \Delta_{*}\V_{i+1}/\V_{i} \ast \Hh_{j,1}\bo_{\kk}$ by using the same argument.
	
	Then we form the following diagram
	\begin{equation} \label{diagram6}
	\xymatrixcolsep{3pc}
	\xymatrix{
		(\Hh_{j, 1}  \ast \Hh_{i,1})\bo_{\kk}  \ar[r]
		& \Hh_{j,1}\bo_{\kk} \ast \Delta_{*}\V_{i+1}/\V_{i}  \ar[r]^{\Hh_{j,1} \ast (0,a_{i})} \ar[d]^{y} 
		&\Hh_{j, 1}\bo_{\kk} \ast \Delta_{*}\V_{i}/\V_{i-1}[1] \ar[d]^{x[1]} \\
		(\Hh_{i, 1}  \ast \Hh_{j,1})\bo_{\kk} \ar[r]  
		& \Delta_{*}\V_{i+1}/\V_{i} \ast \Hh_{j,1}\bo_{\kk}  \ar[r]^{(0,a_{i}) \ast \Hh_{j,1}}
		& \Delta_{*}\V_{i}/\V_{i-1} \ast \Hh_{j,1}\bo_{\kk}[1].
	}
	\end{equation}
	
	Again, by using the same argument in the proof of Theorem \ref{theorem 3}, we can show that the right square in diagram (\ref{diagram6}) is commutative. Thus there is an induced map $z:(\Hh_{j, 1}  \ast \Hh_{i,1})\bo_{\kk} \rightarrow (\Hh_{i, 1}  \ast \Hh_{j,1})\bo_{\kk}$. Since $x$ and $y$ are isomorphisms, using five lemma we obtain $z$ is an isomorphism. Hence $(\Hh_{j, 1}  \ast \Hh_{i,1})\bo_{\kk} \cong (\Hh_{i, 1}  \ast \Hh_{j,1})\bo_{\kk}$.

\end{proof}

The next conditions we prove are conditions (9b) and (10b).

\begin{lemma} (Conditions (9b) and (10b)).  \label{lemma 14} 
	We have the following exact triangle
	\begin{equation*}
	(\Ee_{i+1,s}\ast \Ee_{i,r+1})\bo_{\kk} \rightarrow (\Ee_{i+1,s+1}\ast \Ee_{i,r})\bo_{\kk} \rightarrow (\Ee_{i,r}\ast \Ee_{i+1,s+1})\bo_{\kk}.
	\end{equation*}
\end{lemma}

\begin{proof}
	A simple calculation gives that $(\Ee_{i+1,s+1}\ast \Ee_{i,r})\bo_{\kk}$ is given by
	\[
	i_{1*}((\V_{i}/\V''_{i})^{r} \otimes  (\V_{i+1}/\V''_{i+1})^{s+1})
	\] and $(\Ee_{i,r}\ast \Ee_{i+1,s+1})\bo_{\kk}$ is given by
	\[
	i_{2*}((\V_{i}/\V''_{i})^{r} \otimes  (\V_{i+1}/\V''_{i+1})^{s+1}).
	\]
	
	Here $i_{1}:W_{i+1,i}^{1,1}(\kk) \rightarrow Y(\kk) \times Y(\kk+\alpha_{i}+\alpha_{i+1})$, $i_{2}:W_{i,i+1}^{1,1}(\kk)  \rightarrow Y(\kk) \times Y(\kk+\alpha_{i}+\alpha_{i+1})$ are the natural inclusions with the subvarieties $W_{i+1,i}^{1,1}(\kk) \subset Y(\kk) \times Y(\kk+\alpha_{i}+\alpha_{i+1})$ and $W_{i,i+1}^{1,1}(\kk) \subset Y(\kk) \times Y(\kk+\alpha_{i}+\alpha_{i+1})$ given by
	\begin{align*}
	&W_{i+1,i}^{1,1}(\kk)=\{(V_{\bullet},V_{\bullet}'') \ | \ V''_{i} \subset V_{i}, \ V''_{i+1} \subset V_{i+1}, \ V_{j}=V''_{j} \ \text{for} \ j \neq i,i+1 \}, \\
	&W_{i,i+1}^{1,1}(\kk)=\{(V_{\bullet},V_{\bullet}'') \ | \ V''_{i} \subset V_{i} \subset V''_{i+1} \subset V_{i+1}, \ V_{j}=V''_{j} \ \text{for} \ j \neq i,i+1 \}.
	\end{align*}
	
	It is easy to see that $W_{i,i+1}^{1,1}(\kk) \subset W_{i+1,i}^{1,1}(\kk)$ is a divisor that cut out by the natural section of the line bundle $\mathcal{H}om(\V_{i}/\V''_{i},\V_{i+1}/\V''_{i+1})$. This implies that 
	\begin{equation*}
	\Oo_{W_{i+1,i}^{1,1}(\kk)}(W_{i,i+1}^{1,1}(\kk)) \cong (\V_{i}/\V''_{i})^{\vee} \otimes  \V_{i+1}/\V''_{i+1}.
	\end{equation*}
	
	From the divisor short exact sequence
	\[
	0 \rightarrow \Oo_{W_{i+1,i}^{1,1}(\kk)}(-W_{i,i+1}^{1,1}(\kk))  \cong  \V_{i}/\V''_{i}\otimes  (\V_{i+1}/\V''_{i+1})^{-1} \rightarrow \Oo_{W_{i+1,i}^{1,1}(\kk)} \rightarrow \Oo_{W_{i,i+1}^{1,1}(\kk)} \rightarrow 0.
	\]
	
	Tensoring it with $(\V_{i}/\V''_{i})^r \otimes  (\V_{i+1}/\V''_{i+1})^{s+1}$ we get the following exact triangle
	\begin{equation*}
    (\V_{i}/\V''_{i})^{r+1} \otimes  (\V_{i+1}/\V''_{i+1})^{s}
	\rightarrow  (\V_{i}/\V''_{i})^{r} \otimes  (\V_{i+1}/\V''_{i+1})^{s+1} \\
	\rightarrow \Oo_{W_{i,i+1}^{1,1}(\kk)}\otimes (\V_{i}/\V''_{i})^{r} \otimes  (\V_{i+1}/\V''_{i+1})^{s+1}.
	\end{equation*}
	
	Applying $i_{1*}$ and comparing the kernels, we get 
	\begin{equation*}
	 (\Ee_{i+1,s}\ast \Ee_{i,r+1})\bo_{\kk} \rightarrow (\Ee_{i+1,s+1}\ast \Ee_{i,r})\bo_{\kk} \rightarrow (\Ee_{i,r}\ast \Ee_{i+1,s+1})\bo_{\kk}.
	\end{equation*}
\end{proof}

Next, we verify conditions (11b), (12b). 

\begin{lemma} (Conditions (11b), (12b)). \label{lemma 15}
	\begin{equation*}
	(\Psi^{+}_{i} \ast\Ee_{i+1,r})\bo_{\kk} \cong  ( \Ee_{i+1,r-1} \ast \Psi^{+}_{i})\bo_{\kk}[1].
	\end{equation*}
\end{lemma}

\begin{proof}
	An easy calculation gives that $(\Psi^{+}_{i} \ast\Ee_{i+1,r})\bo_{\kk}$ is given by 
	\begin{equation*}
	\iota_{*}((\V_{i+1}/\V'_{i+1})^{r} \otimes \det(\V'_{i+1}/\V_{i}))[2-k_{i+1}].
	\end{equation*} while $(\Ee_{i+1,r-1}\ast \Psi^{+}_{i})\bo_{\kk}$ is given by 
	\begin{equation*}
	\iota_{*} (\V_{i+1}/\V'_{i+1})^{r-1} \otimes \det(\V_{i+1}/\V_{i})[1-k_{i+1}].
	\end{equation*}
	
	Here $\iota: W^{1}_{i+1}(\kk) \rightarrow Y(\kk) \times Y(\kk+\alpha_{i+1})$ is the inclusion. Note that the following short exact sequence
	\begin{equation*}
	0 \rightarrow \V'_{i+1}/\V_{i} \rightarrow \V_{i+1}/\V_{i} \rightarrow \V_{i+1}/\V'_{i+1} \rightarrow 0
	\end{equation*} gives that $\det(\V'_{i+1}/\V_{i}) \otimes \V_{i+1}/\V'_{i+1} \cong \det(\V_{i+1}/\V_{i})$.  
	
	Combining the above, we get  
	\[
	\iota_{*} (\V_{i+1}/\V'_{i+1})^{r} \otimes \det(\V'_{i+1}/\V_{i})[2-k_{i+1}]
	\cong \iota_{*} (\V_{i+1}/\V'_{i+1})^{r-1} \otimes \det(\V_{i+1}/\V_{i})[2-k_{i+1}]
	\]
	which implies that $(\Psi^{+}_{i} \ast\Ee_{i+1,r})\bo_{\kk} \cong (\Ee_{i+1,r-1}\ast \Psi^{+}_{i})\bo_{\kk}[1]$.
\end{proof}

Finally, we prove conditions (13b), (14b). 

\begin{lemma} (Conditions (13b), (14b)).  \label{lemma 16}
	We have the following exact triangles
	\begin{equation*}
	(\Hh_{i,1} \ast \Ee_{i+1,r})\bo_{\kk} \rightarrow (\Ee_{i+1,r} \ast \Hh_{i,1})\bo_{\kk} \rightarrow \Ee_{i+1,r+1}\bo_{\kk}.
	\end{equation*}
\end{lemma}

\begin{proof}
	First, we know that the kernel $\Hh_{i,1}\bo_{\kk}$ is determined by the exact triangle
	\begin{equation*}
	\Delta_{*} \V_{i}/\V_{i-1} \rightarrow \Hh_{i,1}\bo_{\kk} \rightarrow \Delta_{*} \V_{i+1}/\V_{i},
	\end{equation*} and we calculate the convolution $(\Ee_{i+1,r}\ast \Hh_{i,1} )\bo_{\kk}$ which is the following exact triangle
	\begin{equation} \label{ses15}
	\Ee_{i+1,r}\bo_{\kk} \ast (\Delta_{*} \V_{i}/\V_{i-1}) \rightarrow (\Ee_{i+1,r}\ast \Hh_{i,1} )\bo_{\kk} \rightarrow \Ee_{i+1,r}\bo_{\kk} \ast (\Delta_{*} \V_{i+1}/\V_{i}).
	\end{equation}
	
	We denote the elements in the space that using to calculate the convolution of kernels in (\ref{ses15}) as follows
	\[
	Y(\kk) \times Y(\kk) \times Y(\kk+\alpha_{i+1})=\{(V_{\bullet},V_{\bullet}'',V_{\bullet}')\}
	\] and $\Ee_{i+1,r}\bo_{\kk} \coloneqq \iota_{*} (\V''_{i+1}/\V'_{i+1})^{r}$ where $\iota:W^{1}_{i+1}(\kk) \rightarrow Y(\kk) \times Y(\kk+\alpha_{i+1})$ is the natural inclusion.
	
	A simple calculation gives us that the exact triangle (\ref{ses15}) becomes the following
	\[
	\iota_{*} (\V_{i}/\V_{i-1} \otimes (\V_{i+1}/\V'_{i+1})^{r}) \rightarrow (\Ee_{i+1,r}\ast \Hh_{i,1} )\bo_{\kk} \rightarrow \iota_{*} (\V_{i+1}/\V_{i} \otimes (\V_{i+1}/\V'_{i+1})^{r}).
	\]
	
	Next, using similar argument $(\Hh_{i,1}\ast \Ee_{i+1,r})\bo_{\kk}$ is in the following exact triangle
	\begin{equation} \label{ses16}
    \iota_{*} (\V_{i}/\V_{i-1} \otimes (\V_{i+1}/\V'_{i+1})^{r}) \rightarrow ( \Hh_{i,1} \ast \Ee_{i+1,r} )\bo_{\kk} \rightarrow \iota_{*} (\V'_{i+1}/\V_{i} \otimes (\V_{i+1}/\V'_{i+1})^{r}).
	\end{equation}
	
	Similarly, we denote the elements in the space that using to calculate the convolution of kernels in (\ref{ses16}) as follows
	\[
	Y(\kk) \times Y(\kk+\alpha_{i+1}) \times Y(\kk+\alpha_{i+1})=\{(V_{\bullet},V_{\bullet}''',V_{\bullet}')\}.
	\]

	On $W^{1}_{i+1}(\kk)$, we have the following short exact sequence
	\[
	0 \rightarrow \V'_{i+1}/\V_{i} \rightarrow \V_{i+1}/\V_{i} \rightarrow \V_{i+1}/\V'_{i+1} \rightarrow 0.
	\]
	
	Then we have the following diagram of exact triangles 
	\begin{equation*}
	\begin{tikzcd}
	\iota_{*} (\V_{i}/\V_{i-1} \otimes (\V_{i+1}/\V'_{i+1})^{r}) \arrow[d] \arrow[r] &(\Hh_{i,1} \ast \Ee_{i+1,r} )\bo_{\kk}  \arrow[r] &	\iota_{*} (\V'_{i+1}/\V_{i} \otimes (\V_{i+1}/\V'_{i+1})^{r})  \arrow[d] \\
	\iota_{*} (\V_{i}/\V_{i-1} \otimes (\V_{i+1}/\V'_{i+1})^{r}) \arrow[r] &(\Ee_{i+1,r}\ast \Hh_{i,1} )\bo_{\kk}  \arrow[r] &	\iota_{*}( \V_{i+1}/\V_{i} \otimes (\V_{i+1}/\V'_{i+1})^{r})  
	\end{tikzcd}
	\end{equation*} where the left vertical map is just the identity, and the right vertical map is induced by the inclusion map $\V'_{i+1}/\V_{i} \rightarrow \V_{i+1}/\V_{i}$.

	Using the same argument as in the proof of Theorem \ref{theorem 3}, there induces an map $(\Hh_{i,1} \ast \Ee_{i+1,r} )\bo_{\kk}  \rightarrow (\Ee_{i+1,r}\ast \Hh_{i,1} )\bo_{\kk}$. Finally, completing the diagram to exact triangles, we get 
	\begin{equation*}
	\begin{tikzcd}
	\iota_{*} (\V_{i}/\V_{i-1} \otimes (\V_{i+1}/\V'_{i+1})^{r}) \arrow[d] \arrow[r] &(\Hh_{i,1} \ast \Ee_{i+1,r} )\bo_{\kk}  \arrow[r] \arrow[d] &	\iota_{*} (\V'_{i+1}/\V_{i} \otimes (\V_{i+1}/\V'_{i+1})^{r})  \arrow[d] \\
	\iota_{*} (\V_{i}/\V_{i-1} \otimes (\V_{i+1}/\V'_{i+1})^{r})  \arrow[r] \arrow[d] &(\Ee_{i+1,r}\ast \Hh_{i,1} )\bo_{\kk}  \arrow[r] \arrow[d] &	\iota_{*} (\V_{i+1}/\V_{i} \otimes (\V_{i+1}/\V'_{i+1})^{r}) \arrow[d] \\
	0 \arrow[r] &  	\iota_{*} (\V_{i+1}/\V'_{i+1})^{r+1} \arrow[r] & 	\iota_{*} (\V_{i+1}/\V'_{i+1})^{r+1}
	\end{tikzcd}
	\end{equation*}
	and $\iota_{*}(\V_{i+1}/\V'_{i+1})^{r+1}=\Ee_{i+1,r+1}\bo_{\kk}$.
	
	Thus we get the exact triangle
	\[
	 (\Hh_{i,1} \ast \Ee_{i+1,r})\bo_{\kk} \rightarrow (\Ee_{i+1,r} \ast \Hh_{i,1})\bo_{\kk} \rightarrow \Ee_{i+1,r+1}\bo_{\kk}. 
	\]
\end{proof}

Combing the above results in this section and Theorem \ref{theorem 2}, we prove Theorem \ref{theorem 1}.

\subsection{Grothendieck groups and other (categorical) relations} \label{subsection 5.4}
In this subsection, we examine the categorical action of the shifted $q=0$ affine algebra that we construct in Theorem \ref{theorem 1} at the level of Grothendieck groups. 

By Theorem \ref{theorem 1}, we have a categorical action of $\dot{\Uu}_{0,N}(L\SL_{n})$ on $\bigoplus_{\kk}\Dd^b(Fl_{\kk}(\CC^N))$ where all the generators act on $\bigoplus_{\kk}\Dd^b(Fl_{\kk}(\CC^N))$ by using FM transforms with FM kernels defined in Definition \ref{definition 3}. Passing to the Grothendieck groups, we obtain an action of $\dot{\Uu}_{0,N}(L\SL_{n})$ on $\bigoplus_{\kk}K(Fl_{\kk}(\CC^N))$ and all the generators act on $\bigoplus_{\kk}K(Fl_{\kk}(\CC^N))$ by using K-theoretic FM transforms. For example,
\begin{equation*}
e_{i,r}1_{\kk}:K(Fl_{\kk}(\CC^N)) \rightarrow K(Fl_{\kk+\alpha_{i}}(\CC^N)), \ x \mapsto \pi_{2*}(\pi^{*}_{1}(x) \otimes [\iota(\kk)_{*}(\V_{i}/\V'_{i})^{r}])
\end{equation*} where we denote $[\iota(\kk)_{*}(\V_{i}/\V'_{i})^{r}]$ to be the class of $\iota(\kk)_{*}(\V_{i}/\V'_{i})^{r}$ in $K(Fl_{\kk}(\CC^N) \times Fl_{\kk+\alpha_{i}}(\CC^N))$. Moreover, all the isomorphisms in the categorical relations become equalities,  e.g. $(\Psi_{i}^{+} \ast \Ee_{i,r})\bo_{\kk} \cong (\Ee_{i,r} \ast \Psi_{i}^{+})\bo_{\kk}[-1]$ in Lemma \ref{lemma 3} becomes $\psi_{i}^{+}e_{i,r}1_{\kk}=-e_{i,r}\psi_{i}^{+}1_{\kk}$, which is one of the relation in (\ref{U06}). Among all the relations in $\dot{\Uu}_{0,N}(L\SL_{n})$, the most interesting one is the commutator relation, i.e. (\ref{U09}). We will mainly focus on understanding the commutator relation $[e_{i,r},f_{i,s}]1_{\kk}$ on $K(Fl_{\kk}(\CC^N))$. 

First, we restrict to the $\SL_{2}$ case, i.e. we consider the commutator relation $[e_{r},f_{s}]1_{(k,N-k)}$ on $K(\GG(k,N))$. From (\ref{U09}) we know that $[e_{r},f_{s}]1_{(k,N-k)}$ is equal to one of the following $0$, $\psi^{+}1_{(k,N-k)}$, $\psi^{+}h_{1}1_{(k,N-k)}$, $-\psi^{-}1_{(k,N-k)}$, and $-\psi^{-}h_{-1}1_{(k,N-k)}$ for $-k-1 \leq r+s \leq N-k+1$. Thus we have to know how $\psi^{\pm}1_{(k,N-k)}$ and $h_{\pm 1}1_{(k,N-k)}$ act on $K(\GG(k,N))$.

Since we know the action of $\psi^{\pm}1_{(k,N-k)}$ on derived categories from Definition \ref{definition 3}, it is easy to see that at the level of Grothendieck groups
\begin{align*}
&\psi^{+}1_{(k,N-k)}:K(\GG(k,N)) \rightarrow K(\GG(k,N)), \ x \mapsto x \otimes (-1)^{N-k-1}[\det(\CC^N/\V)], \\
&\psi^{-}1_{(k,N-k)}:K(\GG(k,N)) \rightarrow K(\GG(k,N)), \ x \mapsto x \otimes (-1)^{k-1}[\det(\V)^{-1}],
\end{align*} where $\V$ is the tautological bundle of rank $k$ on $\GG(k,N)$. While for $h_{\pm 1}1_{(k,N-k)}$, although it is not so directly to see the induced actions on $K(\GG(k,N))$, we can use Theorem \ref{theorem'} to help us.
More precisely, the FM kernel $\Hh_{1}\bo_{(k,N-k)}$ fits in the following exact triangle
\begin{equation*}
    \Delta_{*}\V \rightarrow \Hh_{1}\bo_{(k,N-k)} \rightarrow \Delta_{*}\CC^N/\V, 
\end{equation*} and passing to the Grothendieck group we have the following equality in $K(\GG(k,N) \times \GG(k,N))$
\begin{equation*}
[\Hh_{1}\bo_{(k,N-k)}]=[\Delta_{*}\V]+[\Delta_{*}\CC^N/\V]=[\Delta_{*}(\CC^N/\V\oplus\V)]=[\Delta_{*}\CC^N].
\end{equation*} Thus the action of $h_{1}1_{(k,N-k)}$ is given by 
\begin{equation*}
   h_{1}1_{(k,N-k)}:K(\GG(k,N)) \rightarrow K(\GG(k,N)), \ x \mapsto x \otimes [\CC^N] 
\end{equation*} or equivalently multiplication by $N$.

A similar argument also shows that 
$\Hh_{-1}\bo_{(k,N-k)}$ fits in the following exact triangle
\begin{equation*}
    \Delta_{*}(\CC^N/\V)^{\vee} \rightarrow \Hh_{-1}\bo_{(k,N-k)} \rightarrow \Delta_{*}(\V)^{\vee}, 
\end{equation*} and the action of $h_{-1}1_{(k,N-k)}$ is given by 
\begin{equation*}
   h_{-1}1_{(k,N-k)}:K(\GG(k,N)) \rightarrow K(\GG(k,N)), \ x \mapsto x \otimes [(\CC^N)^{\vee}]. 
\end{equation*}

Summarizing, we obtain the commutator relation
\begin{equation*}
[e_{r},f_{s}]1_{(k,N-k)}= \begin{cases}
	\otimes (-1)^{N-k-1}[\det(\CC^N/\V)][\CC^N] & \text{if} \  r+s=N-k+1 \\
	\otimes (-1)^{N-k-1}[\det(\CC^N/\V)] & \text{if} \ r+s=N-k \\
	0 & \text{if} \  -k+1 \leq r+s \leq N-k-1 \\
	\otimes (-1)^{k}[\det(\V)^{-1}] & \text{if} \ r+s=-k \\
	\otimes (-1)^{k}[\det(\V)^{-1}][(\CC^N)^{\vee}]& \text{if} \ r+s=-k-1
	\end{cases}.
\end{equation*}

Next, note that although the presentation of $\dot{\Uu}_{0,N}(L\SL_{2})$ (or more generally $\dot{\Uu}_{0,N}(L\SL_{n})$) we choose to work makes the generators $e_{r}1_{(k,N-k)}$, $f_{s}1_{(k,N-k)}$ defined only for $r$, $s$ within certain ranges (and the main reason is to give a definition of its categorical action), when we consider its action on $K(\GG(k,N))$ or $\Dd^b(\GG(k,N))$, it is still natural to consider $e_{r}1_{(k,N-k)}$, $f_{s}1_{(k,N-k)}$ for $r,\ s \in \ZZ$. 

We try to understand the categorical commutator relation between ${\E}_{r}{\F}_{s}\bo_{(k,N-k)}$ and ${\F}_{s}{\E}_{r}\bo_{(k,N-k)}$ (or more precisely $(\Ee_{r} \ast \Ff_{s})\bo_{(k,N-k)}$ and $(\Ff_{s} \ast \Ee_{r})\bo_{(k,N-k)}$) say, for $r+s \geq N-k+2$. Similarly like Proposition \ref{proposition 5}, it suffices to compare $(\Ee_{0} \ast \Ff_{s})\bo_{(k,N-k)}$ and $(\Ff_{s} \ast \Ee_{0})\bo_{(k,N-k)}$  for $s \geq N-k+2$. 

Consider first $s=N-k+2$, then like the exact triangle (\ref{ses10}) we have the following exact triangle
\begin{equation}  \label{ses'} 
	p_{*} (\V''/\V''')^{N-k+2} \rightarrow p_{*}(\V'/\V)^{N-k+2} \rightarrow p_{*}(\Oo_{(N-k+2)D} \otimes (\V'/\V)^{N-k+2})
\end{equation} by applying $p_{*}$ to (\ref{ses8}) with $n= N-k+2$. Again, a similar argument as in the proof of Proposition \ref{proposition 5} tells us that 
\begin{equation*}
    p_{*}(\Oo_{(N-k+2)D} \otimes (\V'/\V)^{N-k+2}) \cong p_{*}(\Oo_{(N-k+1)D} \otimes (\V'/\V)^{N-k+2}) \cong ...
    \cong p_{*}(\Oo_{3D} \otimes (\V'/\V)^{N-k+2}),
\end{equation*} and applying $t_{*}$ to (\ref{ses'}) we obtain the following exact triangle 
\begin{equation}
    (\Ff_{N-k+2} \ast \Ee_{0})\bo_{(k,N-k)} \rightarrow (\Ee_{0} \ast \Ff_{N-k+2})\bo_{(k,N-k)}  \rightarrow t_{*}p_{*}(\Oo_{3D} \otimes (\V'/\V)^{N-k+2}). \label{ocr}
\end{equation}

Thus the first question is we have to study $t_{*}p_{*}(\Oo_{3D} \otimes (\V'/\V)^{N-k+2})$. Tensoring (\ref{ses9}) with $n=3$ by $(\V'/\V)^{N-k-1}$, we obtain 
\begin{equation*}
\Oo_{D} \otimes (\V'/\V)^{N-k} \otimes (\V/\V''')^{2}  \rightarrow  \Oo_{3D}\otimes (\V'/\V)^{N-k+2} \rightarrow \Oo_{2D}  \otimes  (\V'/\V)^{N-k+2}
\end{equation*} applying $t_{*}p_{*}$, using the projection formula and Proposition \ref{proposition 3}, we get the following exact triangle
\begin{equation} \label{et1}
    \Delta_{*}\Sym^{2}(\V) \otimes \det(\CC^N/\V)[1+k-N]\rightarrow t_{*}p_{*}(\Oo_{3D}\otimes(\V'/\V)^{N-k+2}) \rightarrow  t_{*}p_{*}(\Oo_{2D}\otimes(\V'/\V)^{N-k+2}).
\end{equation} 

Next, tensoring  (\ref{ses9}) with $n=2$ by $(\V'/\V)^{N-k}$, we obtain 
\begin{equation*}
\Oo_{D} \otimes (\V'/\V)^{N-k+1} \otimes (\V/\V''')  \rightarrow \Oo_{2D}\otimes (\V'/\V)^{N-k+2} \rightarrow \Oo_{D}\otimes  (\V'/\V)^{N-k+2}
\end{equation*} again applying $t_{*}p_{*}$, using the projection formula and Proposition \ref{proposition 3}, we get the following exact triangle
\begin{align} \label{et2}
    \begin{split}
    &\Delta_{*}\V \otimes \CC^N/\V \otimes\det(\CC^N/\V)[1+k-N]\rightarrow t_{*}p_{*}(\Oo_{2D}\otimes(\V'/\V)^{N-k+2})  \\
    &\rightarrow  \Delta_{*}\Sym^{2}(\CC^N/\V) \otimes \det(\CC^N/\V)[1+k-N].
    \end{split}
\end{align}

Thus, understanding $t_{*}p_{*}(\Oo_{3D} \otimes (\V'/\V)^{N-k+2})$ is equivalent to understanding the two exact triangles (\ref{et1}) and (\ref{et2}). Moreover, if we define 
\begin{equation} \label{h_2}
\Hh_{2}\bo_{(k,N-k)}:=(\Psi^{+}\bo_{(k,N-k)})^{-1} \ast [t_{*}p_{*}(\Oo_{3D} \otimes (\V'/\V)^{N-k+2})] \in \Dd^b(\GG(k,N)\times \GG(k,N))
\end{equation} which can be thought as a FM kernel for a certain (undefined) functor $\Hhh_{2}\bo_{(k,N-k)}$. Then $\Hh_{2}\bo_{(k,N-k)}$ is build up from the following two exact triangles in $\Dd^b(\GG(k,N) \times \GG(k,N))$
\begin{align}
&\Delta_{*}\Sym^{2}(\V) \rightarrow \Hh_{2}\bo_{(k,N-k)} \rightarrow  (\Psi^{+}\bo_{(k,N-k)})^{-1} \ast [t_{*}p_{*}(\Oo_{2D}\otimes(\V'/\V)^{N-k+2})], \label{ett1} \\
&\Delta_{*}\V \otimes \CC^N/\V \rightarrow (\Psi^{+}\bo_{(k,N-k)})^{-1} \ast [t_{*}p_{*}(\Oo_{2D}\otimes(\V'/\V)^{N-k+2})] \rightarrow  \Delta_{*}\Sym^{2}(\CC^N/\V). \label{ett2}
\end{align}

Another natural question to ask is what is the categorical commutator relation between ${\Hhh}_{2}{\E}_{r}\bo_{(k,N-k)}$ and ${\E}_{r}{\Hhh}_{2}\bo_{(k,N-k)}$ (or exact triangle relates $\Hh_{2} \ast {\Ee}_{r}\bo_{(k,N-k)}$ and ${\Ee}_{r} \ast \Hh_{2}\bo_{(k,N-k)}$). Due to the difficulty and complication, we will like to address those questions in the future and similar for other comparisons between ${\Hhh}_{n}{\E}_{r}\bo_{(k,N-k)}$ and ${\E}_{r}{\Hhh}_{n}\bo_{(k,N-k)}$ with $n \geq 3$.

However, the above discussions can help us to understand what exactly are the commutators for the loop generators $e_{r}1_{(k,N-k)}$, $f_{s}1_{(k,N-k)}$ at the level of Grothendieck groups. From (\ref{ett1}) and (\ref{ett2}), we have the following equality in $K(\GG(k,N) \times \GG(k,N))$
\begin{align*}
[\Hh_{2}\bo_{(k,N-k)}] &= [\Delta_{*}\Sym^{2}(\V)] + [(\Psi^{+}\bo_{(k,N-k)})^{-1} \ast t_{*}p_{*}(\Oo_{2D}\otimes(\V'/\V)^{N-k+2}) ] \\
&=[\Delta_{*}\Sym^{2}(\V)] + [\Delta_{*}\V \otimes \CC^N/\V] +[\Delta_{*}\Sym^{2}(\CC^N/\V)] \\
& = [\Sym^2(\V \oplus \CC^N/\V)]=[\Sym^{2}(\CC^N)].
\end{align*}  Thus by the definition of $\Hh_{2}\bo_{(k,N-k)}$ and the exact triangle (\ref{ocr}), we get that at the level of Grothendieck group 
\begin{equation*}
    [e_{r},f_{s}]1_{(k,N-k)}=\otimes (-1)^{N-k-1} [\Sym^2(\CC^N)][\det(\CC^N/\V)], \ \text{if} \ r+s=N-k+2.
\end{equation*}

We can generalize this to other cases where $r+s \geq N-k+3$ and similarly $r+s \leq -k-2$ by using the same argument as above. As a conclusion, we have the following result which stated as a corollary.

\begin{corollary} \label{corollary cr1}
The commutator relations at the level of Grothendieck groups $K(\GG(k,N))$ for all $e_{r}1_{(k,N-k)}$, $f_{s}1_{(k,N-k)}$ with $r,s \in \ZZ$ is given by 
\begin{equation*}
[e_{r},f_{s}]1_{(k,N-k)}= \begin{cases}
	\otimes (-1)^{N-k-1}[\det(\CC^N/\V)][\Sym^{r+s-N+k}(\CC^N)] & \text{if} \  r+s \geq N-k \\
	0 & \text{if} \  -k+1 \leq r+s \leq N-k-1 \\
	\otimes (-1)^{k}[\det(\V)^{-1}][\Sym^{-r-s-k}(\CC^N)^{\vee}]& \text{if} \ r+s \leq -k
	\end{cases}.
\end{equation*}
\end{corollary}

As a direct generalization to the $\SL_{n}$ case, we obtain 

\begin{corollary}\label{corollary cr2}
The commutator relations at the level of Grothendieck groups $K(Fl_{\kk}(\CC^N))$ for all $e_{i,r}1_{\kk}$, $f_{i,s}1_{\kk}$ with $r,s \in \ZZ$ is given by 
\begin{equation*}
[e_{i,r},f_{i,s}]1_{\kk}= \begin{cases}
	\otimes (-1)^{k_{i+1}-1}[\det(\V_{i+1}/\V_{i})][\Sym^{r+s-k_{i+1}}(\V_{i+1}/\V_{i-1})] & \text{if} \  r+s \geq k_{i+1} \\
	0 & \text{if} \  -k_{i}+1 \leq r+s \leq k_{i+1}-1 \\
	\otimes (-1)^{k_{i}}[\det(\V_{i}/\V_{i-1})^{-1}][\Sym^{-r-s-k_{i}}(\V_{i+1}/\V_{i-1})^{\vee}]& \text{if} \ r+s \leq -k_{i}
	\end{cases}.
\end{equation*}
\end{corollary}

Finally, we give a remark.

\begin{remark}
It would be interesting to find an integral basis for $K(\GG(k,N))$ (or $K(Fl_{\kk}(\CC^N))$) and compute the matrix elements of the generators $e_{r}1_{(k,N-k)}$, $f_{s}1_{(k,N-k)}$ (or $e_{i,r}1_{\kk}$, $f_{i,s}1_{\kk}$).
\end{remark}

\section{Categorical action of the $q=0$ affine Hecke algebras} \label{section 6}

In this final section, we provide an application of the categorical action of the shifted $q=0$ affine algebra (more precisely Theorem \ref{theorem 1}), which is the categorical action of the $q=0$ affine Hecke algebra on the derived category of coherent sheaves on the full flag variety. The main idea is to interpret the Demazure operators as elements in the shifted $q=0$ affine algebra.

\subsection{$q=0$ affine Hecke algebras}

We work on the type $A$ case, and begin with the definition of the affine Hecke algebras.

\begin{definition} \label{definition 4} 
	Let $q \in \CC^*$ with $q \neq 1$. The affine Hecke algebra $\Hh_{N}(q)$ is defined to be the unital associative $\CC$-algebra generated by the elements $T_{1},...,T_{N-1},X_{1}^{\pm 1},...,X_{N}^{\pm 1}$ subject to the following relations
	\begin{equation}  \tag{H1}   \label{H1}
	(T_i-q)(T_{i}+1)=0,
	\end{equation}
	\begin{equation}  \tag{H2}   \label{H2}
	T_{i}T_{j}=T_{j}T_{i} \ \text{if} \ |i-j| \geq 2,
	\end{equation}
	\begin{equation}  \tag{H3}   \label{H3}
	T_{i}T_{j}T_{i}=T_{j}T_{i}T_{j} \ \text{if} \ |i-j|=1,
	\end{equation}
	\begin{equation}  \tag{H4}   \label{H4}
	X_{i}X_{i}^{-1}=X_{i}^{-1}X_{i}=1,
	\end{equation}
	\begin{equation}  \tag{H5}   \label{H5}
	X_{i}X_{j}=X_{j}X_{i} \ \text{for all } i,j,
	\end{equation}
	\begin{equation}  \tag{H6}   \label{H6}
	T_{i}X_{j}=X_{j}T_{i} \ \text{if} \ j \neq i,i+1,
	\end{equation}
	\begin{equation}  \tag{H7}   \label{H7}
	T_{i}X_{i}=X_{i+1}T_{i}-(q-1)X_{i+1},
	\end{equation}
	\begin{equation}  \tag{H8}   \label{H8}
	T_{i}X_{i+1}=X_{i}T_{i}+(q-1)X_{i+1}.
	\end{equation}
\end{definition}

\begin{remark}
Note that the affine Hecke algebra $\Hh_{N}(q)$ can be defined over $\CC[q,q^{-1}]$ for $q$ an indeterminate.
\end{remark}

Then we define the $q=0$ affine Hecke algebra.

\begin{definition} \label{definition 5} 
	The $q=0$ affine Hecke algebra $\Hh_{N}(0)$ is defined to be the unital associative $\CC$-algebra generated by the elements $T_{1},...,T_{N-1},X_{1}^{\pm 1},...,X_{N}^{\pm 1}$ subject to the following relations
	\begin{equation}  \tag{H01}   \label{H01}
	T_{i}^2=T_{i},
	\end{equation}
	\begin{equation}  \tag{H02}   \label{H02}
	T_{i}T_{j}=T_{j}T_{i}, \ \text{if} \ |i-j| \geq 2,
	\end{equation}
	\begin{equation}  \tag{H03}   \label{H03}
	T_{i}T_{j}T_{i}=T_{j}T_{i}T_{j}, \ \text{if} \ |i-j|=1,
	\end{equation}
	\begin{equation}  \tag{H04}   \label{H04}
	X_{i}X_{i}^{-1}=X_{i}^{-1}X_{i}=1,
	\end{equation}
	\begin{equation}  \tag{H05}   \label{H05}
	X_{i}X_{j}=X_{j}X_{i}, \ \text{for all } i,j,
	\end{equation}
	\begin{equation}  \tag{H06}   \label{H06}
	T_{i}X_{j}=X_{j}T_{i} \ \text{if} \ j \neq i,i+1,
	\end{equation}
	\begin{equation}  \tag{H07}   \label{H07}
	X_{i+1}T_{i}=T_{i}X_{i}+X_{i+1},
	\end{equation}
	\begin{equation}  \tag{H08}   \label{H08}
	X_{i}T_{i}=T_{i}X_{i+1}-X_{i+1}.
	\end{equation}
\end{definition}

We give some remarks about the $q=0$ affine Hecke algebra.

\begin{remark}
The $q=0$ affine Hecke algebra is the affine Hecke algebra with variable $q$ specialized at $0$. Note that Here we make the relations simple by twisting the automorphism $T_{i} \mapsto -T_{i}$ and still denote the generators $T_{i}$ by abusing notations. For example, the relation (\ref{H01}) is the quadratic relation (\ref{H1})  $(T_{i}+q)(T_{i}-1)=0$ at $q=0$. 
\end{remark}

\begin{remark}
Note that there is also an algebra called \textit{the (affine) asymptotic Hecke algebra} that defined by Lusztig \cite{Lu2} which can also be thought as a limit of the affine Hecke algebra when $q$ tends to $0$. However, the definition is much more complicated since it uses the theory Kazhdan-Lusztig basis \cite{KL-1}. It would be interesting to see the relations between these two algebras.
\end{remark}

Next we recall the action of $\Hh_{N}(0)$ on the Grothendieck group of the full flag variety. Let $G=\SLL_{N}(\CC)$ and $B \subset G$ be the Borel of upper triangular matrices. Then the full flag variety is defined by
\begin{equation} \label{eq ffl}
G/B:=\{0 \subset V_{1} \subset V_{2} \subset ... \subset V_{N}=\CC^N \ | \ \dim V_{k}=k \ \text{for} \ \text{all} \ k  \}
\end{equation} and similarly we have the partial flag varieties 
\begin{equation} \label{eq pfl}
G/P_{i}:=\{0 \subset V_{1} \subset V_{2} \subset ...V_{i-1} \subset V_{i+1} \subset ... V_{N}=\CC^N \ | \ \dim V_{k} =k \ \text{for} \ k \neq i \}	
\end{equation} where $P_{i}$ is a minimal parabolic subgroup for $1 \leq i \leq N-1$. 

We explain that there is an action of the $q=0$ affine Hecke algebra $\Hh_{N}(0)$ on the K-theory (or complexified Grothendieck group) of the full flag variety, i.e. $K(G/B)$.

The construction is as follows. On $G/B$, we denote $\V_{i}$ to be the tautological bundle of rank $i$. Let $a_{i}=[\V_{i}/\V_{i-1}]$ be the class of the tautological line bundle $\V_{i}/\V_{i-1}$ in $K(G/B)$, where $1 \leq i \leq N$. Note that since $\V_{i}/\V_{i-1}$ is a line bundle, it is invertible and thus $a^{-1}_{i}=[\V_{i}/\V_{i-1}]^{-1}$ exists in $K(G/B)$. 

$K(G/B)$ admits a presentation which is called the Borel presentation.
\begin{equation} \label{eq b}
K(G/B) \cong \CC[a_1^{\pm},...,a_N^{\pm}]/ \left \langle e_i- {N \choose i} \right \rangle	
\end{equation} where $\langle e_i- {N\choose i} \rangle$ is the ideal generated by $\{e_i- {N\choose i}\}_{i=1}^{N}$, and $e_i$ is the $i$-th symmetric polynomial in $a_1, a_2,...,a_{N}$.

We have the natural maps $\pi_{i}:G/B \rightarrow G/P_{i}$ by forgetting the $i$-dimensional vector spaces $V_{i}$ for all $ 1 \leq i \leq N-1$. It is easy to see that $\pi_{i}$ is a $\PP^1$-fibration for all $1 \leq i \leq N-1$.

Those maps induce maps between the Grothendieck groups, they are the pushforwards $\pi_{i_*}:K(G/B) \rightarrow K(G/P_{i})$ and pullbacks $\pi_{i}^*:K(G/P_{i}) \rightarrow K(G/B)$, $1 \leq i \leq N-1$. The fundamental construction of the push-pull operators $T_{i}:=\pi_{i}^*\pi_{i_*}:K(Fl) \rightarrow K(Fl)$ for all $1 \leq i \leq N-1$. Those operators $\{T_{i}\}^{N-1}_{i=1}$ are called the Demazure operators (or divided difference operators). Under the presentation (\ref{eq b}), $T_{i}$ has the following explicit description
\begin{equation*}
T_{i}=\frac{a_{i+1}-a_is_{i}}{a_{i+1}-a_i}	
\end{equation*} where $s_i$ is the simple reflection that permute $a_i$ and $a_{i+1}$ for all $ 1 \leq i \leq N-1$. Finally, we define the operators $X_{j}:K(G/B) \rightarrow K(G/B)$ to be multiplication by $a_{j}$ for all $1 \leq j \leq N$.

Then it is easy to check that those operators $T_{i}$, $X_{j}^{\pm 1}$ defined in such way satisfy the relations of the $q=0$ affine Hecke algebra $\Hh_{N}(0)$. Thus we get 
an action of $\Hh_{N}(0)$ on $K(G/B)$.

It is natural to lift this action to the categorical level. More precisely, we replace $K(G/B)$ by the derived category of coherent sheaves $\Dd^b(G/B)$ and $T_{i}$, $X_{j}$ by certain functors $\TT_{i}$, $\mathbb{X}_{j}$, resp.

So we need to define our functors $\TT_{i}:\Dd^b(G/B) \rightarrow \Dd^b(G/B)$ and $\mathbb{X}_{j}:\Dd^b(G/B) \rightarrow \Dd^b(G/B)$ that satisfy the relations in the categorical setting. Again, we would use the tools of FM transforms/kernels from Section \ref{section 4} to define the categorical action.

From the definition in K-theory we have $T_{i}:=\pi_{i}^*\pi_{i_*}$, and we know that they are defined by using the forgetting maps $\pi_{i}:G/B \rightarrow G/P_{i}$. Since $\pi_{i}$ also induce functors on derived categories of coherent sheaves, i.e. the derived pushforward $\pi_{i*}:\Dd^b(G/B) \rightarrow \Dd^b(G/P_{i})$ and the derived pullback $\pi_{i}^{*}:\Dd^b(G/P_{i}) \rightarrow \Dd^b(G/B)$,  it is natural to define the functors $\TT_{i}:=\pi_{i}^*\pi_{i_*}:\Dd^b(G/B) \rightarrow \Dd^b(G/B)$ for all $1 \leq i \leq N-1$.

In fact, $\TT_{i}$ are also FM transforms. Recall that for a morphism $f:X \rightarrow Y$ between smooth projective varieties, the derived pushforward $f_{*}:\Dd^b(X) \rightarrow \Dd^b(Y)$ is isomorphic to the FM transform with kernel $\Oo_{\Gamma_{f}}$, where $\Gamma_{f}=(id \times f)(X)$ is the graph of $f$ in $X \times Y$. Similarly, the derived pullback $f^*$ is isomorphic to the FM transform with kernel $\Oo_{(f \times id)(X)}$ in $\Dd^b(Y \times X)$. 

Thus a simple calculation of convolution of FM kernels shows that 
\begin{equation*}
\TT_{i} \cong \Phi_{\Oo_{(\pi_{i} \times id)(G/B)}} \circ \Phi_{\Oo_{(id \times \pi_i)(G/B)}} \cong   \Phi_{\Tt_{i}} 
\end{equation*} where $\Tt_{i} \cong \Oo_{G/B \times_{G/P_{i}} G/B}$ for all $1 \leq i \leq N-1$. 

The fibred product variety $G/B \times_{G/P_{i}} G/B$ is called the Bott-Samelson variety for $s_{i}$ (see Section 5.1 in \cite{AK2}). Thus the FM transform with kernel given by structure sheaf of the Bott-Samelson variety for $s_{i}$ categorifies the Demazure operator $T_{i}$ for all $1 \leq i \leq N-1$.

Next, we define the functors $\mathbb{X}_{j}$. Since on K-theory $X_{j}$ is the map given by multiplication with the element $a_{j}=[\V_{j}/\V_{j-1}]$, it is natural to define $\mathbb{X}_{j}$ to be the functor that given by tensoring the line bundle $\V_{j}/\V_{j-1}$. Thus its FM kernel is given by $\Xx_{j}=\Delta_{*}(\V_{j}/\V_{j-1}) \in \Dd^b(G/B \times G/B)$ for all $1\leq j \leq N$,  where $\Delta:G/B \rightarrow G/B \times G/B$ is the diagonal map. Since a line bundle is invertible, we can define its inverse functor $\mathbb{X}_{j}^{-1}$ with FM kernel  given by $\Xx_{j}^{-1}=\Delta_{*}((\V_{j}/\V_{j-1})^{-1}) \in \Dd^b(G/B \times G/B)$.

With all the above setting we can prove the following theorem, which says that there is an categorical action of the $q=0$ affine Hecke algebra on $\Dd^b(G/B)$. More precisely, we lift the relations (\ref{H01}), (\ref{H02}),..., (\ref{H08}) in Definition \ref{definition 5} to the categorical level, which are (\ref{6.1}), (\ref{6.2}), ..., (\ref{6.8}), respectively in Theorem \ref{theorem 4}. The proof will be given in the next section.

\begin{theorem} \label{theorem 4} 
	There is a categorical action of the $q=0$ affine Hecke algebra $\Hh_{N}(0)$ on $\Dd^b(G/B)$. More precisely, if we define the FM kernels $\Tt_{i}=\Oo_{G/B \times_{G/P_{i}} G/B}$ and $\Xx_{j}=\Delta_{*}(\V_{j}/\V_{j-1})$ for $1 \leq i \leq N-1$, $1 \leq j \leq N$, then we have the following categorical relations
	\begin{equation}  \label{6.1}
	\Tt_{i} \ast \Tt_{i} \cong \Tt_{i},
	\end{equation}
	\begin{equation}  \label{6.2}
	\Tt_{i}\ast \Tt_{j} \cong \Tt_{j}\ast \Tt_{i} \ \text{if} \ |i-j| \geq 2,
	\end{equation}
	\begin{equation}  \label{6.3}
	\Tt_{i}\ast \Tt_{j} \ast \Tt_{i} \cong \Tt_{j} \ast \Tt_{i} \ast \Tt_{j} \ \text{if} \ |i-j|=1,
	\end{equation}
	\begin{equation}  \label{6.4}
	\Xx_{i} \ast \Xx_{i}^{-1} \cong \Xx_{i}^{-1} \ast \Xx_{i} \cong \Oo_{\Delta},
	\end{equation}
	\begin{equation} \label{6.5}
	\Xx_{i} \ast \Xx_{j} \cong \Xx_{j}\ast \Xx_{i} \ \text{for all } i,j,
	\end{equation}
	\begin{equation} \label{6.6}
	\Tt_{i}\ast \Xx_{j}\cong \Xx_{j} \ast \Tt_{i} \ \text{if} \ j \neq i,i+1,
	\end{equation}
	We have the following exact triangles in $\Dd^b(G/B \times G/B)$
	\begin{equation}  \label{6.7}
	\Tt_{i}\ast \Xx_{i} \rightarrow \Xx_{i+1}\ast \Tt_{i} \rightarrow \Xx_{i+1}, 
	\end{equation}
	\begin{equation} \label{6.8}
	 \Xx_{i} \ast \Tt_{i} \rightarrow  \Tt_{i} \ast \Xx_{i+1}  \rightarrow \Xx_{i+1}.
	\end{equation}   
\end{theorem}

\begin{remark}
Theorem \ref{theorem 4} essentially categorifies the classical results of Lusztig \cite{Lu1} and Kazhdan-Lusztig \cite{KL} on representations of the affine Hecke algebra in K-theory. 
\end{remark}

\subsection{Application by using shifted $q=0$ affine algebra}

In this section, instead of proving Theorem  \ref{theorem 4} by direct computation of convolution of kernels, we use the theory of categorical action of the shifted $q=0$ affine algebra which developed in Section \ref{section3} and Section \ref{section 5} to help us prove the theorem.

In order to use the results of categorical action, i.e. Theorem \ref{theorem 1} and Definition \ref{definition 2},  the main idea is to interpret the Demazure operators in terms of elements in the shifted $q=0$ affine algebra.

Recall that the Demazure operators are given by $T_{i}:=\pi^{*}_{i}\pi_{i*}:K(G/B) \rightarrow K(G/B)$ where $\pi_{i}:G/B \rightarrow G/P_{i}$ is the natural $\PP^1$-fibration with $1 \leq i \leq N-1$.  With the notation of $n$-step partial flag variety (\ref{eq fl}), the full flag variety in (\ref{eq ffl}) can be written as $G/B=Fl_{(1,1,...,1)}(\CC^N)$. For the partial flag varieties in (\ref{eq pfl}), we observe that 
\begin{align*}
V_{i-1} \overset{2}{\subset} V_{i+1} 
& = \textcolor{red}{V_{i-1}} \overset{0}{\subset} \textcolor{red}{V_{i-1}} \overset{2}{\subset}  V_{i+1} \\
& = V_{i-1} \overset{2}{\subset} \textcolor{blue}{V_{i+1}} \overset{0}{\subset}  \textcolor{blue}{V_{i+1}}
\end{align*} where the number above the inclusions are jump of dimensions. Thus they can be written as $G/P_{i}=Fl_{(1,1,..,1)+\alpha_{i}}(\CC^N)=Fl_{(1,1,..,1)-\alpha_{i}}(\CC^N)$ where $1 \leq i \leq N-1$.

We have the following commutator diagram
\begin{equation*}
     \xymatrix{ 
        & T_{i} \ar@(dl,dr) \\
    	K(G/P_{i}=Fl_{{(1,1,...,1)-\alpha_{i}}}(\CC^N))   \ar@/^/[r]^{e_{i,r}}   
    	& K(G/B=Fl_{{(1,1,...,1)}}(\CC^N))   \ar@/^/[r]^{e_{i,r}}   \ar@/^/[l]^{f_{i,s}}
    	& K(G/P_{i}=Fl_{{(1,1,...,1)+\alpha_{i}}}(\CC^N))   \ar@/^/[l]^{f_{i,s}}  
       }
\end{equation*} where $e_{i,r}1_{(1,1,...,1)}$, $e_{i,r}1_{(1,1,...,1)-\alpha_{i}}$, $f_{i,s}1_{(1,1,...,1)}$, and $f_{i,s}1_{(1,1,...,1)+\alpha_{i}}$ are elements in $\dot{\Uu}_{0,N}(L\SL_{N})$. Thus we can try to interpret $T_{i}$ in terms of elements in $\dot{\Uu}_{0,N}(L\SL_{N})$.

Since we define both the categorical actions of $\dot{\Uu}_{0,N}(L\SL_{N})$ and the Demazure operators $T_{i}$ by using the language of FM transforms/kernels, to relate the generators in $\dot{\Uu}_{0,N}(L\SL_{N})$ to $T_{i}$, we have to compare their kernels.

By definition, we have the kernel $\Ee_{i,0}1_{(1,1,..,1)}=\iota_{*}\Oo_{W^{1}_{i}((1,1,...,1))}$, where 
\begin{equation*}
W^{1}_{i}((1,1,...,1)):=\{(V_{\bullet},V_{\bullet}') \in  Fl_{(1,1,..,1)}(\CC^N) \times Fl_{(1,1,..,1)+\alpha_{i}}(\CC^N) \ | \ V_{j}=V'_{j} \ \Rm{for} \ j \neq i, \ \Rm{and} \ V'_{i} \subset V_{i}\}
\end{equation*} and $\iota:W^{1}_{i}((1,1,...,1)) \rightarrow Fl_{(1,1,..,1)}(\CC^N) \times Fl_{(1,1,..,1)+\alpha_{i}}(\CC^N)=G/B \times G/P_{i}$ is the natural inclusion. Note that by definition $V'_{i}=V_{i-1}$, so the condition $V'_{i} \subset V_{i}$ automatically satisfy. Then $W^{1}_{i}((1,1,...,1))=(id \times \pi_i)(G/B)$. Thus we have $\Ee_{i,0}\bo_{(1,1,..,1)}=\iota_{*}\Oo_{(id \times \pi_i)(G/B)}$, which is the kernel for the derived pushforward $\pi_{i*}$. Using the same argument, we can show that $\Ff_{i,0}\bo_{(1,1,...,1)}$ is also the same as the kernel for $\pi_{i*}$ and $\Ee_{i,0}\bo_{(1,1,..,1)-\alpha_{i}} \cong \Ff_{i,0}\bo_{(1,1,..,1)+\alpha_{i}}$ is the same as the kernel for $\pi_{i}^*$.

Since the Demazure operators are defined as  $T_{i}:=\pi_{i}^*\pi_{i*}$, by the above argument, the FM kernel $\Tt_{i}$ is isomorphic to 
\begin{equation} \label{eq cd}
\Tt_{i} \cong (\Ee_{i,0} \ast \Ff_{i,0})\bo_{(1,1,...,1)} \cong (\Ff_{i,0} \ast \Ee_{i,0})\bo_{(1,1,...,1)}
\end{equation} for all $1 \leq i \leq N-1$. Note that the second isomorphism can be seen from Proposition \ref{proposition 5}. Next, we make a simple modification to the isomorphisms in (\ref{eq cd}) so that it become more "natural" (which will be explained later). From the categorical relations (12)(a) in Definition \ref{definition 2}, we have 
\begin{equation} \label{eq cd1}
    ((\Psi^{+}_{i})^{-1} \ast \Ff_{i,0}) \bo_{(1,1,...,1)}[1] \cong (\Ff_{i,1} \ast (\Psi^{+}_{i})^{-1})\bo_{(1,1,...,1)}.
\end{equation} Observe that by Definition \ref{definition 3}, the kernel $\Psi^{+}_{i}\bo_{(1,1,...,1)-\alpha_{i}}$ is given by $\Oo_{\Delta}[1]$ where $\Delta$ is the diagonal in $G/P_{i} \times G/P_{i}$. Thus the left hand side in (\ref{eq cd1}) is isomorphic to $\Ff_{i,0}\bo_{(1,1,...,1)}$ and we obtain 
\begin{equation} \label{eq cd2}
    \Ff_{i,0} \bo_{(1,1,...,1)} \cong (\Ff_{i,1} \ast (\Psi^{+}_{i})^{-1})\bo_{(1,1,...,1)}.
\end{equation} Similarly, using the categorical relations (11)(a) and the same argument, we obtain  
\begin{equation} \label{eq cd3}
    \Ee_{i,0} \bo_{(1,1,...,1)} \cong (\Ee_{i,-1} \ast (\Psi^{-}_{i})^{-1})\bo_{(1,1,...,1)}.
\end{equation} With (\ref{eq cd2}) and (\ref{eq cd3}), (\ref{eq cd}) can be written as 
\begin{equation} \label{eq cd4}
\Tt_{i} \cong (\Ee_{i,0} \ast \Ff_{i,1} \ast (\Psi^{+}_{i})^{-1})\bo_{(1,1,...,1)} \cong (\Ff_{i,0} \ast \Ee_{i,-1} \ast (\Psi^{-}_{i})^{-1})\bo_{(1,1,...,1)}.
\end{equation}
    
We prefer to use (\ref{eq cd4}) over (\ref{eq cd}) for relating categorical action of Demazure operators to the categorical action of $\dot{\Uu}_{0,N}(L\SL_{N})$. So after passing to the K-theory (decategorifying), we obtain 
\begin{equation*}
    T_{i}=e_{i,0}f_{i,1}(\psi^{+}_{i})^{-1}1_{(1,1,...,1)}=f_{i,0}e_{i,-1}(\psi^{-}_{i})^{-1}1_{(1,1,...,1)}
\end{equation*} for all $1 \leq i \leq N-1$, which interprets the Demazure operators in terms of elements in $\dot{\Uu}_{0,N}(L\SL_{N})$.

Finally, for the kernel $\Xx_{i}$, it is easy to see that $\Xx_{1} \cong (\Psi^{-}_{1})^{-1}\bo_{(1,1,...,1)}$, $\Xx_{i} \cong \Psi^{+}_{i-1}\bo_{(1,1,...,1)} \cong (\Psi^{-}_{i})^{-1}\bo_{(1,1,...,1)}$ for all $2 \leq i \leq N-1$, and $\Xx_{N} \cong \Psi^{+}_{N-1}\bo_{(1,1,...,1)}$. This implies that $X_{1}=(\psi^{-}_{1})^{-1}1_{(1,1,...,1)}$, $X_{i}=\psi^{+}_{i-1}\bo_{(1,1,...,1)}= (\psi^{-}_{i})^{-1}\bo_{(1,1,...,1)}$ for all $2 \leq i \leq N-1$, and $X_{N}=\psi^{+}_{N-1}\bo_{(1,1,...,1)}$.

We are in a position to prove the Theorem \ref{theorem 4}.

%The first is the idempotent property when acting on the two end sides of weight spaces. 

%\begin{lemma} \label{lemma 17}  
%	\begin{align*}
%	&(\Ff_{i} \ast \Ee_{i} \ast \Ff_{i} \ast \Ee_{i}) \bo_{\kk+(k_{i}-1)\alpha_{i}} \cong (\Ff_{i} \ast \Ee_{i}) \bo_{\kk+(k_{i}-1)\alpha_{i}}, &(\Ee_{i} \ast \Ff_{i} \ast \Ee_{i} \ast \Ff_{i}) \bo_{\kk-(k_{i+1}-1)\alpha_{i}} \cong (\Ee_{i} \ast \Ff_{i}) \bo_{\kk-(k_{i+1}-1)\alpha_{i}}, \\
%	&( \Ee_{i} \ast \Ff_{i} \ast \Ee_{i} \ast \Ff_{i}) \bo_{\kk+k_{i}\alpha_{i}} \cong (\Ee_{i} \ast \Ff_{i}) \bo_{\kk+k_{i}\alpha_{i}}, \\
%	&(  \Ff_{i} \ast  \Ee_{i} \ast \Ff_{i} \ast \Ee_{i}) \bo_{\kk-k_{i+1}\alpha_{i}} \cong (\Ff_{i} \ast  \Ee_{i} ) \bo_{\kk-k_{i+1}\alpha_{i}}.
%	\end{align*}
%\end{lemma} 

%\begin{proof}
%We give a proof for the first relation, the rest are similar. Note that all the relations do not involve convolutions for different $i, j$, we can reduce to the $\SL_{2}$ case, i.e. we prove $(\Ff \ast \Ee  \ast \Ff  \ast \Ee ) \bo_{(1,N-1)} \cong (\Ff \ast \Ee ) \bo_{(1,N-1)}$.

%Note that since $\GG(0,N)$ is just a point, it is easy to calculate the kernel $(\Ff \ast \Ee ) \bo_{(1,N-1)}$, which is $\Oo_{\GG(1,N) \times \GG(1,N)}$. Thus 
%\begin{equation*}
%(\Ff \ast \Ee  \ast \Ff  \ast \Ee ) \bo_{(1,N-1)} \cong  \Oo_{\GG(1,N) \times \GG(1,N)} \ast \Oo_{\GG(1,N) \times \GG(1,N)} \cong \Oo_{\GG(1,N) \times \GG(1,N)} \cong (\Ff \ast \Ee ) \bo_{(1,N-1)}
%\end{equation*} which complete the proof.
%\end{proof}

The next relation is the Serre relation. This can be deduced from the relations in the definition of categorical action, i.e. Definition \ref{definition 2}. The proof is leave to the readers.

\begin{lemma}\label{lemma 18}  
	\begin{align*}
	&(\Ee_{i+1,0} \ast \Ee_{i,0} \ast \Ee_{i+1,0})\bo_{\kk} \cong (\Ee_{i+1,0} \ast \Ee_{i+1,0} \ast \Ee_{i,0})\bo_{\kk}, \\
	&(\Ee_{i,0} \ast \Ee_{i+1,0} \ast \Ee_{i,0})\bo_{\kk} \cong (\Ee_{i+1,0} \ast \Ee_{i,0} \ast \Ee_{i,0})\bo_{\kk}, \\
	&(\Ff_{i+1,0} \ast \Ff_{i,0} \ast \Ff_{i+1,0})\bo_{\kk} \cong (\Ff_{i,0} \ast \Ff_{i+1,0} \ast \Ff_{i+1,0})\bo_{\kk}, \\
	&(\Ff_{i,0} \ast \Ff_{i+1,0} \ast \Ff_{i,0})\bo_{\kk} \cong (\Ff_{i,0} \ast \Ff_{i,0} \ast \Ff_{i+1,0})\bo_{\kk}.
	\end{align*}
\end{lemma}

Now we prove Theorem \ref{theorem 4}.
\begin{proof}[Proof of Theorem \ref{theorem 4}]
Now we know how to write the generators $T_{i},\ X_{j}$ for the $q=0$ affine Hecke algebra $\Hh_{N}(0)$ in terms of generators in the shifted $q=0$ affine algebra $\dot{\Uu}_{0,N}(L\SL_{N})$. We apply Definition \ref{definition 2} of the categorical action to deduce the categorical action of $\Hh_{N}(0)$.

First, we prove relation (\ref{6.1}). From (\ref{eq cd4}) we have  $\Tt_{i} \cong (\Ee_{i,0} \ast \Ff_{i,1} \ast (\Psi^{+}_{i})^{-1})\bo_{(1,1,...,1)} \cong (\Ff_{i,0} \ast \Ee_{i,-1} \ast (\Psi^{-}_{i})^{-1})\bo_{(1,1,...,1)}$. We  choose one of the isomorphism (the argument for the other is similar), say $\Tt_{i} \cong (\Ee_{i,0} \ast \Ff_{i,1} \ast (\Psi^{+}_{i})^{-1})\bo_{(1,1,...,1)}$. Then by using relation (11)(a) in Definition \ref{definition 2}, we obtain
\begin{align*}
\Tt_{i} \ast \Tt_{i} & \cong  (\Ee_{i,0} \ast \Ff_{i,1} \ast (\Psi^{+}_{i})^{-1} \ast \Ee_{i,0} \ast \Ff_{i,1} \ast (\Psi^{+}_{i})^{-1}) \bo_{(1,1,...,1)} \\
& \cong (\Ee_{i,0} \ast \Ff_{i,1} \ast \Ee_{i,-1} \ast (\Psi^{+}_{i})^{-1} \ast \Ff_{i,1} \ast (\Psi^{+}_{i})^{-1}) \bo_{(1,1,...,1)}[1] \\
& \cong (\Ee_{i,0} \ast \Psi^{+}_{i}[-1] \ast (\Psi^{+}_{i})^{-1} \ast \Ff_{i,1} \ast (\Psi^{+}_{i})^{-1}) \bo_{(1,1,...,1)}[1] \\
& \cong (\Ee_{i,0} \ast \Ff_{i,1} \ast (\Psi^{+}_{i})^{-1}) \bo_{(1,1,...,1)} \cong \Tt_{i}
\end{align*} where in the third isomorphism we use the categorical commutator relations (16)(b), which is $(\Ff_{i,1} \ast \Ee_{i,-1}) \bo_{(1,1,...,1)-\alpha_{i}} \cong \Psi^{+}_{i}[-1]$. 

Next for relation (\ref{6.2}), we use again an isomorphism, say $\Tt_{i} \cong (\Ee_{i,0} \ast \Ff_{i,1} \ast (\Psi^{+}_{i})^{-1})\bo_{(1,1,...,1)}$. Then it follows from relations (6), (9)(c), (10)(c), (11)(c), (12)(c), and (15) in Definition \ref{definition 2} of the categorical action.
	
For relations (\ref{6.4}) and (\ref{6.5}), they follow from the definition and relation (6) in Definition \ref{definition 2}. 

For relation (\ref{6.6}), let $j \neq i, i+1$ and we choose $\Tt_{i} \cong (\Ee_{i,0} \ast \Ff_{i,1} \ast (\Psi^{+}_{i})^{-1})\bo_{(1,1,...,1)}$. For $i+2 \leq j \leq N$, we use $\Xx_{j}   \cong \Psi^{+}_{j-1}\bo_{(1,1,...,1)}$
and the result follows from relations (6), 11(b)(c) and 12(b)(c) in Definition \ref{definition 2}. Similarly, for $1 \leq j \leq i-1$, we use $\Xx_{j} \cong (\Psi^{-}_{j})^{-1}\bo_{(1,1,...,1)}$ and the result also follows from relations (6), 11(b)(c) and 12(b)(c) in Definition \ref{definition 2}.
	
For relation (\ref{6.7}) and (\ref{6.8}),  by relations (16)(b)(c) (or Proposition \ref{proposition 5}), we have the following exact triangles
\begin{align} 
&(\Ff_{i,1} \ast \Ee_{i,0})\bo_{(1,1,...,1)} \rightarrow (\Ee_{i,0} \ast \Ff_{i,1})\bo_{(1,1,....,1)} \rightarrow \Psi_{i}^{+}\bo_{(1,1,...,1)}  \label{ses 17}, \\
&(\Ee_{i,-1} \ast \Ff_{i,0})\bo_{(1,1,....,1)} \rightarrow (\Ff_{i,0} \ast \Ee_{i,-1})\bo_{(1,1,...,1)} \rightarrow \Psi_{i}^{-}\bo_{(1,1,...,1)} \label{ses 18}.
\end{align}

If we choose $\Tt_{i} \cong (\Ee_{i,0} \ast \Ff_{i,1} \ast (\Psi^{+}_{i})^{-1})\bo_{(1,1,...,1)}$ and $\Xx_{i+1}   \cong \Psi^{+}_{i}\bo_{(1,1,...,1)}$, then we have $(\Ee_{i,0} \ast \Ff_{i,1})\bo_{(1,1,....,1)} \cong \Tt_{i} \ast \Xx_{i+1}$. On the other hand, using (\ref{eq cd}) we have $\Tt_{i} \cong (\Ff_{i,0} \ast \Ee_{i,0})\bo_{(1,1,...,1)}$. Choosing $\Xx_{i} \cong (\Psi^{-}_{i})^{-1}\bo_{(1,1,...,1)}$, we have 
\begin{align*}
\Xx_{i} \ast \Tt_{i} & \cong  ((\Psi^{-}_{i})^{-1} \ast \Ff_{i,0} \ast \Ee_{i,0})\bo_{(1,1,...,1)} \cong ( \Ff_{i,1} \ast (\Psi^{-}_{i})^{-1} \ast \Ee_{i,0})\bo_{(1,1,...,1)}[1] \\
& \cong (\Ff_{i,1}\ast \Ee_{i,0})\bo_{(1,1,...,1)}
\end{align*} where we use relation (12)(a) in the second isomorphism and the third isomorphism comes from the similar fact that the kernel $\Psi^{-}_{i}\bo_{(1,1,...,1)+\alpha_{i}}$ is given by $\Oo_{\Delta}[1]$. With these facts, the exact triangle (\ref{ses 17}) is precisely the following 
\begin{equation*}
\Xx_{i} \ast \Tt_{i} \rightarrow  \Tt_{i} \ast \Xx_{i+1} \rightarrow \Xx_{i+1}
\end{equation*} which is the relation (\ref{6.8}). A similar argument with the exact triangle (\ref{ses 18}) can get  relation (\ref{6.7}).
	
Finally, we prove the braid relation (\ref{6.3}). We prove the case for $j=i+1$, the case for $j=i-1$ is similar. Using (\ref{eq cd}), we have $\Tt_{i} \cong (\Ee_{i,0} \ast \Ff_{i,0})\bo_{(1,1,...,1)} \cong (\Ff_{i,0} \ast \Ee_{i,0})\bo_{(1,1,...,1)}$. Then we have
\begin{align}
\begin{split} \label{eq cb}
\Tt_{i}\ast \Tt_{i+1} \ast \Tt_{i} & \cong (\Ee_{i,0} \ast \Ff_{i,0} \ast \Ee_{i+1,0} \ast \Ff_{i+1,0} \ast \Ee_{i,0} \ast \Ff_{i,0}) \bo_{(1,1,...,1)} \\
& \cong (\Ee_{i,0} \ast  \Ee_{i+1,0} \ast \Ff_{i,0} \ast \Ee_{i,0} \ast  \Ff_{i+1,0} \ast  \Ff_{i,0}) \bo_{(1,1,...,1)} \ (\text{by} \ \text{relation (15)}) \\
& \cong (\Ee_{i,0} \ast  \Ee_{i+1,0} \ast \Ee_{i,0} \ast \Ff_{i,0} \ast  \Ff_{i+1,0} \ast  \Ff_{i,0}) \bo_{(1,1,...,1)} \ (\text{by} \ \text{relation (16)(e)}) \\
& \cong (\Ee_{i+1,0} \ast \Ee_{i,0} \ast \Ee_{i,0} \ast \Ff_{i,0} \ast \Ff_{i,0} \ast \Ff_{i+1,0})\bo_{(1,1,...,1)} \  (\text{by} \ \text{Lemma} \ \ref{lemma 18}) \\
& \cong (\Ee_{i+1,0} \ast \Ee_{i,0} \ast \Ff_{i,0} \ast \Ee_{i,0} \ast \Ff_{i,0} \ast \Ff_{i+1,0})\bo_{(1,1,...,1)} \ (\text{by} \ \text{relation (16)(e)}) \\
\end{split}
\end{align}

Next, we show that $(\Ee_{i,0} \ast \Ff_{i,0} \ast \Ee_{i,0} \ast \Ff_{i,0})\bo_{(1,1,...,1)-\alpha_{i+1}} \cong (\Ee_{i,0} \ast \Ff_{i,0})\bo_{(1,1,...,1)-\alpha_{i+1}}$. First, we observe that $(\Ee_{i,0} \ast \Ff_{i,0})\bo_{(1,1,...,1)-\alpha_{i+1}} \cong (\Ff_{i,0} \ast \Ee_{i,0})\bo_{(1,1,...,1)-\alpha_{i+1}}$. Then by relation (11)(a), we get 
\begin{align*}
(\Ff_{i,0} \ast \Ee_{i,0})\bo_{(1,1,...,1)-\alpha_{i+1}} &\cong (\Ff_{i,0} \ast \Psi_{i}^{-} \ast  \Ee_{i,-1} \ast (\Psi_{i}^{-})^{-1} )\bo_{(1,1,...,1)-\alpha_{i+1}}[-1] \\
& \cong (\Ff_{i,0} \ast \Ee_{i,-1} \ast (\Psi_{i}^{-})^{-1} )\bo_{(1,1,...,1)-\alpha_{i+1}}
\end{align*} where the second isomorphism comes from the fact that $\Psi_{i}^{-}\bo_{(1,1,...,1)-\alpha_{i+1}+\alpha_{i}}$ is given by $\Oo_{\Delta}[1]$. Thus 
\begin{align*}
&(\Ee_{i,0} \ast \Ff_{i,0} \ast \Ee_{i,0} \ast \Ff_{i,0})\bo_{(1,1,...,1)-\alpha_{i+1}}  \\
&\cong (\Ff_{i,0} \ast \Ee_{i,-1} \ast (\Psi_{i}^{-})^{-1} \ast \Ff_{i,0} \ast \Ee_{i,-1} \ast (\Psi_{i}^{-})^{-1} )\bo_{(1,1,...,1)-\alpha_{i+1}} \\
& \cong  (\Ff_{i,0} \ast \Ee_{i,-1}  \ast \Ff_{i,1} \ast (\Psi_{i}^{-})^{-1} \ast \Ee_{i,-1} \ast (\Psi_{i}^{-})^{-1} )\bo_{(1,1,...,1)-\alpha_{i+1}}[1] \ (\text{by} \ \text{relation (12)(a)}) \\
& \cong  (\Ff_{i,0} \ast \Psi_{i}^{-} \ast (\Psi_{i}^{-})^{-1} \ast \Ee_{i,-1} \ast (\Psi_{i}^{-})^{-1} )\bo_{(1,1,...,1)-\alpha_{i+1}}[1][-1] \ (\text{by} \ \text{relation (16)(c)}) \\
& \cong  (\Ff_{i,0} \ast \Ee_{i,-1} \ast (\Psi_{i}^{-})^{-1} )\bo_{(1,1,...,1)-\alpha_{i+1}} \cong (\Ff_{i,0} \ast \Ee_{i,0})\bo_{(1,1,...,1)-\alpha_{i+1}} \cong (\Ee_{i,0} \ast \Ff_{i,0})\bo_{(1,1,...,1)-\alpha_{i+1}}.
\end{align*}

So (\ref{eq cb}) becomes 
\begin{align}
\begin{split} \label{eq cb1}
   (\Ee_{i+1,0} \ast \Ee_{i,0} \ast \Ff_{i,0} \ast \Ff_{i+1,0})\bo_{(1,1,...,1)} & \cong (\Ee_{i+1,0} \ast \Ff_{i,0} \ast \Ee_{i,0} \ast \Ff_{i+1,0})\bo_{(1,1,...,1)} \\
   & \cong (\Ff_{i,0} \ast \Ee_{i+1,0} \ast \Ff_{i+1,0} \ast \Ee_{i,0})\bo_{(1,1,...,1)}.
   \end{split}
\end{align} 

Using the same argument, we can show that $(\Ee_{i+1,0} \ast \Ff_{i+1,0})\bo_{(1,1,...,1)+\alpha_{i}} \cong (\Ee_{i+1,0} \ast \Ff_{i+1,0} \ast \Ee_{i+1,0} \ast \Ff_{i+1,0})\bo_{(1,1,...,1)+\alpha_{i}}$. Hence (\ref{eq cb1}) becomes 
\begin{align*}
&(\Ff_{i,0} \ast \Ee_{i+1,0} \ast \Ff_{i+1,0} \ast \Ee_{i+1,0} \ast \Ff_{i+1,0} \ast \Ee_{i,0})\bo_{(1,1,...,1)} \\
&\cong (\Ff_{i,0} \ast  \Ff_{i+1,0} \ast \Ee_{i+1,0} \ast  \Ff_{i+1,0} \ast \Ee_{i+1,0} \ast \Ee_{i,0})\bo_{(1,1,...,1)} \ (\text{by} \ \text{relation (16)(e)})  \\
&\cong (\Ff_{i,0} \ast  \Ff_{i+1,0}  \ast  \Ff_{i+1,0} \ast \Ee_{i+1,0} \ast \Ee_{i+1,0} \ast \Ee_{i,0})\bo_{(1,1,...,1)} \ (\text{by} \ \text{relation (16)(e)}) \\
& \cong (\Ff_{i+1,0} \ast  \Ff_{i,0}  \ast  \Ff_{i+1,0} \ast \Ee_{i+1,0} \ast \Ee_{i,0} \ast \Ee_{i+1,0})\bo_{(1,1,...,1)} \ (\text{by} \ \text{Lemma} \ \ref{lemma 18}) \\
& \cong (\Ff_{i+1,0} \ast  \Ff_{i,0}  \ast  \Ee_{i+1,0} \ast \Ff_{i+1,0} \ast \Ee_{i,0} \ast \Ee_{i+1,0})\bo_{(1,1,...,1)} \ (\text{by} \ \text{relation (16)(e)})  \\
& \cong (\Ff_{i+1,0} \ast  \Ee_{i+1,0}  \ast  \Ff_{i,0} \ast \Ee_{i,0} \ast \Ff_{i+1,0} \ast \Ee_{i+1,0})\bo_{(1,1,...,1)} \ (\text{by} \ \text{relation (15)})  \\
& \cong \Tt_{i+1}\ast \Tt_{i} \ast \Tt_{i+1} \ (\text{by} \ (\ref{eq cd})) 
\end{align*}

\end{proof}

We give a few remarks to conclude this section.

\begin{remark}
We explain why we prefer (\ref{eq cd4}) over (\ref{eq cd}) (and similarly at the K-theory level) for writing (or interpret) the Demazure operators in terms of elements in the shifted $q=0$ affine algebra. The main reason comes from the naturality of  adjoint functors, i.e. relation (5).

By relation (5), we obtain that 
\begin{align*}
\Hom({\Ee}_{i,0}\bo_{(1,1,...,1)-\alpha_{i}}, {\Ee}_{i,0}\bo_{(1,1,...,1)-\alpha_{i}}) &\cong \Hom(\bo_{(1,1,...,1)}{\Ee}_{i,0} \ast ({\Ee}_{i,0}\bo_{(1,1,...,1)-\alpha_{i}})_{R}, \bo_{(1,1,...,1)})  \\
& \cong \Hom(({\Ee}_{i,0}\ast {\Psi}^{+}_{i}\ast {\Ff}_{i,2} \ast ({\Psi}^{+}_{i})^{-2})\bo_{(1,1,...,1)}[-1], \bo_{(1,1,...,1)}) \\
& \cong \Hom(({\Ee}_{i,0} \ast{\Ff}_{i,1} \ast({\Psi}^{+}_{i})^{-1})\bo_{(1,1,...,1)}, \bo_{(1,1,...,1)}).
\end{align*}

Thus $({\Ee}_{i,0} \ast{\Ff}_{i,1} \ast({\Psi}^{+}_{i})^{-1})\bo_{(1,1,...,1)}$ naturally appears as the kernel for the composition of adjoint functors ${\E}_{i,0} \circ ({\E}_{i,0}\bo_{(1,1,...,1)-\alpha_{i}})^{R}\bo_{(1,1,...,1)}$. Similarly, we also have $({\Ff}_{i,0} \ast{\Ee}_{i,-1} \ast({\Psi}^{-}_{i})^{-1})\bo_{(1,1,...,1)}$ naturally appears as the kernel for the composition of adjoint functors ${\F}_{i,0} \circ ({\F}_{i,0}\bo_{(1,1,...,1)+\alpha_{i}})^{R}\bo_{(1,1,...,1)}$.

Moreover, we expect that the morphisms in the exact triangles (\ref{ses 17}) and (\ref{ses 18}) should be determined come from the identity morphisms $Id:{\Ee}_{i,0}\bo_{(1,1,...,1)-\alpha_{i}} \rightarrow {\Ee}_{i,0}\bo_{(1,1,...,1)-\alpha_{i}}$ and $Id:{\Ff}_{i,0}\bo_{(1,1,...,1)+\alpha_{i}} \rightarrow {\Ff}_{i,0}\bo_{(1,1,...,1)+\alpha_{i}}$, respectively. 
\end{remark}

\begin{remark}
In \cite{Hsu2}, we provide two generalizations of the construction of Demazure operators from the full flag variety $G/B$ to the $n$-step partial flag varieties $Fl_{\kk}(\CC^N)$ by constructing two families of idempotent operators acting on $\Dd^b(Fl_{\kk}(\CC^N))$ and $K(Fl_{\kk}(\CC^N))$. Moreover, these two families of operators generate two variants of the $q=0$ Hecke algebras.

There we give another proof of Theorem \ref{theorem 4} using the main result (i.e. Theorem 5.2) in \textit{loc. cit.}, and the proof is even shorter than the proof (especially for the braid relation) in this article.
\end{remark}

\begin{remark}
We try to relate categorical actions of $q=0$ affine Hecke algebras to the notion of Demazure descent data on a triangulated category that defined in \cite{AK1}, \cite{AK2}. 

From  (\ref{6.2}), (\ref{6.3}) in Theorem \ref{theorem 4}, we know that those FM kernels $\{\Tt_{i}\}_{1 \leq i \leq N-1}$ gives a weak braid monoid action on $\Dd^b(G/B)$. Moreover, (\ref{6.1}) in Theorem \ref{theorem 4} implies that there is a co-projector structure on $\Tt_{i}$ for all $i$. Thus $\{\Tt_{i}\}_{1 \leq i \leq N-1}$ gives a Demazure descent data on $\Dd^b(G/B)$.
\end{remark}

\appendix 

\section{A conjectural presentation} \label{appendix A}

In this appendix, we define another algebra which can be viewed as a second definition of the shifted $q=0$ affine algebra. It has the so-called loop presentation that defined by using generating series. Thus, in this definition we have infinite generators and relations.

Then we construct an action of this algebra on the Grothendieck groups of $n$-step partial flag varieties $\bigoplus_{\kk}K(Fl_{\kk}(\CC^N))$.

Finally, we conjecture that the loop presentation we defined for the shifted $q=0$ affine algebra is equivalent to the presentation that defined in Definition \ref{definition 1}.

\begin{definition} \label{Definition 1}
	Defining the associative $\CC$-algebra $\dot{\Uu}'_{0,N}(L\SL_n)$ that is generated by 
	\begin{equation*}
	\bigcup_{\kk \in C(n,N)} \{1_{\kk}, \ e_{i,r}1_{\kk}, \ f_{i,r}1_{\kk}, \ \psi^{\pm}_{i,\pm s^{\pm}_{i}}1_{\kk}, \ (\psi^{+}_{i,k_{i+1}})^{-1}1_{\kk}, \ (\psi^{-}_{i,-k_{i}})^{-1}1_{\kk} \}_{1 \leq i \leq n-1}^{r \in \ZZ, \ s^{+}_{i} \geq k_{i+1}, \ s^{-}_{i} \geq k_{i}}
	\end{equation*}
	with the following relations (for all $1\leq i,j \leq n-1$, $\kk \in C(n,N)$, $\epsilon,\epsilon' \in \{\pm\}$ and $s^{+}_{i} \geq k_{i+1}$, $s^{-}_{i} \geq k_{i}$).
	\begin{equation}
	1_{\kk}1_{\Ll}=\delta_{\kk,\Ll}1_{\kk}, \ e_{i,r}1_{\kk}=1_{\kk+\alpha_{i}}e_{i,r}, \ f_{i,r}1_{\kk}=1_{\kk-\alpha_{i}}f_{i,r}, \ \psi^{+}_{i, s^{+}_{i}}1_{\kk}=1_{\kk}\psi^{+}_{i, s^{+}_{i}}, \ \psi^{+}_{i, -s^{-}_{i}}1_{\kk}=1_{\kk}\psi^{+}_{i, -s^{-}_{i}}.
	\end{equation}
	
	\begin{equation} 
	[\psi^{\epsilon}_{i,\kk}(z), \psi^{\epsilon'}_{j,\kk}(w)]1_{\kk}=0, \
	(\psi^{+}_{i,k_{i+1}})^{\pm 1} \cdot (\psi^{+}_{i,k_{i+1}})^{\mp 1} 1_{\kk} = 1_{\kk}=(\psi^{-}_{i,-k_{i}})^{\pm 1} \cdot (\psi^{-}_{i,-k_{i}})^{\mp 1}1_{\kk}.
	\end{equation}
	
	\begin{equation} 
	\begin{split}
	&ze_{i,\kk+\alpha_{i}}(z)e_{i,\kk}(w)1_{\kk}=-we_{i,\kk+\alpha_{i}}(w)e_{i,\kk}(z)1_{\kk}, \\
	&we_{i,\kk+\alpha_{i+1}}(z)e_{i+1,\kk}(w)1_{\kk}=(w-z)e_{i+1,\kk+\alpha_{i}}(w)e_{i,\kk}(z)1_{\kk}, \\
	&(z-w)e_{i,\kk+\alpha_{j}}(z)e_{j,\kk}(w)1_{\kk}=(z-w)e_{j,\kk+\alpha_{i}}(w)e_{i,\kk}(z)1_{\kk}, \ \text{if} \ |i-j| \geq 2.
	\end{split}
	\end{equation}
	
	\begin{equation}
	\begin{split}
	&-wf_{i,\kk-\alpha_{i}}(z)f_{i,\kk}(w)1_{\kk}=zf_{i,\kk-\alpha_{i}}(w)f_{i,\kk}(z)1_{\kk}, \\
	&(w-z)f_{i,\kk-\alpha_{i+1}}(z)f_{i+1,\kk}(w)1_{\kk}=wf_{i+1,\kk-\alpha_{i}}(w)f_{i,\kk}(z)1_{\kk}, \\
	&(z-w)f_{i,\kk-\alpha_{j}}(z)f_{j,\kk}(w)1_{\kk}=(z-w)f_{j,\kk-\alpha_{i}}(w)f_{i,\kk}(z)1_{\kk}, \ \text{if} \ |i-j| \geq 2.
	\end{split}
	\end{equation}
	
	\begin{equation}
	\begin{split}
	&z\psi^{+}_{i,\kk+\alpha_{i}}(z)e_{i,\kk}(w)1_{\kk}=-we_{i,\kk}(w)\psi^{+}_{i,\kk}(z)1_{\kk},\\
	&\frac{-w}{z}(\sum_{s \geq 0}(\frac{w}{z})^s)\psi^{+}_{i,\kk+\alpha_{i+1}}(z)e_{i+1,\kk}(w)1_{\kk}=e_{i+1,\kk}(w)\psi^{+}_{i,\kk}(z)1_{\kk},\\
	&\psi^{+}_{i,\kk+\alpha_{i-1}}(z)e_{i-1,\kk}(w)1_{\kk}=(\sum_{s \geq 0}(\frac{w}{z})^s)e_{i-1,\kk}(w)\psi^{+}_{i,\kk}(z)1_{\kk}, \\
	&\psi^{+}_{i,\kk+\alpha_{j}}(z)e_{j,\kk}(w)1_{\kk}=e_{j,\kk}(w)\psi^{+}_{i,\kk}(z)1_{\kk}, \ \text{if} \ |i-j| \geq 2.
	\end{split}
	\end{equation}
	
	\begin{equation}
	\begin{split}
	&z\psi^{-}_{i,\kk+\alpha_{i}}(z)e_{i,\kk}(w)1_{\kk}=-we_{i,\kk}(w)\psi^{-}_{i,\kk}(z)1_{\kk}, \\
	&(\sum_{s \geq 0}(\frac{z}{w})^s)\psi^{-}_{i,\kk+\alpha_{i+1}}(z)e_{i+1,\kk}(w)1_{\kk}=e_{i+1,\kk}(w)\psi^{-}_{i,\kk}(z)1_{\kk}, \\
	&\psi^{-}_{i,\kk+\alpha_{i-1}}(z)e_{i-1,\kk}(w)1_{\kk}=\frac{-z}{w}(\sum_{s \geq 0}(\frac{z}{w})^s)e_{i-1,\kk}(w)\psi^{-}_{i,\kk}(z)1_{\kk},\\
	&\psi^{-}_{i,\kk+\alpha_{j}}(z)e_{j,\kk}(w)1_{\kk}=e_{j,\kk}(w)\psi^{-}_{i,\kk}(z)1_{\kk}, \ \text{if} \ |i-j| \geq 2.
	\end{split}
	\end{equation}
	
	\begin{equation}
	\begin{split}
	&-w\psi^{+}_{i,\kk-\alpha_{i}}(z)f_{i,\kk}(w)1_{\kk}=zf_{i,\kk}(w)\psi^{+}_{i,\kk}(z)1_{\kk}, \\
	&\psi^{+}_{i,\kk-\alpha_{i+1}}(z)f_{i+1,\kk}(w)1_{\kk}=\frac{-w}{z}(\sum_{s \geq 0}(\frac{w}{z})^s)f_{i+1,\kk}(w)\psi^{+}_{i,\kk}(z)1_{\kk}, \\
	&(\sum_{s \geq 0}(\frac{w}{z})^s)\psi^{+}_{i,\kk-\alpha_{i-1}}(z)f_{i-1,\kk}(w)1_{\kk}=f_{i-1,\kk}(w)\psi^{+}_{i,\kk}(z)1_{\kk}, \\
	&\psi^{+}_{i,\kk-\alpha_{j}}(z)f_{j,\kk}(w)1_{\kk}=f_{j,\kk}(w)\psi^{+}_{i,\kk}(z)1_{\kk}, \ \text{if} \ |i-j| \geq 2.
	\end{split}
	\end{equation}
	
	\begin{equation}
	\begin{split}
	&-w\psi^{-}_{i,\kk-\alpha_{i}}(z)f_{i,\kk}(w)1_{\kk}=zf_{i,\kk}(w)\psi^{-}_{i,\kk}(z)1_{\kk}, \\
	&\psi^{-}_{i,\kk-\alpha_{i+1}}(z)f_{i+1,\kk}(w)1_{\kk}=(\sum_{s \geq 0}(\frac{z}{w})^s)f_{i+1,\kk}(w)\psi^{-}_{i,\kk}(z)1_{\kk}, \\
	&\frac{-z}{w}(\sum_{s \geq 0}(\frac{z}{w})^s)\psi^{-}_{i,\kk-\alpha_{i-1}}(z)f_{i-1,\kk}(w)1_{\kk}=f_{i-1,\kk}(w)\psi^{-}_{i,\kk}(z)1_{\kk}, \\
	&\psi^{-}_{i,\kk-\alpha_{j}}(z)f_{j,\kk}(w)1_{\kk}=f_{j,\kk}(w)\psi^{-}_{i,\kk}(z)1_{\kk}, \ \text{if} \ |i-j| \geq 2.
	\end{split}
	\end{equation}
	
	\begin{equation}
	\begin{split}
	e_{i,\kk-\alpha_{i}}(z)f_{i,\kk}(w)1_{\kk}-f_{i,\kk+\alpha_{i}}(w)e_{i,\kk}(z)1_{\kk}=\delta_{ij}\delta(\frac{z}{w})(\psi^{+}_{i,\kk}(z)-\psi^{-}_{i,\kk}(z))1_{\kk}.
	\end{split}
	\end{equation}
	where the generating series are defined as follows
	\[
	e_{i,\kk}(z):=\sum_{r \in \ZZ} e_{i,r}1_{\kk}z^{-r},\ f_{i,\kk}(z):=\sum_{r \in \ZZ} f_{i,r}1_{\kk}z^{-r},
	\]
	\[
	\psi^{+}_{i,\kk}(z):=\sum_{r \geq k_{i+1}} \psi^{+}_{i,r}1_{\kk}z^{-r}, \ \psi^{-}_{i,\kk}(z):=\sum_{r \geq k_{i}} \psi^{-}_{i,-r}1_{\kk}z^{r}, \
	\delta(z):=\sum_{r\in \ZZ} z^r.
	\]
\end{definition}

Although the notion of categorical action of $\dot{\Uu}'_{0,N}(L\SL_n)$ is not easy to define, we can construct an action on the Grothendieck groups of $n$-step partial flag varieties $Fl_{\kk}(\CC^N)$. 

To construct such action, we have to define the action of those generators in $\dot{\Uu}'_{0,N}(L\SL_n)$. First, we define $e_{i,r}1_{\kk}$ and $f_{i,s}1_{\kk}$ for all $r,s \in \ZZ$ via using the K-theoretic FM transforms, i.e.,
\begin{equation*}
e_{i,r}1_{\kk}:K(Fl_{\kk}(\CC^N)) \rightarrow K(Fl_{\kk+\alpha_{i}}(\CC^N)), \ x \mapsto \pi_{2*}(\pi^{*}_{1}(x) \otimes [\iota(\kk)_{*}(\V_{i}/\V'_{i})^{r}])
\end{equation*} for all $r \in \ZZ$, and similarly for $f_{i,s}1_{\kk}$ in the opposite direction with $s \in \ZZ$. Next, for $\psi^{\pm}_{i, \pm s_{i}^{\pm}}1_{\kk}$ we define their action on $K(Fl_{\kk}(\CC^N))$ as the follows
\begin{align}
\begin{split}\label{eq psi1} 
\psi^{+}_{i, s_{i}^{+}}1_{\kk}:&K(Fl_{\kk}(\CC^N)) \rightarrow K(Fl_{\kk}(\CC^N)),  \\
& x \mapsto x \otimes (-1)^{k_{i+1}-1}[\det(\V_{i+1}/\V_{i})][\Sym^{s^{+}_{i}-k_{i+1}}(\V_{i+1}/\V_{i-1})] 
\end{split}
\end{align} where $s^{+}_{i} \geq k_{i+1}$ and similarly 
\begin{align}
\begin{split} \label{eq psi2}
\psi^{-}_{i, -s_{i}^{-}}1_{\kk}:&K(Fl_{\kk}(\CC^N)) \rightarrow K(Fl_{\kk}(\CC^N)),  \\
& x \mapsto x \otimes (-1)^{k_{i}}[\det(\V_{i}/\V_{i-1})^{-1}][\Sym^{s^{-}_{i}-k_{i}}(\V_{i+1}/\V_{i-1})^{\vee}]
\end{split}
\end{align} where $s^{-}_{i} \geq k_{i}$.

Then note that all the results (except Proposition \ref{proposition 5}) we prove above in order to prove Theorem \ref{theorem 1} do not assume any conditions for $r,\ s$ for $e_{i,r}1_{kk}$ and $f_{i,s}1_{\kk}$. Thus passing the works in the main text to the Grothendieck group and using Corollary \ref{corollary cr2}, plus a little extra check, it is easy to verify the following theorem.

\begin{theorem}
There is an action of $\dot{\Uu}'_{0,N}(L\SL_n)$ on $\bigoplus_{\kk} K(Fl_{\kk}(\CC^N))$.
\end{theorem}

Next, we expect that the two presentations in Definition \ref{definition 1} and Definition \ref{Definition 1} are equivalent like the two presentations of shifted quantum affine algebras in \cite{FT} (Theorem 5.5 in loc. cit.).

To relate this definition to Definition \ref{definition 1}, we introduce another set of Cartan generators $\{h_{i,\pm r}\}^{r \geq 1}_{1 \leq i \leq N-1}$  with  
the relations to $\{\psi^{\pm}_{i,\pm s^{\pm}_{i}}1_{\kk}\}$ via the following 
\begin{align*}
&(\psi^{+}_{i,k_{i+1}}z^{-k_{i+1}})^{-1} \psi^{+}_{i,\kk}(z)=(1+h_{i,+}(z))1_{\kk},\\
&(\psi^{-}_{i,-k_{i}}z^{k_{i}})^{-1} \psi^{-}_{i,\kk}(z)=(1+h_{i,-}(z))1_{\kk},
\end{align*}
where $h_{i,\pm}(z)=\sum_{r>0} h_{i,\pm r}z^{\mp r}$.

\begin{remark}
From (\ref{eq psi1}) and (\ref{eq psi2}), it is easy to see that the action of $h_{i,\pm r}$ on the Grothendieck group $K(Fl_{\kk}(\CC^N))$ is given by 
\begin{align*}
&h_{i,r}1_{\kk}:K(Fl_{\kk}(\CC^N)) \rightarrow K(Fl_{\kk}(\CC^N)), \ x \mapsto x \otimes [\Sym^{r}(\V_{i+1}/\V_{i-1})]  \\
&h_{i,-r}1_{\kk}:K(Fl_{\kk}(\CC^N)) \rightarrow K(Fl_{\kk}(\CC^N)), \ x \mapsto x \otimes [\Sym^{r}(\V_{i+1}/\V_{i-1})^{\vee}]
\end{align*}
\end{remark} where $r \geq 1$.

Then we define inductively
\begin{align*}
&e_{i,r}1_{\kk}:= \begin{cases} 
-\psi^{+}_{i}e_{i,r-1}(\psi^{+}_{i})^{-1}1_{\kk}& \text{if} \  r>0 \\
-(\psi^{+}_{i})^{-1}e_{i,r+1}\psi^{+}_{i}1_{\kk}& \text{if} \  r<-k_{i}-1,\\
\end{cases} \\
&f_{i,r}1_{\kk}:= \begin{cases} 
-(\psi^{+}_{i})^{-1}f_{i,r-1}\psi^{+}_{i}1_{\kk}& \text{if} \  r>k_{i+1}+1 \\
-\psi^{+}_{i}f_{i,r+1}(\psi^{+}_{i})^{-1}1_{\kk}& \text{if} \  r<0, \\
\end{cases}\\
&\psi^{+}_{i,r}1_{\kk}:=
[e_{i,r-k_{i+1}-1},f_{i,k_{i+1}+1}]1_{\kk} \  \text{for} \  r \geq k_{i+1}+1, \\
&\psi^{-}_{i,r}1_{\kk}:=
-[e_{i,r},f_{i,0}]1_{\kk} \  \text{for} \  r \leq -k_{i}-1.
\end{align*}

We purpose the following conjecture, which roughly speaking, says that the two presentations defined by Definition \ref{Definition 1} and Definition \ref{definition 1} are equivalent.

\begin{conjecture} \label{conjecture 1}
	There is a $\CC$-algebra isomorphism $\dot{\Uu}_{0,N}(L\SL_n) \rightarrow \dot{\Uu}'_{0,N}(L\SL_n)$ such that 
	\[
	e_{i,r}1_{\kk} \mapsto e_{i,r}1_{\kk}, \ f_{i,r}1_{\kk} \mapsto f_{i,r}1_{\kk}, \ \psi^{+}_{i}1_{\kk} \mapsto \psi^{+}_{i,k_{i+1}}1_{\kk},\ \psi^{-}_{i}1_{\kk} \mapsto \psi^{-}_{i,-k_{i}}1_{\kk},
	\]  for $1 \leq i \leq n-1$.
\end{conjecture}

\bibliographystyle{abberv}
\bibliography{bibfile}

\end{document}